\documentclass[11pt]{amsart}
\usepackage[T1]{fontenc}

\usepackage{amsmath,amssymb,amscd,amsthm,amsxtra}
\usepackage{color}
\usepackage{fancybox}
\usepackage{graphicx, mathrsfs}
\usepackage{mathrsfs,dsfont}
\usepackage{fancyhdr}
\usepackage{ulem}
\usepackage{enumerate}
\usepackage{color}
\usepackage{cases}
\usepackage{hyperref}

\newtheorem{thm}{Theorem}[section]
\newtheorem{prop}[thm]{Proposition}
\newtheorem{lem}[thm]{Lemma}
\newtheorem{cor}[thm]{Corollary}

\theoremstyle{definition}
\newtheorem{definition}[thm]{Definition}

\theoremstyle{remark}

\newtheorem{remark}[thm]{Remark}
\numberwithin{equation}{section}

\newcommand{\case}[2]{ \noindent \textbf{Case #1:} #2 \\}

\newcommand{\ZZZ}{\mathbb{Z}}
\newcommand{\RRR}{\mathbb{R}}

\newcommand{\TTT}{\mathbb{T}}

\newcommand{\dd}{\delta}
\newcommand{\DD}{\Delta}
\newcommand{\w}{\omega}
\newcommand{\ee}{\varepsilon}

\begin{document}

\title{Global well-posedness of the Energy-critical\\ nonlinear
Schr\"{o}dinger
equation on $\mathbb{T}^4$}

\author{HAITIAN YUE}
\address{Department of Mathematics and Statistics, University of Massachusetts
 Amherst}
\email{hyue@math.umass.edu}


\begin{abstract}
In this paper, we first prove global well-posedness for the  defocusing cubic  nonlinear Schr\"odinger equation (NLS) on 4-dimensional tori - either rational or irrational - and with initial data in $H^1$. Furthermore, we prove that if a maximal-lifespan solution of the focusing cubic NLS $u: I\times\TTT^4\to \mathbb{C}$ satisfies $\sup_{t\in I}\|u(t)\|_{\dot{H}^1(\TTT^4)}<\|W\|_{\dot{H}^1(\RRR^4)}$, then it is a global solution. $W$ denotes the ground state on Euclidean space, which is a stationary solution of the corresponding focusing equation in $\RRR^4$.
\end{abstract}

\maketitle

\section{Introduction}

In this paper, we consider the cubic nonlinear Schr\"odinger equation (NLS)  in the periodic setting $x\in \TTT_\lambda^4$
\begin{equation}\label{eq:NLS}
  (i\partial_t + \Delta) u = \mu u|u|^2,
\end{equation}
where $\mu = \pm 1$ ($+1$: the defocusing
 case,
$-1$: the focusing case).
And $u : \RRR\times \TTT_\lambda^4 \to \mathbb{C}$ is a complex-valued function of time space $\RRR$ and spatial space $\TTT^4_\lambda$, a general rectangler tori,
i.e.
\[
\TTT^4_\lambda := \RRR^4/(\prod_{i=1}^4 \lambda_i \ZZZ),\qquad \lambda = (\lambda_1, \lambda_2, \lambda_3, \lambda_4),
\]
where  $\lambda_i\in (0,\infty)$ { for } $i=1, 2 ,3 ,4.$
Specifically, if the ratio of arbitrary two $\lambda_i'\, s$ in $\{\lambda_1, \lambda_2, \lambda_3, \lambda_4\}$ is an irrational number, then $\TTT^4_\lambda$ is called an irrational torus, otherwise $\TTT^4_\lambda$ is called a rational torus. Since our proof doesn't change no matter either rational or irrational tori. For the convenience, we use $\TTT^4: = \TTT^4_\lambda$ hence-forth in the paper.

Solutions of (\ref{eq:NLS}) conserve in both the mass of $u$:
\begin{equation}\label{eq:mass}
  M(u)(t):= \int_{\TTT^4} |u(t)|^2\, dx
\end{equation}
and the energy of $u$:
\begin{equation}\label{eq:energy}
  E(u)(t):= \frac{1}{2} \int_{\TTT^4} |\nabla u(t)|^2\, dx + \frac{1}{4}\mu
  \int_{\TTT^4} |{u(t)}|^4\, dx.
\end{equation}

\subsection{The defocusing case ($\mu = +1$)}
In the defocusing case, our main theorem is global well-posedness  of (\ref{eq:NLS}) with
$H^1(\TTT^4)$ initial data.

\begin{thm}[GWP of the defocusing NLS]\label{thm:main}
If $u_0\in H^1(\TTT^4)$, for any $T\in [0, \infty)$,  there exists a unique global solution $u\in X^1([-T, T])$
of the initial value problem
\begin{equation}\label{eq:IVP}
  (i\partial_t + \Delta) u = u|u|^2,\qquad u(0) = u_0.
\end{equation}
In addition, the mapping $u_0 \to u$ extends to a continuous mapping from
$H^1(\TTT^4)$ to $X^1([-T, T])$ and $M(u)$ and $E(u)$
defined in (\ref{eq:mass}) and (\ref{eq:energy}) are conserved along the flow.
\end{thm}

The space $X^1(I) \subset C(I: H^1(\TTT^4))$  is the adapted atomic space (see Definition \ref{def:Xs}).

On $\RRR^d$,  the scaling symmetry plays an important role in the well-posedness (existence, uniqueness and continuous dependence of the data to solution map) problem of initial value problem (IVP) for NLS:
\begin{equation}\label{eq:pNLS}
\begin{cases}
i\partial_t u + \Delta u = |{u}|^{p-1} u,\qquad p>1\\
u(0, x) = u_0(x) \in \dot{H}^s(\RRR^d).
\end{cases}
\end{equation}
The IVP (\ref{eq:pNLS}) is scale invariant in the Sobolev norm $\dot{H}^{s_c}$, where $s_c := \frac{d}{2} - \frac{2}{p-1}$ is called the scaling critical regularity. 

For $H^s$ data with $s > s_c$ (sub-critical regime), the local-well posedness (LWP) of the IVP (\ref{eq:pNLS}) in sub-critical regime was proven by Cazenave-Weissler \cite{caz1}. 
For $H^s$ data with $s = s_c$ (critical regime), Bourgain \cite{bourgain1999radial} first proved the large data global well-posedness (GWP) and scattering for the defocusing energy-critical ($s_c = 1$) NLS in $\RRR^3$  with the radially symmetric initial data in $\dot{H}^1$ by introducing an induction method on the size of energy and a refined Morawetz inequality. A different proof of the same result was given by Grillakis in \cite{grillakis2000radial}. Then a breakthrough was made by Colliander-Keel-Staffilani-Takaoka-Tao \cite{colliander2008global}. Their work extended the results of Bourgain \cite{bourgain1999radial} and Grillakis \cite{grillakis2000radial}. They proved global well-posedness and scattering of the energy-critical problem in $\RRR^3$ for general large data in $\dot{H}^1$.
Then similar results were proven by Ryckman-Vi{\c{s}}an \cite{visan2007global} and Vi{\c{s}}an 
\cite{vicsan2011global} on the higher dimension $\RRR^d$ spaces. Furthermore, Dodson proved mass-critical ($s_c = 0$) global wellposedness results for $\RRR^d$ in his series of papers \cite{dodson2012global, dodson2016global1, dodson2016global2}.

For the corresponding problems on the tori, the Strichartz estimates on rational tori $\TTT^d$ (see \cite{GV,Yaj,KTao} for the Strichartz estimates in the Euclidean spaces $\RRR^d$), which prove the local well-posedness of the periodic NLS, was initially developed by Bourgain \cite{bourgain1993fourier}. In \cite{bourgain1993fourier}, the number theoretical related lattice counting arguments were used, hence this method worked better in the rational tori than irration tori.  Recently Bourgain-Demeter's work \cite{bourgain2015proof} proved the optimal Strichartz estimates on both rational and irrational tori via a totally different approach which doesn't depend on the lattice counting lattice. Also there are other important references \cite{BourgainStrichartz2007, DeSilvaNatasaStaffilani2007, CatoireWang2010, BourgainStrichartz2013, GuoOhWang2014, KillipVisanStrichartz2016, Demirbas2017, DengGermainGuth2017, FanBilinear2018} on the Strichartz estimates on the tori and global existence of solution of the Cauchy problem in sub-critical regime.
On the general compact manifolds, Burq-Gerard-Tzvetkov derived the Strichartz type estimates and applied these estimates to the global well-posedness of NLS on compact manifolds in a series of papers \cite{BurqGerardTzvetkov2004, BurqGerardTzvetkov2005Bilinear, BurqGerardTzvetkov2005multilinear, BurqGerardTzvetkov2007}. We also refers to \cite{Zhong2008, GerardPierfelice2010, ZaherBilinear2012, Zaher2012} and references therein for the other results of global existence sub-critical NLS on compact manifolds.

In the critical regime, Herr-Tataru-Tzvetkov \cite{herr2011global} studied the global existence of the energy-critical NLS on $\TTT^3$ and first proved the global well-posedness with small initial data in $H^1$. They used a crucial trilinear Strichartz type estimates in the context of the critical atomic spaces $U^p$ and $V^p$, which were originally developed in unpublished work on wave maps by Tararu. These atomic spaces were systematically formalized by Hadac-Herr-Koch \cite{hadac2009well} (see also \cite{KochTataru2005}\cite{herr2014strichartz}) and now the atomic spaces $U^p$ and $V^p$ are widely used in the field of the critical well-posedness theory of nonlinear dispersive equations.
The large data global well-posedness result of the energy-critical NLS on rational $\TTT^3$ was proven by Ionescu-Pausader \cite{ionescu2012energy}, which is the first large data critical global well-posedness result of NLS on a compact manifold.  
In a series of papers, Ionescu-Pausader \cite{ionescu2012energy}\cite{ionescu2012global2} and Ionescu-Pausader-Staffilani \cite{ionescu2012global1} developed a method to obtain energy-critical large data global well-posedness in more general manifolds ($\TTT^3$, $\TTT^3\times\RRR$, and $\mathbb{H}^3$) based on the corresponding results on the Euclidean spaces in the same dimension. So far, their method 
has been successfully applied to other manifolds in several following papers \cite{pausader2014global, strunk2015global, zhao2017gT2R2, zhao2017TR3}. In particular, based on the recent developments on the large data global well-posedness theory in the product spaces $\RRR^m\times \TTT^n\ (m\geq 1, n\geq 1)$, many authors (\cite{TzvetkovVisciglia2012, HaniPausader2014, tzvetkov2014well, modifiedScattering2015, Grebert2016, cheng2017scattering, zhao2017gT2R2, zhao2017TR3, Liu2018}) studied the long time behaviors (scattering, modified scattering and etc) of the solutions of the NLS. 

In this paper, we prove the large data global wellposedness result of defocusing energy-critical NLS on the both rational and irrational tori in the 4-dimension. Our proof is closely related to the strategy developed by Ionescu-Pausader \cite{ionescu2012energy}\cite{ionescu2012global2}. Compared to the $\TTT^3\times \RRR$ case, in Ionescu-Pausader \cite{ionescu2012global2}, there is less dispersion of the Schr\"{o}dinger operator in the compact manifold $\TTT^4$. This means that it is more difficult to obtain a sharp enough Strichartz type estimates as Proposition 2.1 in \cite{ionescu2012global2}. 
We use the sharp Strichartz type estimates (Lemma 
\ref{prop:strichartz}) recently proven by Bourgain-Demeter 
\cite{bourgain2015proof} in our proof.
Moreover, since Lemma \ref{prop:strichartz} works both for rational or irrational tori, we can prove the result on both rational and irrational tori.

The main parts in the proof of {Theorem \ref{thm:main}} will follow the concentration-compactness framework of Kenig-Merle 
\cite{KenigMerle}, which is a deep and broad road map to deal with critical problems (see also in \cite{KenigMerle2008NLW}\cite{KenigMerle2010}).
Our first step is to obtain the critical local well-posedness theory 
and the stability theory of (\ref{eq:NLS}) in $\TTT^4$. For that 
purpose, we follow Herr-Tataru-Tzvetkov's idea \cite{herr2011global}\cite{herr2014strichartz} and introduce the adapted critical spaces $X^s$ and $Y^s$, which are frequency localized modifications of atomic spaces $U^p$ and $V^p$, as our solution spaces and nonlinear spaces.  
Applying Proposition \ref{prop:strichartz} and the strip decomposition technique in \cite{herr2011global, herr2014strichartz} to the atomic spaces in time-space frequency space, we 
obtain a crucial bilinear estimate and then the local well-posedness of (\ref{eq:NLS}). 
Then we measure the solution in a weaker critical space-time norm $Z
$, which plays a similar role as $L^{10}_{x,t}$ norm in 
\cite{colliander2008global}. On the one hand, equipped with $Z$-norm, we 
obtain the refined bilinear estimate (Lemma \ref{lem:bilinear}) and 
hence it is proven that the solution stay regular as long as $Z$-norm stay finite (i.e. global well-posedness with a priori $Z$-norm bound). 
On the other hand, we show that concentration of a large amount of the Z -norm in fnite time is self-defeating. The reason is as follows.
Concentration of a large amount the Z-norm in finite
time can only happen around a space-time point, which can be 
considered as a Euclidean-like solution. 
To implement this, arguing by contradiction, we construct a 
sequence of initial data which implies a sequence of solutions and 
leads the $Z$-norm towards infinity. Then following the profile decompositon idea (firstly by Gerard \cite{gerard1998} in Sobolev embedding and Merle-Vega \cite{merle1998} in the Schr\"odinger equation), we perform a linear profile decomposition 
of the sequence of initial data with one Scale-1-profile and a 
series of Euclidean profiles that concentrate at space-time points. 
We get nonlinear profiles by running the linear profiles along the 
nonlinear Schr\"{o}dinger flow as initial data. By the contradiction 
condition, the scattering properties of nonlinear Euclidean profiles 
and the defect of interaction between different profiles show that 
there is actually at most one profile which is the Euclidean 
profile. And the corresponding nonlinear Euclidean profile is just 
the Euclidean-like solution we want.
Euclidean-like solution can be interpreted in some sense as 
solutions in the Euclidean space $\RRR^4$, however, these kind of 
concentration as a Euclidean-like solution is prevented by the global well-posedness 
theory on the Euclidean space $\RRR^4$ in Vi{\c{s}}an-Ryckman
\cite{visan2007global} and Vi{\c{s}}an \cite{vicsan2011global}'s 
papers.

\subsection{The focusing case ($\mu = -1$)}

In the focusing case, we prove global well-posedness when both the modified energy and  kinect energy of the initial data less than energy and modified kinect energy of the ground state $W$ in $\RRR^4$. Moreover,
\begin{equation}\label{eqn:groundstate1}
W(x) = W(x, t) = \frac{1}{ 1 + \frac{|x|^2}{8}}  \qquad \text{in } \dot{H}^1(\RRR^4)
\end{equation}
which is a stationary solution of the focusing case of (\ref{eq:NLS}) and also solves the elliptic equation in $\RRR^4$
\begin{equation}\label{eqn:groundstate2}
\Delta W +|W|^{2}W = 0.
\end{equation}
Then we define a constant $C_4$ by using the stationary solution $W$. And also $C_4$ is the best constant in Sobolev embedding (see Remark \ref{rmk:bestconstant}).
\begin{equation}\label{eq:WC}
\|W\|^2_{\dot{H}^1(\RRR^4)} = \|W\|^4_{L^4(\RRR^4)} := \frac{1}{C_4^4} \qquad\text{ and then } \qquad
E_{\RRR^4}(W) = \frac{1}{4C^4_4},
\end{equation}
where $E_{\RRR^4}(W)$ is the energy of $W$ in the Euclidean space $\RRR^4$:
\begin{equation}\label{def:energyinEuclidean}
E_{\RRR^4}(W):= \frac{1}{2} \int_{\RRR^4} |\nabla W(x)|^2\, dx - \frac{1}{4}
  \int_{\RRR^4} |{W(x)}|^4\, dx.
\end{equation}

\begin{thm}[GWP of the focusing NLS]\label{thm:focusing}
Assume $u_0\in H^1(\TTT^4)$. Assume that $u$ is a maximal-lifespan solution $u: I\times\TTT^4\to \mathbb{C}$ satisfying
\begin{equation}\label{cond:focusing1}
\sup_{t\in I}\|u(t)\|_{\dot{H}^1(\TTT^4)}<\|W\|_{\dot{H}^1(\RRR^4)},
\end{equation} 
then for any $T\in [0, +\infty)$, $u\in X^1([-T, T])$ is a solution
of the initial value problem
\begin{equation}\label{eq:IVP2}
  (i\partial_t + \Delta) u = -u|u|^2,\qquad u(0) = u_0.
\end{equation}
\end{thm}

For the technical reason in the focusing case, we should introduce two modified energies of $u$:
\begin{equation}\label{def:modifiedEnergy1}
E_*(u)(t) := \frac{1}{2}(\|u(t)\|^2_{\dot{H}^1(\TTT^4)}+c_* \|u(t)\|^2_{L^2(\TTT^4)}) - \frac{1}{4}\|u(t)\|_{L^4(\TTT^4)}^4,
\end{equation}
and
\begin{equation}\label{def:modifiedEnergy2}
E_{**}(u)(t) := \frac{1}{2}(\|u(t)\|^2_{\dot{H}^1(\TTT^4)}+c_* \|u(t)\|^2_{L^2(\TTT^4)})- \frac{1}{4}\|u(t)\|_{L^4(\TTT^4)}^4  + \frac{c_*^2 C^4_4}{4}\|u(t)\|^4_{L^2(\TTT^4)},
\end{equation}
where $c_*$ is a fixed constant determined by the Sobolev embedding on $\TTT^4$ (Lemma \ref{lem:sobolev}). By the definitions (\ref{def:modifiedEnergy1})(\ref{def:modifiedEnergy2}), $E_*(u)(t)$ and $E_{**}(u)(t)$ are conserved in time.

We also introduce $\|u\|_{H_*^1(\TTT^4)}$ as a modified inhomogenous Sobolev norm:
\begin{equation}\label{def:modifiedNorm}
\|u\|^2_{H_*^1(\TTT^4)} = \|u\|^2_{\dot{H}^1(\TTT^4)} + c_* \|u\|^2_{L^2(\TTT^4)}
\end{equation}
 Obviously, $H^1_*(\TTT^4)$-norm and $H^1(\TTT^4)$-norm are two comparable norms ($\|u\|_{H^1_*(\TTT^4)} \simeq \|u\|_{H^1(\TTT^4)}$).

\begin{cor}\label{cor:focusing}
Assume that $u_0\in H^1(\TTT^4)$ satisfying
\begin{equation}\label{cond:focusing2}
\|u_0\|_{H^1_*(\TTT^4)}<\|W\|_{\dot{H}^1(\RRR^4)},\quad E_*(u_0)<E_{\RRR^4}(W);
\end{equation}
OR
\begin{equation}\label{cond:focusing3}
\|u_0\|_{\dot{H}^1(\TTT^4)}<\|W\|_{\dot{H}^1(\RRR^4)},\quad E_{**}(u_0)<E_{\RRR^4}(W),
\end{equation}
where $E_*(u_0)$ and $E_{**}(u_0)$ are two modified Energies defined in (\ref{def:modifiedEnergy1}) and (\ref{def:modifiedEnergy2}), and $E_{\RRR^4}(W)$ is the Energy in the Euclidean space defined in (\ref{def:energyinEuclidean}).
Then for any $T\in [0, \infty)$,  there exists a unique global solution $u\in X^1([-T, T])$
of the initial value problem (\ref{eq:IVP2}).
In addition, the mapping $u_0 \to u$ extends to a continuous mapping from
$H^1(\TTT^4)$ to $X^1([-T, T])$ for any $T\in [0, \infty)$.
\end{cor}

\begin{remark}
By the the energy trapping lemma (Theorem 2.5) in the Section 2, either 
(\ref{cond:focusing2}) or (\ref{cond:focusing3}) implies the condition (\ref{cond:focusing1}) in Theorem \ref{thm:focusing}. 
\end{remark}

In the focusing case, global well-posedness result usually doesn't 
hold for arbitary data. For the energy-critical focusing NLS on $\RRR^d$, Kenig-Merle \cite{KenigMerle} first proved the global 
well-posedness and scattering with initial data below a ground state 
threshold ($E_{\RRR^d}(u_0)<E_{\RRR^d}(W)$ and $\|u_0\|_{\dot{H}^1}< \|W\|_{\dot{H}^1}$) in the 
radial case ($d\geq 3$). And then the corresponding results without 
the radial conditions were proven by Killp-Vi{\c{s}}an \cite{KillipVisanFocusing} ($d\geq 5$) and Dodson \cite{DodsonFocusing} ($d=4$). We also refer to \cite{DuyckaertsFocusing2008, HolmerRoudenko2008, Cazenave2011, DuyckaertsRoudenko2015, Masaki2015, DodsonMurphy2017, dodson2017new} for other focusing NLS results.

In this paper, we also prove a similar global well-posedness result for the energy-critical focusing NLS on $\TTT^4$ below the ground state threshold. As in the defocusing case, we follow the idea in Ionescu-Pausader \cite{ionescu2012energy}\cite{ionescu2012global2} and use the 
focusing global well-posedness result \cite{DodsonFocusing} in $\RRR^4$ as a black box. 
It is known that the conditions in $\RRR^d$ are $E_{\RRR^d}(u_0)<E_{\RRR^d}(W)$ 
and $\|u_0\|_{\dot{H}^1}< \|W\|_{\dot{H}^1}$, which are tightly related to the Sobolev embedding with the best constant in $\RRR^d$. However the 
sharp version of Sobolev embedding (Lemma \ref{lem:sobolev}) is 
quite different. So compared to the conditions for initial data in Eulidean space $\RRR^d$, the conditions in Corollary \ref{cor:focusing} should be different. A similar case is the focusing 
NLS on the hyperbolic space, it doesn't share the same sharp version 
of Sobolev embedding either. On the hyperbolic space, Fan-Kleinhenz
\cite{FanVariational} give a minimal ratio between energy and $L^2$ 
norm under the condition $E_{\mathbb{H}^3}(u_0)<E_{\RRR^3}(W)$ and $\|u_0\|_{\dot{H}^1}< \|W\|_{\dot{H}^1}$ in the radial case, and also in Banica-Duyckaerts \cite{BanicaDuyckaerts}'s paper about the focusing NLS on the 
hyperbolic space, they modified the Sobolev norm and energy by 
subtracting  a multiple of the $L^2$ norm.
 On $\TTT^d$, based on the best constants of  Sobolev embedding (Lemma \ref{lem:sobolev}) on $\TTT^d$, we should also modify the energy and Sobolev norm by adding terms related to the $L^2$ norm , so that the modified 
 conditions together with Sobolev embedding derive the energy 
 trapping property which controls the Sobolev norm globally in time. 
 In Section 2, we will discuss the Sobolev embedding and energy 
 trapping lemma in detail.

\subsection{Outline of the following paper}
The rest of the paper is organized as follows. In Section 2, we prove the energy trapping property for the focusing NLS. In Section 3, we introduce the adapted atomic spaces $X^s$, $Y^s$ and $Z$ norm and provide some corresponding embedding properties of the spaces. In Section 4, we use Herr-Tataru-Tzvetkov's method and Bourgain-Demeter's sharp Strichartz estimate to develop a large-data local well-posedness and stability theory for (\ref{eq:NLS}).
In Section 5, {we study} the behavior of Euclidean-like solutions to the linear and nonlinear equation concentrating to a point in space and time. In Section 6, we recall a similar profile decomposition as the Section 5 in \cite{ionescu2012energy} to measure the defects of compactness in the Strchartz inequality. In Section 7, we prove the main theorems (Theorem \ref{thm:main} and Theorem \ref{thm:focusing}) except for a lemma. In Section 8, we prove the remaining lemma about approximate solutions.

\subsection*{Acknowledgements}
The author is greatly indebted to his advisor, Andrea R. Nahmod, for suggesting this problem and her patient guidance and warm encouragement over the past years. The author would like to thank Beno\^{i}t Pausader for prompting the author to study the focusing case. The author also would like to thank Chenjie Fan for his helpful discussions on the focusing case of this paper.  The author acknowledges support from the National Science Foundation
through his advisor Andrea R. Nahmod's grants NSF-DMS 1201443 and NSF-DMS 1463714.

\section{Energy trapping for the focusing NLS}
Before proceeding to the proofs of main theorems (Theorem \ref{thm:main} and Theorem \ref{thm:focusing}), we explain how Theorem \ref{thm:focusing} implies Corollary \ref{cor:focusing} in the focusing case by using the energy trapping argument. In this section, we'll prove the energy trapping argument in $\TTT^4$ which is different from the energy trapping argument (Theorem 3.9 in \cite{KenigMerle}) in $\RRR^4$.

\begin{lem}[Sobolev embedding with best constants by \cite{Aubin2}\cite{Hebey2}\cite{hebey1996}]\label{lem:sobolev}
Let $f\in H^1(\TTT^4)$, then there exists a positive constant $c_*$, such that
\begin{equation}\label{eq:Sobolev}
\|f\|^2_{L^4(\TTT^4)} \leq C^2_4 (\|f\|^2_{\dot{H}^1(\TTT^4)} + c_* \|f\|^2_{L^2(\TTT^4)}).
\end{equation}
where $C_4$ is the best constant of this inequality.
\end{lem}

\begin{remark}\label{rmk:bestconstant}
$C_4$ is the same constant as expressed in (\ref{eq:WC}), because $C_4$ is also the best constant of the Sobolev embedding in $\RRR^4$,
$
\|f\|^2_{L^4(\RRR^4)} \leq C^2_4 \|f\|^2_{\dot{H}^1(\RRR^4)},
$
and the function $W(x)$ holds the Sobolev embedding with the best constant $C_4$ in $\RRR^4$.
\end{remark}

\begin{remark}
Since $\|u\|^2_{H_*^1(\TTT^4)} = \|u\|^2_{\dot{H}^1(\TTT^4)} + c_* \|u\|^2_{L^2(\TTT^4)}$, the Sobolev embedding (Lemma \ref{lem:sobolev}) can be also written in the form:
\[
\|f\|^2_{L^4(\TTT^4)} \leq C^2_4 \|f\|^2_{{H}_*^1(\TTT^4)}.
\]
Suppose $c_{opt}:= \inf\{c_* : c_* \text{ holds } (\ref{eq:Sobolev})\}$.
By taking $f=1$, it's easy to check that $c_{opt} \geq \frac{1}{C_4^2 Vol(\TTT^4)^{1/2}}$. 
\end{remark}  

\begin{lem}\label{lem:EnergyTrapping}
(i) Suppose $f\in H^1(\TTT^4)$ and $\delta_0>0$ satisfying 
\begin{equation}\label{eq:conditionOfEnergytrapping}
\|f\|_{H^1_*(\TTT^4)}< \|W\|_{\dot{H}^1(\RRR^4)} \qquad\text{ and }\qquad E_*(f)<(1-\delta_0)E_{\RRR^4}(W),
\end{equation}
then there exists $\bar{\delta} = \bar{\delta}(\delta_0)>0$ such that
\begin{align}\label{eq:energyresult1}
\|f\|^2_{H^1_*(\TTT^4)}< (1-\bar{\delta}) \|W\|^2_{\dot{H}^1(\RRR^4)}\\\label{eq:energyresult2}
\|f\|^2_{H^1_*(\TTT^4)} - \|f\|^4_{L^4(\TTT^4)} \geq \bar{\delta}\|f\|^2_{H^1_*(\TTT^4)},
\end{align} 

and in particular
\begin{equation}\label{eq:energyresult3}
E_*(f) \geq \frac{1}{4}(1+\bar{\delta})\|f\|^2_{H^1_*(\TTT^4)}.
\end{equation}

(ii)  Suppose $f\in H^1(\TTT^4)$ and $\delta_0>0$ satisfying 
\begin{equation}\label{eq:conditionOfEnergytrapping2}
\|f\|_{\dot{H}^1(\TTT^4)}< \|W\|_{\dot{H}^1(\RRR^4)} \qquad\text{ and }\qquad E_{**}(f)<(1-\delta_0)E_{\RRR^4}(W),
\end{equation}
then there exists $\bar{\delta} = \bar{\delta}(\delta_0)>0$ such that
\begin{align}\label{eq:energyresult12}
\|f\|^2_{\dot{H}^1(\TTT^4)}< (1-\bar{\delta}) \|W\|^2_{\dot{H}^1(\RRR^4)}\\\label{eq:energyresult22}
\|f\|^2_{\dot{H}^1(\TTT^4)} - \|f\|^4_{L^4(\TTT^4)} + 2c_* \|f\|_{L^2(\TTT^4)}+c_*^2 C^4_4 \|f\|_{L^2(\TTT^4)}^4\geq \bar{\delta}\|f\|^2_{\dot{H}^1(\TTT^4)},
\end{align} 

and in particular
\begin{equation}\label{eq:energyresult32}
E_{**}(f) \geq \frac{1}{4}(1+\bar{\delta})\|f\|^2_{\dot{H}^1(\TTT^4)}.
\end{equation}
\end{lem}

\begin{proof}
In the proof of part (i), we almost identically follow the proof of Lemma 3.4 in Kenig-Merle's paper \cite{KenigMerle}, but use ${H^1_*(\TTT^4)}$-norm instead of ${\dot{H}^1(\TTT^4)}$-norm. Consider a quadratic function $g_1 = \frac{1}{2}y - \frac{C^4_4}{4}y^2$, and plug in $\|f\|^2_{H^1_*(\TTT^4)}$, by Sobolev embedding (Lemma \ref{lem:sobolev}) and the assumption (\ref{eq:conditionOfEnergytrapping}), we have that
\begin{equation}\label{eq:energytrapping1}
\begin{split}
g_1(\|f\|^2_{H^1_*}) & = \frac{1}{2} \|f\|^2_{H^1_*} - \frac{C^4_4}{4} \|f\|^4_{H^1_*}\\
&\leq \frac{1}{2} \|f\|^2_{H^1_*} - \frac{1}{4}\|f\|_{L^4}^4 = E_*(f)\\
&< (1-\delta_0)E_{\RRR^4}(W)= (1-\delta_0)g_1(\|W\|^2_{\dot{H}^1(\RRR^4)}).
\end{split}
\end{equation}

It is easy to know $\|f\|^2_{H^1_*(\TTT^4)}<(1-\bar{\delta})\|W\|^2_{\dot{H}^1(\RRR^4)}$, from (\ref{eq:energytrapping1}) and the property of quadratic function $g_1$, where $\bar{\delta}\sim \delta_0^{\frac{1}{2}}$.

Then choose $g_2(y) = y - C^4_4 y^2$, if plug in $\|f\|^2_{H^1_*(\TTT^4)}$, by Sobolev embedding (Lemma \ref{lem:sobolev}), we have that 
\begin{equation}\label{eq:energytrapping2}
g_2(\|f\|^2_{H^1_*(\TTT^4)}) = \|f\|^2_{H^1_*(\TTT^4)}- C^4_4 \|f\|^4_{H^1_*(\TTT^4)} \leq \|f\|^2_{H^1_*(\TTT^4)} - \|f\|^4_{L^4(\TTT^4)}.
\end{equation}

Since $g_2(0) = 0$, $g''_2(y) = - 2C^4_4<0$ and 
$ \|f\|^2_{H^1_*(\TTT^4)} < (1-\bar{\delta}) \|W\|^2_{\dot{H}^1(\RRR^4)}$, by Jensen's inequality and (\ref{eq:WC}),
\begin{equation}\label{eq:energytrapping3}
g_2(\|f\|^2_{H^1_*(\TTT^4)})> g_2((1-\bar{\delta})\|W\|^2_{\dot{H}^1(\RRR^4)})\frac{\|f\|^2_{H^1_*(\TTT^4)}}{(1-\bar{\delta})\|W\|^2_{\dot{H}^1(\RRR^4)}}=\bar{\delta}\|f\|^2_{H^1_*(\TTT^4)}. 
\end{equation}

Together (\ref{eq:energytrapping2}) and (\ref{eq:energytrapping3}), we get (\ref{eq:energyresult2}).

By (\ref{eq:energyresult2}), we get (\ref{eq:energyresult3})
\[
E_*(f) = \frac{1}{4}\|f\|^2_{H^1_*(\TTT^4)} + \frac{1}{4}(\|f\|^2_{H^1_*(\TTT^4)} -\|f\|^4_{L^4(\TTT^4)})\geq \frac{1}{4}(1+\bar{\delta})\|f\|^2_{H^1_*(\TTT^4)}.
\]

The proof of part (ii) would be similar with part (i). Under the assumptions (\ref{eq:conditionOfEnergytrapping2}) of part (ii), by squaring Sobolev embedding (Lemma \ref{lem:sobolev}), we have that
\begin{equation}\label{eq:squareSobolev}
C^4_4 \|f\|^4_{\dot{H}^1(\TTT^4)}\geq \|f\|^4_{L^4} -2c_* \|f\|^2_{L^2(\TTT^4)}-c_*^2C^4_4\|f\|^4_{L^2(\TTT^4)}
\end{equation}

Plugging $\|f\|^2_{\dot{H}^1(\TTT^4)}$ into $g_1$, by (\ref{eq:squareSobolev}), we hold that
\begin{equation}\label{eq:energytrapping12}
\begin{split}
g_1(\|f\|^2_{\dot{H}^1}) & = \frac{1}{2} \|f\|^2_{\dot{H}^1} - \frac{C^4_4}{4} \|f\|^4_{\dot{H}^1}\\
&\leq \frac{1}{2} \|f\|^2_{\dot{H}^1} - \frac{1}{4}\|f\|_{L^4}^4 + \frac{c_*}{2}\|f\|^2_{L^2(\TTT^4)}+\frac{c_*^2 C^4_4}{4}\|f\|^4_{L^2(\TTT^4)}= E_{**}(f)\\
&< (1-\delta_0)E_{\RRR^4}(W)= (1-\delta_0)g_1(\|W\|^2_{\dot{H}^1(\RRR^4)}).
\end{split}
\end{equation}
It is easy to know $\|f\|^2_{\dot{H}^1(\TTT^4)}<(1-\bar{\delta})\|W\|^2_{\dot{H}^1(\RRR^4)}$, from (\ref{eq:energytrapping12}) and the property of quadratic function $g_1$, where $\bar{\delta}\sim \delta_0^{\frac{1}{2}}$. Similarly, we can also hold (\ref{eq:energyresult22})(\ref{eq:energyresult32}) under the assumption (\ref{eq:conditionOfEnergytrapping2}).
\end{proof}

\begin{thm}[Energy trapping]\label{thm:EnergyTrapping}
(i)
Let $u$ be a solution of IVP (\ref{eq:IVP2}), such that for $\delta_0>0$
\begin{equation}\label{ass:i}
\|u_0\|_{H^1_*(\TTT^4)}<\|W\|_{\dot{H}^1(\RRR^4)},\quad E_*(u_0)<(1-\delta_0)E_{\RRR^4}(W);
\end{equation}

Let $I\ni 0$ be the maximal interval of existence, then there exists $\bar{\delta} = \bar{\delta}(\delta_0)>0$ such that for all $t\in I$
\begin{align}
\|u(t)\|^2_{H^1_*(\TTT^4)}< (1-\bar{\delta}) \|W\|_{\dot{H}^1(\RRR^4)},\\
\|u(t)\|^2_{H^1_*(\TTT^4)} - \|u(t)\|^4_{L^4(\TTT^4)} \geq \bar{\delta}\|u(t)\|^2_{H^1_*(\TTT^4)},
\end{align} 
and in particular
\begin{equation}\label{eq:energyresult4}
E_*(u)(t) \geq \frac{1}{4}(1+\bar{\delta})\|u(t)\|^2_{H^1_*(\TTT^4)}.
\end{equation}

(ii)
Let $u$ be a solution of IVP (\ref{eq:IVP2}), such that for $\delta_0>0$
\begin{equation}\label{ass:ii}
\|u_0\|_{\dot{H}^1(\TTT^4)}<\|W\|_{\dot{H}^1(\RRR^4)},\quad E_{**}(u_0)<(1-\delta_0)E_{\RRR^4}(W);
\end{equation}

Let $I\ni 0$ be the maximal interval of existence, then there exists $\bar{\delta} = \bar{\delta}(\delta_0)>0$ such that for all $t\in I$
\begin{align}
\|u(t)\|^2_{\dot{H}^1(\TTT^4)}< (1-\bar{\delta}) \|W\|_{\dot{H}^1(\RRR^4)},\\
\|u(t)\|^2_{\dot{H}^1(\TTT^4)} - \|u(t)\|^4_{L^4(\TTT^4)} + 2c_* \|u(t)\|^2_{L^2(\TTT^4)} + c_*^2C^4_4\|u(t)\|^4_{L^2(\TTT^4)}\geq \bar{\delta}\|u(t)\|^2_{\dot{H}^1(\TTT^4)},
\end{align} 
and in particular
\begin{equation}\label{eq:energyresult42}
E_{**}(u)(t) \geq \frac{1}{4}(1+\bar{\delta})\|u(t)\|^2_{\dot{H}^1(\TTT^4)}.
\end{equation}
\end{thm}
\begin{proof}
By the conservation of energy and mass, this theorem directly from Lemma \ref{lem:EnergyTrapping} by the continuity argument.
\end{proof}

\begin{remark}\label{rmk:EnergySimSobolev}
The energy trapping lemma (Theorem \ref{thm:EnergyTrapping}) shows that if the initial data satisfies the condition (\ref{cond:focusing2}) or (\ref{cond:focusing3}) then the solution $u(t)$ satisfies $\|u(t)\|_{\dot{H}^1(\TTT^4)}< \|W\|_{\dot{H}^1(\RRR^4)}$ for all $t$ in the lifespan of the solution.
So Theorem \ref{thm:focusing} implies Corollary \ref{cor:focusing}.
 In particular,  we also obtain that $E_*(u)(t)\simeq \|u(t)\|^2_{H^1_*(\TTT^4)}$ under the assumption (\ref{ass:i}) and $E_{**}(u)(t)\simeq \|u(t)\|^2_{\dot{H}^1(\TTT^4)}$ under the assumption (\ref{ass:ii}) by Theorem \ref{thm:EnergyTrapping}.
\end{remark}

\section{Adapted function spaces}
In this section, we introduce $X^s$ and $Y^s$ spaces which are based on the atomic spaces $U^p$ and $V^p$ which were originally developed in unpublished work on wave maps by Tararu and then were applied to PDEs in \cite{hadac2009well}\cite{herr2011global}\cite{herr2014strichartz}, while we'll use the $X^s$ and $Y^s$ spaces in the proof of the defocusing and focusing global wellposedness.
$\mathcal{H}$ is a separable Hilbert space on $\mathbb{C}$ and $\mathcal{Z}$ denotes the set of finite partitions $-\infty = t_0 < t_1 < ... <t_K = \infty$ of the real line, with the convention that  $v(\infty) := 0$ for any function $v : \mathbb{R} \to \mathcal{H}$.
\begin{definition}[Definition 2.1 in \cite{herr2011global}]
Let $1\leq p < \infty$. For $\{t_k\}_{k=0}^K \in \mathcal{Z}$ and $\{\phi_k\}_{k=0}^{K-1} \subset \mathcal{H}$ with $\sum_{k=0}^K \|\phi_k\|_{\mathcal{H}}^p = 1$ and $\phi_0 = 0$. A $U^p$-atom is a piecewise defined function  $a : \mathbb{R} \to \mathcal{H}$ of the form
$$
a = \sum_{k=1}^K \mathds{1}_{[t_{k-1}, t_k)} \phi_{k-1}.
$$
The atomic Banach space $U^p(\mathbb{R}, \mathcal{H})$ is then defined to be the set of all functions $u : \mathbb{R} \to \mathcal{H}$ such that 
$$
u = \sum_{j=1}^{\infty} \lambda_{j} a_j,\qquad  \text{for}\   U^p \text{-atoms}\   a_j,\qquad \{\lambda_j\}_j \in \ell^1,\quad \|u\|_{U^p}<\infty,
$$
where
$$
\|u\|_{U^p} : = \inf \{\sum_{j=1}^{\infty} |\lambda_j| : u = \sum_{j=1}^\infty \lambda_j a_j,\  \lambda_j \in \mathbb{C}\ \text{and}\ a_j\ \text{an}\ U^p\ \text{atom}\}.
$$
Here $\mathds{1}_{I}$ denotes the indicator function over the time interval $I$.
\end{definition}

\begin{definition}[Definition 2.2 in \cite{herr2011global}]
Let $1\leq p < \infty$. The Banach space $V^p(\mathbb{R}, \mathcal{H})$ is defined to be the set of all functions $v: \mathbb{R} \to \mathcal{H}$ with $v(\infty):=0$ and $v(-\infty) := \lim_{t\to -\infty} v(t)$ exists, such that 
$$
\|v\|_{V^p} : = \sup_{\{t_k\}_{k=0}^K\in \mathcal{Z}} (\sum_{k=1}^K \|v(t_k)-v(t_{k-1})\|^p_\mathcal{H})^\frac{1}{p}\quad
 \text{is}\  \text{finite}.
$$
Likewise, let $V_{-}^p$ denote the closed subspace of all $v \in V^p$ with $\lim_{t\to -\infty}v(t) = 0$. $V_{-, rc}^p$ means all right-continuous $V_{-}^p$ functions.
\end{definition}

\begin{remark}[Some embeding properties]\label{rmk:embedding}
Note that for $1\leq p \leq q < \infty$,
\begin{equation}
U^p(\mathbb{R}, \mathcal{H}) \hookrightarrow U^q(\mathbb{R}, \mathcal{H}) \hookrightarrow  L^{\infty}(\mathbb{R},\mathcal{H}),
\end{equation}
and functions in $U^p(\mathbb{R}, \mathcal{H})$ are right continuous, and $\lim_{t\to -\infty} u(t) = 0$ for each $u \in U^p(\mathbb{R}, \mathcal{H})$.  Also note that,
\begin{equation}
U^p(\mathbb{R}, \mathcal{H}) \hookrightarrow V^p_{-,rc} (\mathbb{R}, \mathcal{H}) \hookrightarrow U^q(\mathbb{R}, \mathcal{H}).
\end{equation}
\end{remark}

\begin{definition}[Definition 2.5 in \cite{herr2011global}]
For $s\in \mathbb{R}$, we let $U^p_{\Delta}H^s$, respectively $V^p_{\Delta}H^s$, be the space of all functions $u : \mathbb{R}\to H^s(\mathbb{T}^d)$ such that $t\mapsto e^{-it\Delta}u(t)$ is in $U^p(\mathbb{R}, H^s)$, respectively in  $V^p(\mathbb{R}, H^s)$ with norm
$$
\|u\|_{U^p(\mathbb{R}, H^s)} := \|e^{-it\Delta}u(t)\|_{U^p(\mathbb{R},H^s)}, 
\quad \|u\|_{V^p(\mathbb{R}, H^s)} := \|e^{-it\Delta}u(t)\|_{V^p(\mathbb{R},H^s)}.
$$
\end{definition}

\begin{definition}[Definition 2.6 in \cite{herr2011global}]\label{def:Xs}
For $s\in \mathbb{R}$, we define $X^s$ as the space of all functions $u : \mathbb{R}\to H^s(\mathbb{T}^d)$ such that for every $n\in \mathbb{Z}^d$, the map $t\mapsto e^{it|n|^2}\widehat{u(t)}(n)$ is in $U^2(\mathbb{R}, \mathcal{C})$, and with the norm
\begin{equation}
\|u\|_{X^s} : = (\sum_{n\in \mathbb{Z}^d}\langle n \rangle^{2s} \|e^{it|n|^2}\widehat{u(t)}(n)\|^2_{U_t^2})^{\frac{1}{2}}\quad \text{is finite}. 
\end{equation}
\end{definition}

\begin{definition}[Definition 2.7 in \cite{herr2011global}]
For $s\in \mathbb{R}$, we define $Y^s$  as the space of all functions $u : \mathbb{R}\to H^s(\mathbb{T}^d)$ such that for every $n\in \mathbb{Z}^d$, the map $t\mapsto e^{it|n|^2}\widehat{u(t)}(n)$ is in $V_{rc}^2(\mathbb{R}, \mathcal{C})$, and with the norm
\begin{equation}
\|u\|_{Y^s} : = (\sum_{n\in \mathbb{Z}^d}\langle n \rangle^{2s} \|e^{it|n|^2}\widehat{u(t)}(n)\|^2_{V_t^2})^{\frac{1}{2}}\quad \text{is finite}. 
\end{equation}
\end{definition}

Note that 
\begin{equation}\label{eq:Embedding}
U^2_{\Delta}H^s \hookrightarrow X^s \hookrightarrow Y^s \hookrightarrow V^2_{\Delta}H^s.
\end{equation}

\begin{prop}[Proposition 2.10 in \cite{hadac2009well}]\label{EstimateFreeSolution}
Suppose $u :=  e^{it\Delta}\phi$ which is a linear Schr\"{o}dinger solution, then for any $T>0$ we obtain that
$$\|u\|_{X^s([0,T])} \leq \|\phi\|_{H^s}.$$
\end{prop}
\begin{proof}
Since $u :=  e^{it\Delta}\phi$, then $\|u\|_{X^s}  = (\sum_{n\in \mathbb{Z}^d}\langle n \rangle^{2s} \|\widehat{\phi}(n)\|^2_{U_t^2})^{\frac{1}{2}} \leq \|\phi\|_{H^s}$.
\end{proof}

\begin{remark}
  Compared with Bourgain's $X^{s,b}$ first introduced in Bourgain,

\begin{align*}
\|v\|_{X^{s,b}} & = \|e^{-it\Delta}v\|_{H^b_tH^s_x},\\
\|v\|_{U^p_\Delta H^s} & = \|e^{-it\Delta}v\|_{U^p_tH^s_x},\\
\|v\|_{V^p_\Delta H^s} & = \|e^{-it\Delta}v\|_{V^p_tH^s_x}.\\
\end{align*}

  Also later we will see that the atomic spaces enjoy the similar duality and
transfer principle properties with $X^{s,b}$.
\end{remark}

\begin{remark}
Follow the definitions, it's easy to check the following embedding property:
\begin{equation}
U^2_{\Delta}H^s \hookrightarrow X^s \hookrightarrow Y^s \hookrightarrow V^2_{\Delta}H^s
\hookrightarrow L^\infty(\RRR, H^s).
\end{equation}
\end{remark}

\begin{definition}[$X^s$ and $Y^s$ restricted to a time interval $I$]
  For intervals $I\subset \RRR$, we define $X^s(I)$ and $Y^s(I)$ as following
  \begin{equation*}
    X^s(I) := \{ v\in C(I: H^s) : \|v\|_{X^s(I)} := \sup_{J\subset I ,\  |J|\leq 1}
    \inf_{\tilde{v}|_J = v}
    \|\tilde{v}\|_{X^s}<\infty\},
  \end{equation*}
  and
    \begin{equation*}
    Y^s(I) := \{ v\in C(I: H^s) : \|v\|_{Y^s(I)} := \sup_{J\subset I ,\  |J|\leq 1}
    \inf_{\tilde{v}|_J = v}
    \|\tilde{v}\|_{Y^s}<\infty\}.
  \end{equation*}
  
\end{definition}

We will consider our solution in $X^1(I)$ spaces, and then let's introduce nonlinear
norm $N(I)$.

\begin{definition}[Nonlinear norm $N(I)$]
  Let $I=[0, T]$, then
  \begin{equation*}
    \|f\|_{N(I)} := \left\| \int_0^t e^{i(t-t')\Delta} f(t') dt'\right\|_{X^1(I)}
  \end{equation*}
\end{definition}

\begin{prop}[Proposition 2.11 in \cite{herr2014strichartz}]\label{prop:dual}
Let $s>0$. For $f\in L^1(I, H^1(\mathbb{T}^4))$
we have
\begin{equation}
\|f \|_{N(I)}\leq
\sup_{v\in Y^{-1}(I)}
\left|\int_I\int_{\mathbb{T}^4} f(t, x)\overline{v(t,x)}dxdt\right|.
\end{equation}
\end{prop}

Now, we will need a weaker norm $Z$, which plays a similar role as
$L^{10}_{t,x}$ norm in \cite{colliander2008global}.

\begin{definition}
  \[\|v\|_{Z(I)} := \sup_{J\subset I, \ |J|\leq 1}
  \left( \sum_{N \in 2^\ZZZ} N^2 \|P_N v\|^4_{L^4(\TTT^4\times J)}
  \right)^{\frac{1}{4}}.\]
\end{definition}

\begin{remark}
  $\|v\|_{Z(I)}$ actually can be considered as
  \[
  \sum_{p\in \{p_1, p_2, \cdots, p_k\}} \sup_{J\subset I, \ |J|\leq 1}
  \left( \sum_{N \in 2^\ZZZ} N^{6-p} \|P_N v\|^p_{L^p(\TTT^4\times J)}  \right)^{\frac{1}{p}},
  \]
  where $\{p_1, p_2, \cdots, p_k\}$ should be the $L^p$ estimates that we need to use in
  the proof of nonlinear estimate.
 In our case, we only need $\|P_N u\|_{L^4(\TTT^4\times I)} \lesssim \|P_N u\|_{Z(I)}$
  in the proof of the nonlinear estimates, so we choose $\{p_1, p_2, \cdots, p_k\}= \{4\}$.
\end{remark}

The following property shows us that $Z(I)$ is a weaker norm than $X^1(I)$.
\begin{prop}\label{prop:ZinX}
  \[\|v\|_{Z(I)} \lesssim \|v\|_{X^1(I)}\].
\end{prop}
\begin{proof}
  By the definition of $Z(I)$ and the following Strichartz type estimates
  {Proposition \ref{coro:strichartz}}, we obtain that
  \begin{align*}
  \sup_{J\subset I, \ |J|\leq 1}
  \left( \sum_{N \text{ dydic number}} N^2 \|P_N v\|^4_{L^4(\TTT^4\times J)}  \right)^{\frac{1}{4}}
  &\lesssim \sup_{J\subset I, \ |J|\leq 1}
  \left( \sum_{N \text{ dydic number}} N^4 \|P_N v\|^4_{X^0(J)}  \right)^{\frac{1}{4}}\\
  &\lesssim \|v\|_{X^1(I)}.
  \end{align*}
\end{proof}

\begin{prop}[Proposition 2.19 in \cite{hadac2009well}]\label{prop:transfer}
Let $T_0 : L^2 \times \cdots \times L^2 \to L^{1}_{loc}$ be a m-linear operator.
Assume that for some $1\leq p,\ q \leq \infty$
\begin{equation}
\| T_0 (e^{it\Delta }\phi_1,\cdots, e^{it\Delta }\phi_m)\|_{L^p(\mathbb{R},
L_x^q(\mathbb{T}^d))} \lesssim \prod_{i=1}^m\|\phi_i\|_{L^2(\mathbb{T}^d)}.
\end{equation}
Then, there exists an extension $T : U^p_{\Delta}\times \cdots \times U^p_{\Delta}
\to L^p(\mathbb{R}, L^q(\mathbb{T}^d))$ satisfying
\begin{equation}
\|T(u_1, \cdots, u_m)\|_{L^p(\mathbb{R}, L^q(\mathbb{T}^d))} \lesssim \prod_{i=1}^m
\|u_i\|_{U^p_{\Delta}},
\end{equation}
such that $T(u_1, \cdots , u_m) (t, \cdot) = T_0 (u_1(t), \cdots , u_m(t))(\cdot)$,
a.e.
\end{prop}

\begin{lem}[Strichartz type estimates \cite{bourgain1993fourier}\cite{bourgain2015proof}]
  \label{prop:strichartz}
  If $p>3$, then
  \begin{equation*}
    \|P_N e^{it\Delta} f\|_{L^p_{t,x}([-1,1]\times\TTT^4)} \lesssim_p
    N^{2-\frac{6}{p}} \|f\|_{L^2_x}
  \end{equation*}
  and
  \begin{equation*}
    \|P_C e^{it\Delta} f\|_{L^p_{t,x}([-1,1]\times\TTT^4)} \lesssim_p
    N^{2-\frac{6}{p}} \|f\|_{L^2_x}
  \end{equation*}
  where $C$ is a cube of side length N and $f\in L^2(\TTT^4)$.
\end{lem}

By the transfer principle proposition {(Proposition \ref{prop:transfer})} and Strichartz type estimate
Lemma \ref{prop:strichartz}, we obtain the following corollary:
\begin{cor}\label{coro:strichartz}
  If $p>3$, for any $v\in U^p_\Delta([-1, 1])$,
  \[
  \|P_N v\|_{L^p([-1,1]\times\TTT^4)} \lesssim_p N^{2-\frac{6}{p}}
  \|v\|_{U^p_\Delta([-1, 1])},
  \]
  and
  \[
  \|P_C v\|_{L^p([-1,1]\times\TTT^4)} \lesssim_p N^{2-\frac{6}{p}}
  \|v\|_{U^p_\Delta([-1, 1])},
  \]
  where $C$ is a cube of side length $N$.
\end{cor}

\section{Local well-posedness and Stability theory}
In this section, we present large-data local well-poesdness and stability
results. Although Herr, Tataru and Tzvetkov's idea \cite{herr2014strichartz}
together with Bourgain and Demeter's result \cite{bourgain2015proof} gives the local well-posedness
of (\ref{eq:IVP}), to obtain the stability results, we need a refined
nonlinear estimate and the local well-posedness result.

\begin{definition}[Definition of solutions]
  Given an interval $I\subseteq \RRR$, we call $u\in C(I: H^1(\TTT^4))$ a strong
  solution of (\ref{eq:IVP}) if $u\in X^1(I)$ and $u$ satisfies that for all
  $t, s\in I$,
  \[
  u(t) = e^{i(t-s)\Delta} u(s) - i\mu
  \int_s^t e^{i(t-t')\Delta} u(t')|u(t')|^2 dt'.
  \]
\end{definition}

First, we need to introduce
\begin{equation}\label{def:Z'}
\|u\|_{Z'(I)} := \|u\|_{Z(I)}^{\frac{3}{4}}\|u\|^\frac{1}{4}_{X^1(I)}.
\end{equation}

\begin{lem}[Bilinear estimates in \cite{herr2014strichartz}]\label{lem:bilinear}
  Assuming $|I|\leq 1$ and $N_1\geq N_2$, then we hold that
  \begin{equation}
    \|P_{N_1} u_1 P_{N_2} u_2\|_{L^2_{x,t}(\TTT^4\times I)}
    \lesssim (\frac{N_2}{N_1}+\frac{1}{N_2})^\kappa \|P_{N_1} u_1\|_{Y^0(I)}
    \|P_{N_2}u_2\|_{Y^1(I)}
  \end{equation}
  for some $\kappa >0$.
\end{lem}

\begin{remark}
  This Bilinear estimate is Proposition 2.8 in \cite{herr2014strichartz}. The proof
  of Lemma \ref{lem:bilinear} relies on $L^p$ estimates in
  Corolla \ref{prop:strichartz} (for some $p<4$). In the proof not only
  we need the decoupling properties for spatial frequency, but also
  we need further strip partitions to apply the decoupling properties for time
  frequency.
\end{remark}

Let's introduce a refined nonlinear estimate.

\begin{prop}[Refined nonlinear estimate]\label{prop:nonlinear}
  For $u_k\in X^1(I)$, $k=1, 2, 3$, $|I|\leq 1$, we hold the estimate
  \begin{equation}\label{eq:nonlinear}
    \|\prod_{k=1}^3 \widetilde{u_k}\|_{N(I)}\lesssim
    \sum_{\{i, j, k\}=\{1, 2, 3\}} \|u_i\|_{X^1(I)}\|u_j\|_{Z'(I)}\|u_k\|_{Z'(I)}
  \end{equation}
  where $\widetilde{u_k} =u_k$ or $\widetilde{u_k} = \overline{u_k}$
  for $k=1, 2, 3$.

  In particular, if there exist constants $A, B>0$, such that $u_1 = P_{>A} u_1$,
  $u_2 = P_{>A} u_2$ and $u_3 = P_{<B} u_3$, then we obtain that
\begin{equation}\label{eq:nonlinearP}
  \|\prod_{k=1}^3 \widetilde{u_k}\|_{N(I)}\lesssim
  \|u_1\|_{X^1(I)}\|u_2\|_{Z'(I)}\|u_3\|_{Z'(I)}
  +  \|u_2\|_{X^1(I)}\|u_1\|_{Z'(I)}\|u_3\|_{Z'(I)}.
\end{equation}
\end{prop}

\begin{proof}
  Suppose $N_0$, $N_1$, $N_2$, $N_3$ are dyadic and
  WLOG we assume $N_1\geq N_2 \geq N_3$. By the Proposition \ref{prop:dual}, we obtain that
\begin{align*}
  &\|\prod_{k=1}^3 \widetilde{u_k}\|_{N(I)} \lesssim
  \sup_{\|u_0\|_{Y^{-1}}} |\int_{\TTT^4\times I} \overline{u_0}
  \prod_{k=1}^3 \widetilde{u_k} \,dxdt|\\
  \leq & \sup_{\|u_0\|_{Y^{-1}}} \sum_{N_0, N_1\geq N_2\geq N_3}|
  \int_{\TTT^4\times I} \overline{P_{N_0} u_0}
  \prod_{k=1}^3 P_{N_k}\widetilde{u_k} \,dxdt|
\end{align*}

Then we know that $N_1 \sim \max (N_2, N_0)$ by the spatial frequency orthogonality.
There are two cases:
\begin{enumerate}
  \item $N_0\sim N_1 \geq N_2 \geq N_3$;
  \item $N_0\leq N_2 \sim N_1 \geq N_3$.
\end{enumerate}

\case{1}{$N_0\sim N_1 \geq N_2 \geq N_3$}
By Cauchy-Schwartz inequality and Lemma \ref{lem:bilinear}, we have that
\begin{equation}\label{3.2}
\begin{split}
  & |\int \overline{P_{N_0} u_0} P_{N_1} \widetilde{u_1} P_{N_2} \widetilde{u_2}
  P_{N_3} \widetilde{u_3}\, dxdt|
  \leq \|P_{N_0} u_0 P_{N_2} u_2\|_{L^2_{x,t}}
    \|P_{N_1}u_1 P_{N_3} u_3\|_{L^2_{x,t}}\\
  \lesssim& (\frac{N_3}{N_1} + \frac{1}{N_3})^\kappa
  (\frac{N_2}{N_0} + \frac{1}{N_2})^\kappa \|P_{N_0} u_0\|_{Y^0(I)}
  \|P_{N_1}u_1\|_{Y^0(I)}\|P_{N_2}u_2\|_{X^1(I)}\|P_{N_3}u_3\|_{X^1(I)}
\end{split}
\end{equation}

Assume $\{C_j\}$ is a cube partition of size $N_2$, and $\{C_k\}$ is a cube
partition of size $N_3$. By $\{P_{C_j}P_{N_0} u_0 P_{N_2} u_2\}_j$
and $\{P_{C_k}P_{N_1}u_1 P_{N_3}u_3\}_k$ are both almost orthogonal,
Corollary \ref{coro:strichartz} and definition of $Z$ norm, we obtain that
\begin{equation}\label{3.3}
  \begin{split}
    & |\int \overline{P_{N_0} u_0} P_{N_1} \widetilde{u_1} P_{N_2} \widetilde{u_2}
    P_{N_3} \widetilde{u_3}\, dxdt|
    \leq \|P_{N_0} u_0 P_{N_2} u_2\|_{L^2_{x,t}}
    \|P_{N_1}u_1 P_{N_3} u_3\|_{L^2_{x,t}}\\
    \lesssim & (\sum_{C_j} \|P_{C_j}P_{N_0} u_0 P_{N_2} u_2\|^2_{L^2_{x,t}})^{\frac{1}{2}}
    (\sum_{C_k} \|P_{C_k}P_{N_1}u_1 P_{N_3}u_3 \|^2)^{\frac{1}{2}}\\
    \lesssim & (\sum_{C_j} \|P_{C_j}P_{N_0} u_0\|^2_{L^4_{x,t}}
    \|P_{N_2} u_2\|^2_{L^4_{x,t}})^{\frac{1}{2}}
    (\sum_{C_k} \|P_{C_k}P_{N_1}u_1\|_{L^4_{x,t}}^2
    \| P_{N_3}u_3 \|^2_{L^4_{x,t}})^{\frac{1}{2}}\\
    \lesssim & (\sum_{C_j} \|P_{C_j}P_{N_0} u_0\|^2_{Y^0(I)}
    (N_2^{\frac{1}{2}}\|P_{N_2} u_2\|_{L^4_{x,t}})^2)^{\frac{1}{2}}
    (\sum_{C_k} \|P_{C_k}P_{N_1} u_1\|^2_{Y^0(I)}
    (N_3^{\frac{1}{2}}\|P_{N_3} u_3\|_{L^4_{x,t}})^2)^{\frac{1}{2}}\\
    \lesssim& \|P_{N_0}u_0\|_{Y^0(I)}\|P_{N_1}u_1\|_{Y^0(I)}\|P_{N_2}u_2\|_{Z(I)}
    \|P_{N_2}u_2\|_{Z(I)}.
  \end{split}
\end{equation}

Interpolating (\ref{3.2}) with (\ref{3.3}) we obtain that
\begin{equation}\label{3.4}
\begin{split}
  &|\int \overline{P_{N_0} u_0} P_{N_1} \widetilde{u_1} P_{N_2} \widetilde{u_2}
  P_{N_3} \widetilde{u_3}\, dxdt|\\
  \lesssim&
  (\frac{N_3}{N_1} + \frac{1}{N_3})^{\kappa_1}
  (\frac{N_2}{N_0} + \frac{1}{N_2})^{\kappa_1}
  \|P_{N_0}u_0\|_{Y^{-1}(I)}\|P_{N_1}u_1\|_{X^1(I)}\|P_{N_2}u_2\|_{Z'(I)}
  \|P_{N_2}u_2\|_{Z'(I)}.
\end{split}
\end{equation}

Then we sum (\ref{3.4}) over all $N_0\sim N_1\geq N_2\geq N_3$,
\begin{align*}
    &\sum_{N_0\sim N_1\geq N_2\geq N_3}
    (\frac{N_3}{N_1} + \frac{1}{N_3})^{\kappa_1}
    (\frac{N_2}{N_0} + \frac{1}{N_2})^{\kappa_1}
    \|P_{N_0}u_0\|_{Y^{-1}(I)}\|P_{N_1}u_1\|_{X^1(I)}\|P_{N_2}u_2\|_{Z'(I)}
    \|P_{N_2}u_2\|_{Z'(I)}\\
    \lesssim & \|u_0\|_{Y^{-1}(I)}\|u_1\|_{X^1{I}}\|u_2\|_{Z'(I)}\|u_3\|_{Z'(I)}.
\end{align*}

\case{2}{$N_0\leq N_2\sim N_1\geq N_3$}

Similarly we have that
\begin{equation}\label{3.5}
\begin{split}
  &|\int \overline{P_{N_0} u_0} P_{N_1} \widetilde{u_1} P_{N_2} \widetilde{u_2}
  P_{N_3} \widetilde{u_3}\, dxdt|\\
  \lesssim& (\frac{N_3}{N_1} + \frac{1}{N_3})^\kappa
  (\frac{N_0}{N_2} + \frac{1}{N_0})^\kappa \|P_{N_0} u_0\|_{Y^0(I)}
  \|P_{N_1}u_1\|_{Y^0(I)}\|P_{N_2}u_2\|_{X^1(I)}\|P_{N_3}u_3\|_{X^1(I)}.
\end{split}
\end{equation}

Similar with (\ref{3.3}), we obtain that:

\begin{equation}\label{3.6}
\begin{split}
  &|\int \overline{P_{N_0} u_0} P_{N_1} \widetilde{u_1} P_{N_2} \widetilde{u_2}
  P_{N_3} \widetilde{u_3}\, dxdt|\\
  \lesssim& \|P_{N_0} u_0\|_{Y^0(I)}
  \|P_{N_1}u_1\|_{Y^0(I)}\|P_{N_2}u_2\|_{Z(I)}\|P_{N_3}u_3\|_{Z(I)}.
\end{split}
\end{equation}

We interpolate (\ref{3.5}) with (\ref{3.6}) and sum over $N_0\leq N_2\sim N_1\geq N_3$. Then we have that
\begin{align*}
  &\sum_{N_0\leq N_2\sim N_1\geq N_3}
  |\int \overline{P_{N_0} u_0} P_{N_1} \widetilde{u_1} P_{N_2} \widetilde{u_2}
  P_{N_3} \widetilde{u_3}\, dxdt|\\
  \lesssim & \|P_{N_0} u_0\|_{Y^{-1}(I)}
  \|P_{N_1}u_1\|_{X^1(I)}\|P_{N_2}u_2\|_{Z'(I)}\|P_{N_3}u_3\|_{Z'(I)}.
\end{align*}

 Next we summarize these two cases and similarly consider
 $N_1\geq N_3\geq N_2$, $N_2\geq N_1\geq N_3$, $N_2\geq N_3\geq N_1$,
 $N_3\geq N_1\geq N_2$, and $N_3\geq N_2\geq N_1$,
 we can get the desired estimate (\ref{eq:nonlinear}).

In particular, if there exist constants $A, B>0$ such that $u_1 = P_{>A} u_1$,
$u_2 = P_{>A} u_2$ and $u_3 = P_{<B} u_3$, then we only consider the sum when
 $N_1\geq N_2\gtrsim N_3$ and
$N_2\geq N_1\gtrsim N_3$. So we get the estimate (\ref{eq:nonlinearP}).
\end{proof}

\begin{prop}[Local Wellposedness]\label{prop:lwp}
  Assume that $E>0$ is fixed. There exists $\delta_0 = \delta_0(E)$ such that if
  \[
  \|e^{it\Delta} u_0\|_{Z'(I)} < \delta
  \]
  for some $\delta \leq \delta_0$, some interval $0\in I$ with $|I|\leq 1$ and
  some function $u_0\in H^1(\TTT^4)$ satisfying $\|u_0\|_{H^1}\leq E$, then
  there exists a unique strong solution to (\ref{eq:NLS}) $u\in X^1(I)$
  such that $u(0) = u_0$. Besides we also have
  \begin{equation}\label{ineq:smoothing}
    \|u-e^{it\Delta}u_0\|_{X^1(I)}\leq \dd^{\frac{5}{3}}.
  \end{equation}
\end{prop}
\begin{proof}
  First, we consider the set
  \[
  S = \{ u\in X^1(I): \|u\|_{X^1(I)}\leq 2E,\qquad \|u\|_{Z'(I)}\leq a\},
  \]
  and the mapping
  \[
  \Phi(v) = e^{it\Delta} u_0 -i \mu\int_0^t e^{i(t-s)\DD} v(s)|v(s)|^2\, ds.
  \]

  For $u$, $v\in S$, by Proposition \ref{prop:nonlinear}, there exists
  a constant $C>0$, we have that
  \begin{align*}
    &\|\Phi(u) - \Phi(v)\|_{X^1(I)}\\
    \leq& C \left( \|u\|_{X^1(I)}+\|v\|_{X^1(I)}\right)
    \left( \|u\|_{Z'(I)}+\|v\|_{Z'(I)}\right) \|u-v\|_{X^1(I)}\\
    \leq&C Ea\|u-v\|_{X^1(I)}
  \end{align*}
Similarly, using Proposition \ref{EstimateFreeSolution} and nonlinear estimate
Proposition \ref{prop:nonlinear}, we also obtain that
\begin{align*}
  \|\Phi(u)\|_{X^1(I)} & \leq \|\Phi(0)\|_{X^1(I)} + \|\Phi(u)-\Phi(0)\|_{X^1(I)}\\
  &\leq \|u_0\|_{H^1} + C Ea^2
\end{align*}
and
\begin{align*}
  \|\Phi(u)\|_{Z'(I)} & \leq \|\Phi(0)\|_{Z'(I)} + \|\Phi(u)-\Phi(0)\|_{Z'(I)}\\
  &\leq \dd + C Ea^2.
\end{align*}

Now, we choose $a=2\dd$ and we let $\dd_0 =\dd_0(E)$ be small enough.
We see that $\Phi$ is a contraction on $S$, so we have a fixed point $u$.
And it's easy to check (\ref{ineq:smoothing}) and uniqueness in $X^1(I)$.
\end{proof}
By a similar idea of Herr-Tataru-Tzvetkov \cite{herr2011global}, we can easily prove the global well-posedness result with
small initial data by using Theorem \ref{prop:lwp}.
\begin{prop}[Small data global wellposedness]\label{prop:smallDataGWP}
  If $\|\phi\|_{H^1(\TTT^4)} =\dd\leq \dd_0$, then the unique strong solution
  with initial data $\phi$ is global and satisfies
  \[\|u\|_{X^1([-1, 1])} \leq 2\dd\] and moreover
  \[\|u-e^{it\DD}\phi\|_{X^1([-1, 1])}\lesssim \dd^2.\]
\end{prop}

\begin{lem}[$Z$-norm controls the global existence]\label{lem:cgwp}
Assume that $I\subseteq \RRR$ is a bounded open interval.
  \begin{enumerate}
    \item  If $E$ is a nonnegative finite number, that $u$ is
    a strong solution of (\ref{eq:NLS}) and 
    \[
    \|u\|_{L^\infty_t(I, H^1)}\leq E.
    \]
    Then, if
    \[
    \|u\|_{Z(I)} <+\infty
    \]
    there exists an open interval $J$ with $\bar{I}\subset J$ such that $u$
    can be extended to a strong solution of (\ref{eq:NLS}) on $J$, besides
    \[
	\|u\|_{X^1(I)}\leq C(E, \|u\|_{Z(I)}).    
    \]
    \item (GWP with a priori bound)Assume $C$ is some positive finite number and we have a priori bound $\|u\|_{Z(I)}<C$,
    for any solution $u$ of
    (\ref{eq:NLS}) in the interval $I$, then this IVP (\ref{eq:IVP}) is well-posedness
    on $I$.
    (In particular, if $u$ blows up in finite time, then $u$ blows up in the
    $Z$-norm.)
  \end{enumerate}
\end{lem}

\begin{proof}
  Consider the case $I = (0, T)$.
  \begin{enumerate}
    \item By the continuity arguments of $h(s) =
    \|e^{i(t-T_1)\DD}u(T_1)\|_{Z'(T_1, T_1+s)}$ where $T_1\geq T-1$ such that
    $\|u\|_{Z(T_1, T)}\leq \ee$.
    \item Combined (1) and Proposition \ref{prop:lwp}, it's trivial to know.
  \end{enumerate}
\end{proof}

\begin{remark}
  This proof determines the ratio of $X^1$ norm and $Z$ norm in the definition
  of $Z'$ norm \ref{def:Z'}, and the portion of $Z$ norm in $Z'$ norm can be any number strictly between
  $\frac{1}{2}$ and $1$.
\end{remark}

\begin{prop}[Stability]\label{prop:stability}
  Assume $I$ is an open bounded interval, $\mu \in [-1, 1]$, and $\tilde{u}
  \in X^1(I)$ satisfies the approximate Schr\"{o}dinger equation
  \begin{equation}\label{eq:aNLS}
    (i\partial_t + \Delta) \tilde{u} =\mu \tilde{u}|\tilde{u}|^2 + e,\qquad
    \text{on }\TTT^4\times I.
  \end{equation}
  Assume in addition that
  \begin{equation}
    \|\tilde{u}\|_{Z(I)} + \|\tilde{u}\|_{L^\infty_t(I, H^1(\TTT^4)}\leq M,
  \end{equation}
  for some $M\in [1, \infty]$. Assume $t_0\in I$ and $u_0\in H^1(\TTT^4)$ is
  such that the smallness condition:
  \begin{equation}\label{ineq:err}
    \|u_0 - \tilde{u}(0)\|_{H^1(\TTT^4)} + \|e\|_{N(I)}\leq \ee
  \end{equation}
  holds for some $0<\ee<\ee_1$, where $\ee_1 \leq 1$.
  $\ee_1 =\ee_1(M)>0$ is a small constant.

  Then there exists a strong solution $u\in X^1(I)$ of the NLS
  \[
    (i\partial_t + \Delta) u =\mu u|u|^2,
  \]
  such that $u(t_0) =u_0$ and
  \begin{equation}\label{3.11}
    \begin{split}
      \|u\|_{X^1(I)} +\|\tilde{u}\|_{X^1(I)}&\leq C(M),\\
      \|u-\tilde{u}\|_{X^1(I)}&\leq C(M)\ee.
    \end{split}
  \end{equation}
\end{prop}
\begin{proof}
  First, we need to show the short time Stability, which follows a similar proof
  as the proof of Proposition \ref{prop:lwp}.
  Then, by using Lemma \ref{lem:cgwp}, we extend to the entire time interval.
\end{proof}
\section{Euclidean profiles}
In this section, we introduce the Euclidean profiles which are linear and nonlinear Sch\"odinger solutions on $\TTT^4$  concentrated at a point. The Euclidean profiles perform similar with the solutions in the Euclidean space $\RRR^4$ and hence Euclidean profiles hold some similar well-posedness and scattering properties 
by using the theory for the NLS in Euclidean space $\RRR^4$, which is proven by
Ryckman and Vi{\c{s}}an \cite{visan2007global}\cite{vicsan2011global} (the defocusing case) and Dodson \cite{DodsonFocusing} (the focusing case), as a black box. This is an analogue in 4 dimensions of the section 4 in \cite{ionescu2012energy}, we follows closely the argument in the section 4 of \cite{ionescu2012energy}.

We fix a spherically symmetric function $\eta \in C_0^{\infty}(\RRR^4)$ supported
in the ball of radius 2 and equal to 1 in the ball of radius 1.

Given $\phi\in \dot{H}^1(\RRR^4)$ and a real number $N\geq 1$ we define

\begin{equation}\label{def:euclideanProfile}
  \begin{split}
    Q_N \phi \in H^1(\RRR^4),&\qquad (Q_N\phi)(x) =\eta(x/N^{\frac{1}{2}})\phi(x),\\
    \phi_N \in H^1(\RRR^4),&\qquad \phi_N (x) = N(Q_N \phi)(Nx),\\
    f_N \in H^1(\TTT^4),&\qquad f_N(y) =\phi_N(\Psi^{-1}(y)),
  \end{split}
\end{equation}
where $\Psi :\{x\in \RRR^4: |x|<1 \} \to O_0\subseteq \TTT^4, \quad \Psi(x)=x$.

The cutoff function $\eta(\frac{x}{N^{1/2}})$
is useful to concentrate our focus on the range of a point, and the choice
of the order $1/2$ actually can be chosen any number between $1/2$ and $1$.

Thus $Q_N \phi$ is a compactly supported modification of profile $\phi$.
$\phi_N$ is a $\dot{H}^1$-invariant rescaling of $Q_N \phi$, and $f_N$ is
the function obtained by transferring $\phi_N$ to a neighborhood of 0 in $\TTT^4$.

\begin{thm}[GWP of the defocusing cubic NLS in $\RRR^4$ \cite{visan2007global}\cite{vicsan2011global}]
  \label{thm:GWPinR}
  Assume $\phi \in \dot{H}^1(\RRR^4)$ then there is an unique global solution
  $v\in C(\RRR: \dot{H}^1(\RRR^4))$ of the initial-value problem
  \begin{equation}
    (i\partial_t +\DD)v = v|v|^2, \qquad v(0) = \phi,
  \end{equation}
  and
  \begin{equation}
    \| \nabla_{\RRR^4} v\|_{(L_t^\infty L_x^2\cap L^2_tL_x^4)(\RRR\times\RRR^4)}
    \leq C(E_{\RRR^4}(\phi)) < +\infty.
  \end{equation}

  Moreover, this solution scatters in the sense that there exists $\phi^{\pm\infty}
  \in \dot{H}^1(\RRR^4)$, such that
  \begin{equation}
    \|v(t) - e^{it\DD}\phi^{\pm\infty}\|_{\dot{H}^1(\RRR^4)}\to 0,\text{ as }
    t\to \pm\infty.
  \end{equation}
  Besides, if $\phi\in H^5(\RRR^4)$ then $v\in C(\RRR: H^5(\RRR^4))$ and
  \begin{equation}
    \sup_{t\in\RRR} \|v(t)\|_{H^5(\RRR^4)} \lesssim_{\|\phi\|_{H^5(\RRR^4)}}
    \lesssim 1.
  \end{equation}
\end{thm}

\begin{thm}[GWP of the focusing cubic NLS in $\RRR^4$ \cite{DodsonFocusing}]\label{thm:GWPfocusing}
  Assume $\phi \in \dot{H}^1(\RRR^4)$, under the assumption that
  \[
\sup_{t\in \text{lifespan of }v}\|v(t)\|_{\dot{H}^1(\RRR^4)}<\|W\|_{\dot{H}^1(\RRR^4)},
  \]
  then there is an unique global solution
  $v\in C(\RRR: \dot{H}^1(\RRR^4))$ of the initial-value problem
  \begin{equation}\label{eq:ivpfocusing}
    (i\partial_t +\DD)v = -v|v|^2, \qquad v(0) = \phi,
  \end{equation}
  and
  \begin{equation}
    \| \nabla_{\RRR^4} v\|_{(L_t^\infty L_x^2\cap L^2_tL_x^4)(\RRR\times\RRR^4)}
    \leq C(\|\phi\|_{\dot{H}^1(\RRR^4)}, E_{\RRR^4}(\phi))<+\infty.
  \end{equation}

  Moreover, this solution scatters in the sense that there exists $\phi^{\pm\infty}
  \in \dot{H}^1(\RRR^4)$, such that
  \begin{equation}
    \|v(t) - e^{it\DD}\phi^{\pm\infty}\|_{\dot{H}^1(\RRR^4)}\to 0,\text{ as }
    t\to \pm\infty.
  \end{equation}
  Besides, if $\phi\in H^5(\RRR^4)$ then $v\in C(\RRR: H^5(\RRR^4))$ and
  \begin{equation}
    \sup_{t\in\RRR} \|v(t)\|_{H^5(\RRR^4)} \lesssim_{\|\phi\|_{H^5(\RRR^4)}} 1.
  \end{equation}
\end{thm}

\begin{remark}[Persistence of regularity]
  Consider $\phi \in H^5(\RRR^4)$, and $v\in C(\RRR: \dot{H}^1(\RRR^4))$ is the solution of (\ref{eq:NLS}) with $v(0) = \phi$ and satisfying
  \[
\| \nabla_{\RRR^4}v \|_{(L_t^\infty L_x^2\cap L^2_tL_x^4)(\RRR\times\RRR^4)}
  < +\infty.\]

  So we can have a finite partition $\{I_k\}_{k=1}^K$ of $\RRR$, ($I_k =
  [t_{k-1},t_k)$, where $t_k =\infty$.)
  s.t. $\| \nabla_{\RRR^4} v\|_{L_t^4L_x^{8/3}}< \frac{1}{2}$, for each $k$,
  \begin{align*}
    &\|v(t)\|_{L_t^\infty(I_k: H^5(\RRR^4))}\leq \|e^{i(t-t_{k-1})\DD}v(t_{k-1})\|_{H^5}
    + \|\langle \nabla  \rangle^5 |v(t)|^2v(t)\|_{L^2_tL_x^{4/3}(I_k)}\\
    \leq & \|v(t_{k-1})\|_{H^5_x}+\|\langle \nabla\rangle^5 v\|_{L^\infty_tL^2_x(I_k)}
    \|v(t)\|^2_{L^4_tL_x^8(I_k)}\\
    \leq & \|v(t_{k_1})\|_{H^5} + \frac{1}{4} \|v\|_{L^\infty_t(I_k: H^5(\RRR^4))}
  \end{align*}
  which implies $\|v(t)\|_{L^\infty_t(I_k: H^5(\RRR^4))} \leq \frac{4}{3}
  \|v(t_{k-1})\|_{H^5}$ for each $1\leq k \leq K$, so
  $\|v(t)\|_{L^\infty_t(\RRR: H^5_x(\RRR^4))}<\infty$.
\end{remark}

\begin{thm}\label{thm:4.2}
  Assume  $T_0\in (0,\infty)$, and
  $\mu \in\{-1, 0, 1\}$ are given, and define $f_N$ as (\ref{def:euclideanProfile})
  above. 
  Suppose 
  \[\|\phi\|_{\dot{H}^1(\RRR^4)}< +\infty, \qquad \text{when }\mu \in \{0, 1\};\]
  or under the assumption that if $v$ is a solution of (\ref{eq:ivpfocusing}) then $v$ satisfies
  \[\sup_{t\in \text{lifespan of }v}\|v(t)\|_{\dot{H}^1(\RRR^4)}<\|W\|_{\dot{H}^1(\RRR^4)}, , \qquad \text{when }\mu \in \{-1\}.\]
  Then the following conclusions hold:
  \begin{enumerate}
    \item There is $N_0 = N_0(\phi, T_0)$ sufficiently large such that
    for any $N\geq N_0$ there is an unique solution $U_N\in C((-T_0N^{-2},T_0N^{-2}):H^1(\TTT^4))$
    of the initial value problem
    \begin{equation}
      (i\partial_t +\DD)U_N = \mu U_N|U_N|^2, \qquad U_N(0)=f_N.
    \end{equation}
    Moreover, for any $N\geq N_0$,
    \begin{equation}
      \|U_N\|_{X^1(-T_0N^{-2},T_0N^{-2})}\lesssim_{E_{\RRR^4}(\phi),\,\|\phi\|_{\dot{H}^1(\RRR^4)}} 1.
    \end{equation}

    \item Assume $\ee_1\in (0,1]$ is sufficiently small (depending on only $E_{\RRR^4}\phi$),
    $\phi'\in H^5(\RRR^4)$, and $\|\phi -\phi'\|_{\dot{H}^1(\RRR^4)}\leq \ee_1$.
    Let $v'\in C(\RRR: H^5(\RRR^4))$ denote the solution of the initial value problem
    \begin{equation}
      (i\partial_t +\DD)v' = \mu v'|v'|^2, \qquad v'(0)=\phi'.
    \end{equation}
    For $R$, $N\geq 1$, we define
    \begin{equation}
      \begin{split}
        v_R'(x, t) =\eta(x/R) v'(x, t), &\qquad (x, t)\in \RRR^4\times(-T_0, T_0)\\
        v'_{R,N} (x, t) =N v_R'(Nx, N^2t), &\qquad (x, t)\in \RRR^4\times(-T_0N^{-2}, T_0N^{-2})\\
        V_{R, N}(y, t) = v'_{R,N}(\Psi^{-1}(y), t), &\qquad (y, t)\in
        \TTT^4\times(-T_0N^{-2}, T_0N^{-2}).
      \end{split}
    \end{equation}
    Then there is $R_0\geq 1$(depending on $T_0$, $\phi'$ and $\ee_1$), for any
    $R\geq R_0$, we obtain that
    \begin{equation}
    \limsup_{N\to\infty} \|U_N-V_{R,N}\|_{X^1(-T_0N^{-2}, T_0N^{-2})}
    \lesssim_{E_{\RRR^4}(\phi),\,\|\phi\|_{\dot{H}^1(\RRR^4)}}\ee_1.
    \end{equation}
  \end{enumerate}
\end{thm}

$V_{R,N}$ can be considered as solve NLS firstly, then cutoff and scaling, while
$U_N$ can be considered as cutoff and scaling firstly, then solve NLS.

\begin{proof}
  We show Part(1) and Part(2) together, by Proposition \ref{prop:stability} (stability).

  Using Theorem \ref{thm:GWPinR} and Theorem \ref{thm:GWPfocusing}, we know $v'$ globally exists and satisfying
  \begin{equation*}
    \| \nabla_{\RRR^4} v' \|_{(L_t^\infty L_x^2\cap L^2_tL_x^4)(\RRR\times\RRR^4)}
    \lesssim 1,
  \end{equation*}
  and
  \begin{equation}\label{eq:scatteringsmooth}
    \sup_{t\in\RRR} \|v'(t)\|_{H^5(\RRR^4)}\lesssim_{\|\phi'\|_{H^5(\RRR^4)}} 1.
  \end{equation}

  Let's consider $v_R'(x, t) = \eta(x/R) v'(x,t)$.
  \begin{align*}
    &(i\partial_t + \DD_{\RRR^4})v_R' =(i\partial_t +\DD_{\RRR^4})
    (\eta(x/R)v'(x,t))\\
    =&\eta(x/R)(i\partial_t + \DD_{\RRR^4})v'(x,t) + R^{-2} v'(x,t)
    (\DD_{\RRR^4}\eta)(x/R)+2R^{-1}\sum_{j=1}^4\partial_j v'(x,t)\partial_j\eta(x/R).
  \end{align*}
  which implies
  \begin{equation*}
    (i\partial_t + \DD_{\RRR^4})v_R' = \lambda |v_R'|^2v_R' + e_R(x,t),
  \end{equation*}
  where $e_R(x,t) = \mu (\eta(x/R)-\eta^3(x/R))v'|v'|^2 +
  R^{-2} v'(x,t)(\DD_{\RRR^4}\eta)(x/R)+2R^{-1}
  \sum_{j=1}^4\partial_j v'(x,t)\partial_j\eta(x/R).$
  After scaling, we get
  \begin{equation*}
    (i\partial_t + \DD_{\RRR^4})v'_{R,N} = \mu |v_{R,N}'|^2v_R' +
    e_{R,N}(x,t),
  \end{equation*}
  where $e_{R,N}(x,t) = N^3 e_R (Nx, N^2t)$.
  with $V_{R,N}(y,t) = v'_{R,N}(\Phi^{-1}(y), t)$ and taking $N\geq 10R$,
  we obtain that
  \begin{equation}\label{4.3}
    (i\partial_t + \DD_{\RRR^4})V_{R,N}(y,t) = \mu |V_{R,N}|^2V_{R,N}
    +E_{R,N}(y,t),
  \end{equation}
where $E_{R,N}(y,t) = e_{R,N}(\Phi^{-1}(y),t).$

By Proposition \ref{prop:stability}, we need following conditions:
\begin{enumerate}
  \item $\|V_{R,N}\|_{L_t^\infty([-T_0N^{-2},T_0N^{-2}]: H^1(\TTT^4))}
  +\|V_{R,N}\|_{Z([-T_0N^{-2},T_0N^{-2}])}\leq M$;
  \item $\|f_N -V_{R,N}(0)\|_{H^1(\TTT^4)}\leq \ee$;
  \item $\|E_{R,N}\|_{N([-T_0N^{-2},T_0N^{-2}])}\leq \ee$.
\end{enumerate}
We will prove all 3 condition above:

\case{1}{$\|V_{R,N}\|_{L_t^\infty([-T_0N^{-2},T_0N^{-2}]: H^1(\TTT^4))}
+\|V_{R,N}\|_{Z([-T_0N^{-2},T_0N^{-2}])}\leq M$.}
Since $v'(x,t)$ globally exists, $V_{R,N}(y,t)$ also globally exists.
Given $T_0\in (0, \infty)$,
\begin{align*}
  &\sup_{t\in[-T_0N^{-2},T_0N^{-2}]}\|V_{R,N}(t)\|_{H^1(\TTT^4)}
  \leq \sup_{t\in[-T_0N^{-2},T_0N^{-2}]}\|v'_{R,N}(t)\|_{H^1(\RRR^4)}\\
  =&\sup_{t\in[-T_0N^{-2},T_0N^{-2}]}\|Nv'_R(Nx, N^2t)\|_{H^1(\RRR^4)}\\
  =&\sup_{t\in[-T_0N^{-2},T_0N^{-2}]}\frac{1}{N} \|v'_R(N^2t)\|_{L^2(\RRR^4)}
  +\|v'_R(N^2t)\|_{\dot{H}^1(\RRR^4)}\\
  \leq& \sup_{t\in[-T_0, T_0]} \|v'_R\|_{H^1(\RRR^4)} =\sup_{t\in[-T_0,T_0]}
  \|\eta(x/R)v'(x,t)\|_{H^1(\RRR^4)}\\
  \leq& \sup_{t\in[-T_0, T_0]}\|\eta(x/R)v'(x,t)\|_{L^2(\RRR^4)}+
  \|\nabla\eta(x/R)v'(x,t)\|_{L^2(\RRR^4)}+\|\eta(x/R)\nabla v'(x,t)\|_{L^2(\RRR^4)}\\
  \leq & 2\|v'(x,t)\|_{H^1(\RRR^4)}\leq 2\|\phi'(t)\|_{H^5(\RRR^4)}.
\end{align*}
By Littlewood-Paley theorem and Sobolev embedding, we obtain that 
\begin{align*}
  &\|V_{R,N}\|_{Z([-T_0N^{-2},T_0N^{-2}])}=\sup_{J\subset[-T_0N^{-2},T_0N^{-2}]}
  (\sum_{M dyadic} M^2\|P_M V_{R,N}\|^4_{L^4(J\times\TTT^4)})^{\frac{1}{4}}\\
  =&\sup_{J\subset[-T_0N^{-2},T_0N^{-2}]}
  \|(\sum_M (\langle 1-\DD \rangle^{\frac{1}{4}}
  P_M V_{R,N})^4)^{\frac{1}{4}}\|_{L^4(J\times \TTT^4)}\\
  \leq&  \sup_{J\subset[-T_0N^{-2},T_0N^{-2}]}
  \|(\sum_M (\|\langle 1-\DD \rangle^{\frac{1}{4}}
  P_M V_{R,N})^2)^{\frac{1}{2}}\|_{L^4(J\times \TTT^4)}\\
  \lesssim& \sup_{J\subset[-T_0N^{-2},T_0N^{-2}]}
  \|\langle 1-\DD \rangle^{\frac{1}{4}} V_{R,N}\|_{L^4(J\times\TTT^4)}\\
  \leq&  \sup_{J\subset[-T_0N^{-2},T_0N^{-2}]}
  \|\langle 1-\DD \rangle^{\frac{1}{2}} V_{R,N}
  \|_{L^4_t(J)L^\frac{8}{3}_x(\TTT^4)}\\
  \lesssim& \||v'_{R,N}| +|\nabla_{\RRR^4} v'_{R,N}|\|_{L^4_tL_x^{\frac{8}{3}}
  ([-T_0N^{-2},T_0N^{-2}]\times\RRR^4)}\\
  \lesssim& \|v'_R\|_{L^4_tL_x^{\frac{8}{3}}([-T_0,T_0]\times\RRR^4)}+
  \||\nabla_{\RRR^4}v'_{R}|\|_{L^4_tL_x^{\frac{8}{3}}([-T_0,T_0]\times\RRR^4)}.
\end{align*}

Since
 $\|v'_R\|_{L^4_tL_x^{\frac{8}{3}}([-T_0,T_0]\times\RRR^4)}+
  \||\nabla_{\RRR^4}v'_{R}|\|_{L^4_tL_x^{\frac{8}{3}}([-T_0,T_0]\times\RRR^4)}\lesssim \sup_{t} \|v'(t)\|_{H^5}$, by (\ref{eq:scatteringsmooth}) we obtain $\|V_{R,N}\|_{Z([-T_0N^{-2},T_0N^{-2}])}\lesssim_{\|\phi'\|_{H^5(\RRR^4)}} 1.$

\case{2}{$\|f_N -V_{R,N}(0)\|_{H^1(\TTT^4)}\leq \ee$.}

By H\"{o}lder inequality, we obtain that

\begin{align*}
  \|f_N - V_{R,N}(0)\|_{H^1(\TTT^4)}&\leq \|\phi_N(\Psi^{-1}(y))-
  \phi'_{R,N}(\Psi^{-1}(y))\|_{\dot{H}^1(\TTT^4)}\\
  &\leq \|\phi_N -\phi'_{R,N}\|_{\dot{H}^1(\RRR^4)}
  =\|Q_N\phi -\phi'_R\|_{\dot{H}^1(\RRR^4)}\\
  &= \|\eta(\frac{x}{N^{\frac{1}{2}}})\phi(x) -
  \eta(\frac{x}{N^{\frac{1}{2}}})\phi'(x)\|_{\dot{H}^1(\RRR^4)}\\
  &\leq \|\eta(\frac{x}{N^{\frac{1}{2}}})\phi(x)-\phi(x)\|_{\dot{H}^1(\RRR^4)}
  +\|\phi -\phi'\|_{\dot{H}^1(\RRR^4)}
  +\|\eta(\frac{x}{N^{\frac{1}{2}}})\phi'(x)-\phi'(x)\|_{\dot{H}^1(\RRR^4)}.
\end{align*}
With $N\geq 10 R$, and $R> R_0$, $R_0$ large enough, we have that
\[
\|f_N -V_{R,N}(0)\|_{H^1(\TTT^4)}\leq 2\ee_1.
\]

\case{3}{$\|E_{R,N}\|_{N([-T_0N^{-2},T_0N^{-2}])}\leq \ee$.}
Next, by Proposition \ref{prop:dual} and scaling invariance, we obtain that
\begin{align*}
\|E_{R,N}\|_{N([-T_0N^{-2},T_0N^{-2}])}
&= \|\int_{0}^t e^{i(t-s)\DD} E_{R,N}(s)\, ds\|_{X^1([-T_0N^{-2},T_0N^{-2}])}\\
&\leq \sup_{\|u_0\|_{Y^{-1}}=1}
\left| \int_{\TTT^4\times[-T_0N^{-2},T_0N^{-2}]} \overline{u_0}\cdot E_{R,N}
\, dxdt\right|\\
&\leq \sup_{\|u_0\|_{Y^{-1}}=1} \||\nabla|^{-1}u_0\|_{L_t^\infty L_x^2}
\||\nabla| E_{R,N}\|_{L_t^1L_x^2}\\
&\leq \sup_{\|u_0\|_{Y^{-1}}=1}
\|u_0\|_{Y^{-1}} \||\nabla|E_{R,N}\|_{L^1_tL^2_x([-T_0N^{-2},T_0N^{-2}]
\times\TTT^4)}\\
&\leq \|\nabla_{\RRR^4}\,e_{R,N}\|_{L^1_tL^2_x([-T_0N^{-2},T_0N^{-2}]\times\RRR^4)}\\
&=\|\nabla_{\RRR^4}\, e_R\|_{L^1_tL^2_x([-T_0, T_0]\times\RRR^4)}.
\end{align*}

\begin{align*}
|\nabla_{\RRR^4}\, e_R(x,t)|&=| \nabla_{\RRR^4}(\mu (\eta(x/R)-\eta^3(x/R)))
v'(x,t)|v'(x,t)|^2\\
+& R^{-2} v'(x,t)(\DD_{\RRR^4} \eta)(\frac{x}{R}) + 2R^{-1}(\sum_{j=1}^4
\partial_j v'(x,t)\partial_j \eta(x/R)|\\
&\leq |\nabla_{\RRR^4}(\eta(\frac{x}{R})-\eta(\frac{x}{R})^3)v'(x,t)|v'(x,t)|^2|
+3|(\eta(\frac{x}{R})-\eta(\frac{x}{R})^3)\nabla_{\RRR^4}v'(x,t)|v'(x,t)|^2|\\
+& R^{-3}|v'(x,t)\nabla_{\RRR^4} \DD_{\RRR^4} \eta(\frac{x}{R})|
+ R^{-2}|\nabla_{\RRR^4}v'(x,t)(\DD_{\RRR^4}\eta)(\frac{x}{R})|
+ R^{-1}|\DD_{\RRR^4}v'(x,t)\nabla_{\RRR^4}\eta(\frac{x}{R})|\\
 &\lesssim_{\|\phi'\|_{H^5(\RRR^4)}} \mathds{1}_{[R, 2R]}(|x|)
 \left( |v'(x,t)| +|\nabla_{\RRR^4} v'(x,t)| \right)+\frac{1}{R}
\left( |\langle \nabla_{\RRR^4}\rangle^2 v'(x,t)| \right).
\end{align*}

Since $\|\nabla_{\RRR^4}^2 v'(x,t)\|_{L^\infty_x}\lesssim_{\|\phi'\|_{H^5}} 1$,
$\|\nabla_{\RRR^4} v'(x,t)\|_{L^\infty_x}\lesssim_{\|\phi'\|_{H^5}} 1$, and
$\|v'(x,t)\|_{L^\infty_x}\lesssim_{\|\phi'\|_{H^5}} 1$ (by Sobolev embedding), we obtain that

\begin{align*}
  &\|\nabla_{\RRR^4}\, e_R\|_{L^1_tL^2_x([-T_0,T_0]\times\RRR^4)}
  =\int_{-T_0}^{T_0} (\int_{\RRR^4} |\nabla_{\RRR^4} \, e_R|\, dx)^{\frac{1}{2}}dt\\
  &\leq\int_{-T_0}^{T_0} \left(\int_{\RRR^4}\mathds{1}_{[R, 2R]}(|x|)
  (|v'(x,t)|^2 + |\nabla_{\RRR^4} v'(x,t)|^2)\,dx +\frac{1}{R^2}\int_{\RRR^4}
  |\langle \nabla_{\RRR^4}\rangle^2 v'(x,t)|^2\,dx
  \right)^\frac{1}{2}dt\\
  &\lesssim_{\|\phi'\|_{H^5}}
  2T_0 \left( \int_{\RRR^4}\mathds{1}_{[R,2R]}(|x|)
  \langle \nabla_{\RRR^4}\rangle^2 v'(x,t)|^2\,dx \right)^{\frac{1}{2}}+\frac{1}{R}
  \to 0,\text{ as } R\to \infty.\\
\end{align*}

So we can obtain that \begin{equation*}
  \|\nabla E_{R,N}\|_{L^1_tL^2_x([-T_0N^{-2}, T_0N^{-2}]\times\TTT^4)}<\ee_1,
\end{equation*}
where $R>R_0$, and $R_0$ large enough.

By checking all there conditions above, we have the desired result.

\end{proof}

Next, we prove a extinction lemma as Ionescu and Pausader \cite{ionescu2012energy}
did in their paper about energy critical NLS in $\TTT^3$. The extinction lemma
is the essential part why we prove the GWP result in $\TTT^4$.
\begin{lem}[Extinction Lemma]\label{lem:extinction}
Let $\phi\in\dot{H}^1(\RRR^4)$, and define $f_N$ as in (\ref{def:euclideanProfile}).
For any $\ee >0$, there exist $T = T(\phi, \ee)$ and $N_0(\phi, \ee)$ such that
for all $N\geq N_0$,
there holds that
\[
\|e^{it\DD} f_N\|_{Z([TN^{-2}, T^{-1}])}\lesssim \ee.
\]
\end{lem}

\begin{proof}
  For $M\geq 1$, we define
  \[
  K_M(x,t) = \sum_{\xi\in\ZZZ^4} e^{-i[t|\xi|^2 +x\cdot\xi]\eta(\xi/M)}
  =e^{it\DD}P_{\leq M}\delta_0.
  \]
  We know from [Lemma 3.18, Bourgain\cite{bourgain1993fourier}] that
  $K_M$ satisfies
  \begin{equation}\label{4.9}
    |K_M(x,t)|\lesssim \prod_{i=1}^4 \left(\frac{M}{\sqrt{q_i}(1+M|t/(\lambda_i)-a_i/q_i|^{\frac{1}{2}})}
    \right),
  \end{equation}
  if $a_i$ and $q_i$ satisfying $\frac{t}{\lambda_i} =\frac{a_i}{q_i}+\beta_i$, where
  $q_i\in \{1,\cdots, M\}$, $a_i\in \ZZZ$, $(a_i, q_i)=1$ and $|\beta_i|\leq (Mq_i)^{-1}$ for all $i=1,2,3,4.$

  From this, we conclude that for any $1\leq S\leq M$,
  \begin{equation}\label{4.11}
    \|K_M(x,t)\|_{L_{x,t}^\infty(\TTT^4\times[SM^{-2}, S^{-1}])} \lesssim
    S^{-2}M^4.
  \end{equation}

  This follows directly from (\ref{4.9}) and Dirichlet's approximation lemma
  which is stated as following:
  \textit{For any real numbers $\alpha$, and any positive integer $N$, there exists
  integers $p$ and $q$ such $1\leq q\leq N$ and $|q\alpha -p|<\frac{1}{N}$}.
  
  Assume that $|t|\leq \frac{1}{S}$. $\frac{t}{\lambda_i} =\frac{a_i}{q_i}+\beta_i$ and
  $|\beta_i|\leq \frac{1}{Mq_i}\leq \frac{1}{M}\leq \frac{1}{S}$. So we obtain that
  \[\left|\frac{a_i}{q_i}\right|\leq \frac{2}{S}\quad \implies \quad q_i\geq \frac{a_i}{2}.\]

  Therefore either $q_i\geq \frac{1}{2}S\quad (a_i\geq 1)$ or $a_i=0$ for each $i$.
  If {$q_i\geq \frac{1}{2}S\quad (a_i\geq 1)$}, then
  \[\frac{M}{\sqrt{q_i}(1+M|t/(\lambda_i)-a_i/q_i|^{\frac{1}{2}})}
    \lesssim  \frac{M}{\sqrt{q}}\lesssim S^{-\frac{1}{2}}M.
  \]

  If {$a_i=0$}, then
  \[\frac{M}{\sqrt{q_i}(1+M|t/(\lambda_i)-a_i/q_i|^{\frac{1}{2}})}
    \lesssim \frac{M}{\sqrt{q_i}+M|t|^{1/2}}\lesssim |t|^{-\frac{1}{2}}
  \leq S^{-\frac{1}{2}}M.
  \]
So we have that $|K_M(x,t)|\lesssim S^{-2}M^4$.

By the definition as in (\ref{def:euclideanProfile}), to prove the extinction lemma,
we may assume that $\phi\in C_0^{\infty}(\RRR^4)$, we claim that
\begin{equation}\label{4.12}
  \begin{split}
    \|f_N\|_{L^1(\TTT^4)} &\lesssim_{\phi} N^{-3}\\
    \|P_K f_N\|_{L^2(\TTT^4)}&\lesssim_{\phi} \left( 1+ \frac{K}{N}\right)^{-10}N^{-1}.
  \end{split}
\end{equation}

Let's consider the bound of $\|f_N\|_{L^1(\TTT^4)}$:
\begin{align*}
  \|f_N\|_{L^1(\TTT^4)}& = \|\phi_N(\Psi^{-1}(y))\|_{L^1(\TTT^4)}
  = \|\phi_N(x)\|_{L^1(\RRR^4)}\\
  &=\int_{\RRR^4} |N(Q_N\phi)(Nx)|\,dx \\
  &= \frac{1}{N^3}\int_{\RRR^4}|Q_N \phi|(x)\, dx =
  \frac{1}{N^3} \int_{\RRR^4} |\eta(\frac{x}{N^{1/2}}) \phi(x)|\,dx\\
  &\leq \frac{1}{N^3} \|\phi\|_{L^1(\RRR^4)}.
\end{align*}

Let's consider the bound of $\|P_K f_N\|_{L^2(\TTT^4)}$:
\begin{align*}
  \|P_K f_N\|_{L^2(\TTT^4)} &=\|P_K \phi_N\|_{L^2(\RRR^4)}\\
  &=\|P_K N(Q_N\phi)(Nx)\|_{L^2(\RRR^4)}\\
  &=\|N(P_{\frac{K}{N}}Q_N\phi)(Nx)\|_{L^2(\RRR^4)}=
  \frac{1}{N}\|P_{\frac{K}{N}}Q_N\phi\|_{L^2(\RRR^4)}\\
  &=\frac{1}{N}\|P_{\frac{K}{N}}(\eta(\frac{x}{N^{\frac{1}{2}}})\phi(x))\|_{L^2(\RRR^4)}\\
  &\leq \frac{1}{N}\left(1+\frac{K}{N} \right)^{-10}\|\eta(\frac{x}{N^{\frac{1}{2}}})
  \phi(x)\|_{H^{10}(\RRR^4)}\\
  &\leq \frac{1}{N}\left(1+\frac{K}{N} \right)^{-10} \|\phi\|_{H^{10}}.
\end{align*}

By Proposition \ref{prop:strichartz}, for $p>3$ we obtain that
\begin{equation}\label{4.121}
\|e^{it\DD}P_K f_N\|_{L^p_{x,t}(\TTT^4\times[-1,1])}\leq K^{2-\frac{6}{p}}
\left(1+\frac{K}{N} \right)^{-10} N^{-1} \|\phi\|_{H^{10}}.
\end{equation}

Then let's estimate $\|e^{it\DD}f_N\|_{Z([TN^{-2},T^{-1}])}$. We know that
\begin{equation*}
  \|e^{it\DD}f_N\|_{Z([TN^{-2},T^{-1}])} = \sup_{J\subset[TN^{-2},T^{-1}]}
  \left( \sum_K K^2 \|P_K e^{it\DD}f_N\|^4_{L^4(J\times\TTT^4)}\right)^\frac{1}{4}
\end{equation*}

To estimate it, we decompose the sum above into three part:
  $$\left(\sum_{K\leq NT^{-\frac{1}{100}}} +
  \sum_{K\geq NT^{\frac{1}{100}}}+
  \sum_{NT^{-\frac{1}{100}}\leq K\leq NT^{\frac{1}{100}}}\right)
  K^2 \|P_K e^{it\DD}f_N\|^4_{L^4([TN^{-2}, T^{-1}]\times\TTT^4)}$$.

\case{1}{$K\leq NT^{-\frac{1}{100}}$:}
By (\ref{4.121}), we obtain that
\begin{align*}
  &\sum_{K\leq NT^{-\frac{1}{100}}} K^2
  \|P_K e^{it\DD}f_N\|^4_{L^4([TN^{-2}, T^{-1}]\times\TTT^4)}\\
  \leq & \sum_{K\leq NT^{-\frac{1}{100}}} K^4 \left(1+\frac{K}{N}\right)^{-40}
  N^{-4}\|\phi\|_{H^{10}}\\
  \lesssim_{\phi}& (NT^{-\frac{1}{100}})^4N^{-4} = T^{-\frac{1}{25}}.
\end{align*}

\case{2}{$K\geq NT^{\frac{1}{100}}$:}
By (\ref{4.121}), we obtain that
\begin{align*}
  &\sum_{K\geq NT^{\frac{1}{100}}} K^2
  \|P_K e^{it\DD}f_N\|^4_{L^4([TN^{-2}, T^{-1}]\times\TTT^4)}\\
  \leq & \sum_{K\geq NT^{\frac{1}{100}}} K^4 \left(1+\frac{K}{N}\right)^{-40}
  N^{-4}\|\phi\|_{H^{10}}\\
  \leq &\sum_{K\geq NT^{\frac{1}{100}}} K^{-36}N^{36}\|\phi\|_{H^10}\\
  \lesssim_{\phi} & T^{-\frac{4}{100}}.
\end{align*}

\case{3}{$NT^{-\frac{1}{100}}\leq K\leq NT^{\frac{1}{100}}$:}
Let's consider $K\in [NT^{-\frac{1}{100}}, NT^{\frac{1}{100}}]$ and set $M\sim \max{(K, N)}$
and $S\sim T$.

\begin{equation}\label{4.13}
  \begin{split}
    \|e^{it\DD} P_K f_N\|_{L^{\infty}_{x,t}(\TTT^4\times[TN^{-2}, T^{-1}])}
    &=\|K_M * f_N\|_{L^{\infty}_{x,t}(\TTT^4\times[TN^{-2}, T^{-1}])}\\
    &\leq \|K_M\|_{L^\infty_{x,t}(\TTT^4\times[TN^{-2}, T^{-1}])
    \|f_N\|_{L^1_{x,t}(\TTT^4)}}\\
    &\lesssim_{\phi} T^{-2}K^4N^{-3}\leq T^{-2+\frac{1}{25}}N.
  \end{split}
\end{equation}

\begin{equation}\label{4.14}
  \begin{split}
    \|e^{it\DD} P_N f_N\|_{L_{x,t}^{3}(\TTT^4\times[TN^{-2}, T^{-1}])}&\lesssim_\phi K^{\ee}\left(
    1+\frac{K}{N}\right)^{-10}N^{-1}\\
    &\leq N^{-1+\ee}T^{-\frac{\ee}{100}}.
  \end{split}
\end{equation}

Interpolating (\ref{4.13}) with (\ref{4.14}), we have that
\begin{equation}
  \|e^{it\DD}P_K f_N\|_{L^4_{x,t}([TN^{-2},T^{-1}])}\lesssim_{\phi}
  \left( N^{-1+\ee} T^{\frac{\ee}{100}} \right)^{\frac{3}{4}}
  \left(T^{-2+\frac{1}{25}} N \right)^{\frac{1}{4}}
  \leq
  N^{-\frac{1}{4}}T^{-\frac{1}{100}}.
\end{equation}

Summing  $K^2 \|P_K e^{it\DD}f_N\|^4_{L^4([TN^{-2}, T^{-1}]\times\TTT^4)}$ over  $K$, we obtain that
\begin{align*}
  \sum_{NT^{-\frac{1}{100}}\leq K\leq NT^{\frac{1}{100}}}
  K^2 \|P_K e^{it\DD}f_N\|^4_{L^4([TN^{-2}, T^{-1}]\times\TTT^4)}
  &  \leq \sum_{NT^{-\frac{1}{100}}\leq K\leq NT^{\frac{1}{100}}} K^2
  (N^{-\frac{1}{4}}T^{-\frac{1}{100}})^4\\
  &\leq \sum_{NT^{-\frac{1}{100}}\leq K\leq NT^{\frac{1}{100}}}
  K^2N^{-2} T^{-\frac{1}{25}}\\
  &\leq T^{-\frac{1}{50}}.
\end{align*}

Summarizing all three cases by setting $T$ large enough, we hold the estimate.
\end{proof}

Let's now consider $f\in L^2(\TTT^4)$, $t_0\in \RRR$ and $x_0\in \TTT^4$,
\begin{align*}
  (\pi_{x_0} f)(x)&:= f(x-x_0),\\
  (\Pi_{t_0,x_0}) f(x)&:= (\pi_{x_0}e^{-it_0\DD} f)(x).\\
\end{align*}

As in (\ref{def:euclideanProfile}), given $\phi\in\dot{H}^1(\RRR^4)$
and $N\geq 1$, we define
\[
T_N \phi (x):= N\tilde{\phi} (N\Psi^{-1}(x)),\text{ where }
\tilde{\phi} (y):= \eta(y/N^{\frac{1}{2}})\phi(y)
\]
and claim that
$T_N : \dot{H}^1(\RRR^4)\to H^1(\RRR^4)$ is a linear operator with
$\|T_N \phi\|_{H^1(\TTT^4)}\lesssim \|\phi\|_{\dot{H}^1(\RRR^4)}$.

\begin{remark}
  To show $\|T_N \phi\|_{H^1(\TTT^4)}\lesssim \|\phi\|_{\dot{H}^1(\RRR)}$.
  \begin{align*}
    \|T_N \phi(x)\|_{H^1(\TTT^4)}& \lesssim \|T_N \phi(x)\|_{\dot{H}^1(\TTT^4)}\\
    &= \|\nabla(N\eta(N^{\frac{1}{2}})\phi(Ny))\|_{L^2(\RRR^4)}\\
    &\leq \|N^{\frac{3}{2}} (\nabla \eta) (N^{\frac{1}{2}}y)\phi(Ny)\|_{L^2(\RRR^4)}
    +\|N^2 \eta(N^{\frac{1}{2}}) \nabla \phi(Ny)\|_{L^2(\RRR^4)}\\
    &\leq \|\mathds{1}_{[0, N^{\frac{1}{2}}]}\|_{L^4(\RRR^4)}
    \|N^{\frac{3}{2}} |(\nabla \eta) (N^{\frac{1}{2}}y)\phi(Ny)|\|_{L^4_{\RRR^4}}
    +\|\phi\|_{\dot{H}^1(\RRR^4)}\\
    &\leq \|N\phi(Ny)\|_{\dot{H}^1(\RRR^4)}+\|\phi\|_{\dot{H}^1(\RRR^4)}\\
    &\leq 2\|\phi\|_{\dot{H}^1(\RRR^4)}.
  \end{align*}
\end{remark}

\begin{definition}
  Let $\widetilde{\mathcal{F}_e}$ denote the set of renormalized Euclidean frames
\begin{align*}
  \widetilde{\mathcal{F}_e}:=&
  \{
  (N_k, t_k, x_k)_{k\geq 1}: N_k \in [1,\infty),\ t_k\to 0,\ x_k\in\TTT^4,\
  N_k\to\infty \\
  &\text{ and either } t_k=0 \text{ for any } k\geq 1 \text{ or }
  \lim_{k\to \infty} N_k^2|t_k| =\infty \}.
\end{align*}
\end{definition}

\begin{prop}[Euclidean profiles]\label{prop:4.4}
  Assume that $\mathcal{O} = (N_k, t_k, x_k)_k \in \widetilde{\mathcal{F}_e}$ and $\mu \in \{-1, 0, 1\}$.
  $\phi\in \dot{H}^1(\RRR^4)$.
  Suppose 
  \[\|\phi\|_{\dot{H}^1(\RRR^4)}< +\infty, \qquad \text{when }\mu \in \{0, 1\};\]
 or under the assumption that if $v$ is a solution of (\ref{eq:ivpfocusing}) with $v(0) =\phi$ then $v$  satisfies
  \[\sup_{t\in \text{lifespan of }v}\|v(t)\|_{\dot{H}^1(\RRR^4)}<\|W\|_{\dot{H}^1(\RRR^4)}, , \qquad \text{when }\mu \in \{-1\}.\] Then
  \begin{enumerate}
    \item there exists $\tau =\tau(\phi)$ such that for $k$ large enough
    (depending only on $\phi$ and $\mathcal{O}$) there is a nonlinear solution
    $U_k\in X^1(-\tau, \tau)$ of the initial value problem (\ref{eq:NLS}) with initial data $U_k(0) = \Pi_{t_k, 0} (T_{N_k} \phi)$ and
\begin{equation}\label{4.4.1}
      \|U_k\|_{X^1(-\tau, \tau)}\lesssim_{E_{\RRR^4}(\phi),\, \|\phi\|_{\dot{H}^1(\RRR^4)}} \, 1;
    \end{equation}
    \item there exists an Euclidean solution $u\in C(\RRR: \dot{H}^1(\RRR^4))$
    of
    \[
    (i\partial_t +\DD_{\RRR^4})u = \mu u|u|^2
    \]
    with scattering data $\phi^{\pm \infty}$ defined as Theorem \ref{thm:GWPinR} such
    that the following holds, up to a subsequence: for any $\ee >0$, there exists
    $T(\phi, \ee)$ such that for all $T\geq T(\phi, \ee)$, there exists
    $R(\phi, \ee, T)$ such that for all $R\geq R(\phi, \ee, T)$,
    there holds that
    \begin{equation}\label{4.4.2}
      \|U_k -\widetilde{u_k}\|_{X^1(\{|t-t_k|\leq TN_k^{-2}\}\cap\{|t|<T^{-1}\})}
      \leq \ee,
    \end{equation}\label{eq:tildeu}
    for $k$ large enough, where
    \begin{equation}
    (\pi_{-x_k} \widetilde{u_k})(x, t) = N_k
     \eta(N_k \Psi^{-1}(x)/R) u(N_k\Psi^{-1}(x), N_k^2(t-t_k)).
    \end{equation}
    In addition, up to a subsequence,
    \begin{equation}\label{4.4.3}
      \|U_k(t) -\Pi_{t_k-t,x_k}
      T_{N_k}\phi^{\pm\infty}\|_{X^1(\{\pm (t-t_k)\geq \pm TN_k^{-2}\}\cap\{|t|<T^{-1}\})},
      \leq \ee,
    \end{equation}
    for $k$ large enough (depending on $\phi$, $\ee$, $T$, and $R$).
  \end{enumerate}
\end{prop}
\begin{proof}
  By the statement, it is equivalent to prove the case when $x_k = 0$.

  \noindent Part (1):
First, for $k$ large enough, we can make
\[
\|\phi -\eta(\frac{x}{N^{\frac{1}{2}}}) \phi\|_{\dot{H}^1(\RRR^4)}\leq \ee_1.
\]
For each $N_k$, we choose $T_{0,N_k} = \tau N_k^2$($T_{0, N_k}$ is the coefficient
in Lemma \ref{thm:4.2}).
For each $T_{0, N_k}$, we make $R_k$ large enough to make Theorem \ref{thm:4.2} work.
(Note: in this case, $R_k$ determined by $T_{0, N_k}$ as in the proof of
Theorem \ref{thm:4.2}.)

\noindent Part(2):
Let's consider first case in Euclidean frame: $t_k =0$ for all $k$.
(\ref{4.4.1}) is directly from Theorem \ref{thm:4.2}, by choosing $k$, $R$ for any
fixed $T$ large enough.

To prove (\ref{4.4.2}), we need to choose $T(\phi, \dd)$ large enough, to make
sure
\[
\|\nabla_{\RRR^4} u\|_{L^3_{x,t}(\RRR^4\times\{|t|>T(\phi,\dd)\})}\leq \dd.
\]

By Theorem \ref{thm:GWPinR}, we obtain that
\[
\|u(\pm T(\phi, \dd))-e^{\pm iT(\phi,\dd)\DD}\phi^{\pm\infty}\|_{\dot{H}^1(\RRR^4)}
\leq \dd,
\]
which implies
\begin{equation}
\|U_{N_k}(\pm TN_k^{-2}) -\Pi_{-\pm T, x_k} T_{N_k} \phi^{\pm\infty}\|_{H^1(\TTT^4)}
\leq \dd.
\end{equation}

By Proposition \ref{prop:ZinX} and Proposition \ref{EstimateFreeSolution}, we have
\begin{equation}\label{4.4.5}
\|e^{it\DD}\left( U_{N_k}(\pm TN_k^{-2}) -\Pi_{-\pm T, x_k} T_{N_k}
 \phi^{\pm\infty}\right)\|_{X^1(|t|<T^{-1})}.
\end{equation}

By Proposition \ref{prop:lwp}, we obtain that
\begin{equation}\label{4.4.6}
\|U_{N_k} -e^{it\DD}U_{N_k}(\pm TN_k^{-2})\|_{X^1}\leq \dd,
\end{equation}
and combining (\ref{4.4.5}) and (\ref{4.4.6}), we have
\[
\|U_{N_k} -\Pi_{-t, x_k} T_{N_k}\phi^{\pm\infty} \|_{X^1(\{\pm t\geq \pm TN_k^{-2}\}\cap\{|t|<T^{-1}\})}
\leq \ee.
\]
when we choose $\dd$ small enough.

The second case: $N_k^2|t_k|\to \infty$.
\begin{align*}
  U_k(0) &= \Pi_{t_k, 0} (T_{N_k} \phi)\\
  &= e^{-it_k\DD} \left(N_k^{\frac{1}{2}}\widetilde{\phi}(N_k\Psi^{-1}(x))\right)\\
  &= e^{-it_k\DD} \left( N_k^{\frac{1}{2}} \eta(N_k^{\frac{1}{2}}\Psi^{-1}(x))
  \phi(N_k\Psi^{-1}(x))\right).
\end{align*}

By existence of wave operator of NLS, we know the following initial value problem
is global well-posed, so there exists $v$ satisfying:
\begin{equation}
\begin{cases}
  (i\partial_t+\DD_{\RRR^4}) v = \mu v|v|^2,\\
  \lim_{t\to -\infty} \|v(t)- e^{it\DD} \phi\|_{\dot{H}^1(\RRR^4)} = 0.
\end{cases}
\end{equation}

We set
\[\widetilde{v_k}(t) = N_k^{\frac{1}{2}} \eta(N_k\Psi^{-1}(x)/R)
v(N_k\Psi^{-1}(x), N_k^2 t),
\]
so we have $\widetilde{v_k}(-t_k) = N_k^{\frac{1}{2}} \eta(N_k\Psi^{-1}(x)/R)
v(N_k\Psi^{-1}(x), -N_k^2 t_k).$

For $k$ and $R$ large enough,
\begin{align*}
  &\|\widetilde{v_k}(-t_k) - e^{-it_k\DD} N_k^{\frac{1}{2}}
  \eta(N_k^{\frac{1}{2}}\Psi^{-1}(x))\phi(N_k\Psi^{-1}(x))\|_{\dot{H}^1(\TTT^4)}\\
  \leq& \|\eta(\frac{x}{N_k^{\frac{1}{2}}})v(x, -N_k^2t_k)
  -e^{it_kN_k^2\DD}\eta(\frac{x}{N_k^{\frac{1}{2}}})\phi(x)
  \|_{\dot{H}^1(\RRR^4)}\\
  \leq &\ee.
\end{align*}

So $V_k(t)$ solves initial value problem (\ref{eq:IVP}) in $\TTT^4$, with
initial data $V_k(0) = \widetilde{V_k}(0)$, which implies $V_k(t)$ exists
in $[-\dd, \dd]$, and $\|V_k(t)-\widetilde{V_k}(t)\|_{X^1([-\dd, dd])}\lesssim \ee$.

By the stability property (Proposition \ref{prop:stability}),
$
\|U_k - V_k\|_{X^1([-\dd, \dd])}\to 0, \text{ as } k\to \infty.
$
\end{proof}

The following corollary (Corollary \ref{lem:decomposition1}) decompose the nonlinear Euclidean profiles $U_k$ defined in the Proposition \ref{prop:4.4}. This corollary follows closely in a part of the proof of Lemma 6.2 in \cite{ionescu2012energy}. I state it here as a corollary because the almost orthogonality of nonlinear profiles (Lemma \ref{prop:almostorth}) heavily relies on this decomposition lemma (Corollary \ref{lem:decomposition1}).
\begin{cor}[Decomposition of the nonlinear Euclidean profiles $U_k$]\label{lem:decomposition1}
  Consider $U_k$ is the nonlinear Euclidean profiles w.r.p.t. $\mathcal{O} = (N_k, t_k, x_k)_k \in \widetilde{\mathcal{F}_e}$ defined above. For any
  $\theta>0$, there exist $T_{\theta}^0$ sufficiently large such that
  for all $T_{\theta}\geq T_{\theta}^0$ and $R_{\theta}$
  sufficiently large such that for all $k$ large enough
  (depending on $R_{\theta}$) we can decompose $U_k$ as following:
  \[  \mathds{1}_{(-T_{\theta}^{-1},T_{\theta}^{-1})}(t)
    U_k = \w_k^{\theta,-\infty}+\w_k^{\theta,+\infty}
    +\w_k^{\theta}+\rho^{\theta}_k,\]
    and $\w_k^{\theta,\pm\infty}$, $\w_k^{\theta}$,
    and $\rho^{\theta}_k$ satisfy the following
    conditions:
  \begin{equation}\label{eq:decomposition1}
  \begin{split}
    \|\w_k^{\theta,\pm\infty}\|_{Z'(-T_{\theta}^{-1},T_{\theta}^{-1})}
    +\|\rho^{ \theta}_k\|_{X^1(-T_{\theta}^{-1},T_{\theta}^{-1})}\leq \theta,\\
    \|\w_k^{\theta,\pm\infty}\|_{X^1(-T_{\theta}^{-1},T_{\theta}^{-1})}+
    \|\w_k^{\theta}\|_{X^1(-T_{\theta}^{-1},T_{\theta}^{-1})}\lesssim 1,\\
    \w_k^{\theta,\pm\infty} = P_{\leq R_{\theta}N_{k}}
    \w_k^{ \theta,\pm\infty}\\
    |\nabla_x^m \w_k^{\theta}|+(N_{k})^{-2}\mathds{1}_{S_k^{ \theta}}
    |\partial_t \nabla_x^m \w_k^{ \theta}|\leq R_{\theta}
    (N_{k})^{|m|+1}\mathds{1}_{S_k^{\theta}},\ 0\leq |m| \leq 10,
  \end{split}
  \end{equation}
  where
  \[
  S_k^{\theta} :=\{ (x,t)\in \TTT^4\times (-T_{\theta}, T_{\theta})
  : |t-t_{k}|< T_{\theta}(N_{k})^{-2},\ |x-x_{k}|\leq R_{\theta}(N_{k})^{-1}\}.
  \]
\end{cor}

\begin{proof}
By Proposition \ref{prop:4.4}, there exists
  $T(\phi, \frac{\theta}{4})$, such that for all
  $T\geq T(\phi, \frac{\theta}{4})$, there exists $R(\phi,\frac{\theta}{4},T)$
  such that for all $R\geq R(\phi,\frac{\theta}{2},T)$, there holds that
  \begin{equation*}
    \|U_k - \widetilde{u_k}\|_{X^1(\{
      |t-t_k|\leq T(N_{k})^{-2}\}\cap\{
      |t|<T^{-1}
    \})}\leq \frac{\theta}{2},
  \end{equation*}
  for $k$ large enough, where
  \begin{equation*}
    \left(\pi_{-x_k} \widetilde{u_k}\right)(x,t)
    =N_{k} \eta(N_{k} \Psi^{-1}(x)/R)
    u(N_{k} \Psi^{-1}(x), N_{k}^2(t-t_k)),
  \end{equation*}
  where $u$ is a solution of (\ref{eq:NLS}) with scattering data $\phi^{\pm\infty}$.

  In addition, up to subsequence,
  \begin{equation*}
    \|U_k - \Pi_{t_k-t, x_k} T_{N_{k}} \phi^{\pm\infty}
    \|_{X^1(\{\pm(t-t_k)\geq T(N_k)^{-2}\}\cap\{
    |t|\leq T^{-1}
    \})}\leq \frac{\theta}{4},
  \end{equation*}
  for $k$ large enough (depending on $\phi$, $\theta$, $T$, and $R$).

  Choose a sufficiently large $T_{\theta}> T(\phi,\frac{\theta}{4})$
  based on the extinction lemma(Lemma \ref{lem:extinction}), such that
  \begin{equation*}
    \|e^{it\DD} \Pi_{t_k, x_k} T_{N_k} \phi^{\pm\infty}
    \|_{Z(T_{\theta}(N_k)^{-2}, T_{\theta}^{-1})}\leq \frac{\theta}{4}
  \end{equation*}
  when k large enough.

  And then we choose $R_{\theta} = R(\phi, \frac{\theta}{2}, T_{\theta})$.

  Denote:
  \begin{enumerate}
    \item $\w_k^{\theta,\pm\infty} := \mathds{1}_{\{
    \pm(t-t_k)\geq T_{\theta}(N_k)^{-2},|t|\leq T_{\theta}^{-1}\}}
    \left(\Pi_{t_k-t, x_k} T_{N_k}\phi^{\theta,\pm\infty}
    \right)$,

    {where }
    \[
    \|\phi^{\theta,\pm\infty}\|_{\dot{H}^1(\RRR^4)}\lesssim 1,
      \ \phi^{\theta,\pm\infty} = P_{\leq R_{\theta}} (\phi^{\theta,\pm\infty}),
      \]
      which implies   $\w_k^{ \theta,\pm\infty} = P_{\leq R_{\theta}N_{\theta}}
        \w_k^{ \theta,\pm\infty}$.
    \item $\w_k^{\theta} := \widetilde{u_k}\cdot \mathds{1}_{S_k^{ \theta}},$
    where $
    S_k^{\theta} :=\{ (x,t)\in \TTT^4\times (-T_{\theta}, T_{\theta})
    : |t-t_k|< T_{\theta}(N_k)^{-2},\ |x-x_k|\leq R_{\theta}(N_k)^{-1}\}$.

   \noindent By the stability property (Proposition \ref{prop:stability}) and Theorem \ref{thm:4.2}, we can adjust
    $\w_k^{\theta}$ and $\w_k^{\theta,\pm\infty}$, with an acceptable
    error, to make
    \[|\nabla_x^m \w_k^{\theta}|+(N_k)^{-2}\mathds{S_k^{\alpha, \theta}}
    |\partial_t \nabla_x^m \w_k^{ \theta}|\leq R_{\theta}
    (N_k)^{|m|+1}\mathds{1}_{S_k^{ \theta}},\ 0\leq |m| \leq 10.\]

    \item $\rho_k:= \mathds{1}_{(-T_{\theta}^{-1},T_{\theta}^{-1})}(t)
    U_k^{\alpha} -\w^{\theta}_k -\w^{\theta,+\infty}-\w^{\theta,-\infty}$.
  \end{enumerate}
  By (\ref{4.4.2}) and (\ref{4.4.3}), we obtain that
  \[
  \|\rho_k^{\theta}\|_{X^1(\{|t|<T_{\theta}^{-1}\})}\leq \frac{\theta}{2}.
  \]
  and then we have
  \begin{align*}
    \|\w_k^{\theta,\pm\infty}\|_{Z'(-T_{\theta}^{-1},T_{\theta}^{-1})}
    +\|\rho^{\theta}_k\|_{X^1(-T_{\theta}^{-1},T_{\theta}^{-1})}\leq \theta,\\
    \|\w_k^{\theta,\pm\infty}\|_{X^1(-T_{\theta}^{-1},T_{\theta}^{-1})}+
    \|\w_k^{\theta}\|_{X^1(-T_{\theta}^{-1},T_{\theta}^{-1})}\lesssim 1.
  \end{align*}
\end{proof}

\section{Profile decomposition}

In this section, we construct the profile decomposition on $\TTT^4$ for linear
Schr\"{o}dinger equations. The arguments and propositions in this section is almost identical to those in the Section 5 of \cite{ionescu2012global2}, except for one more lemma (Lemma \ref{prop:almostorth}) about almost  orthogonality of nonlinear profiles which is useful in the focusing case.

As in the previous section, given $f\in L^2(\RRR^4)$, $t_0\in\RRR$, and
$x_0\in\TTT^4$, we define:
\begin{align*}
  (\Pi_{t_0, x_0} )f(x)&:= (e^{-it_0\DD} f)(x-x_0)\\
  T_N \phi (x) &:= N\widetilde{\phi} (N\Psi^{-1}(x)),
\end{align*}
where $\widetilde{\phi}(y):= \eta(\frac{y}{N^{\frac{1}{2}}})\phi(y).$

Observe that $T_N : \dot{H}^1(\RRR^4) \to H^1(\TTT^4)$ is a linear operator with
$\|T_N \phi\|_{H^1(\TTT^4)}\lesssim \|\phi\|_{\dot{H}^1(\RRR^4)}$.


\begin{definition}[Euclidean frames]
  \begin{enumerate}
    \item We define a Euclidean frame to be a sequence $\mathcal{F}_e =(N_k, t_k, x_k)_k$
    with $N_k\geq 1$, $N_k\to +\infty$, $t_k\in\RRR$, $t_k\to 0$, $x_k\in\TTT^4$.
    We say that two frames, $(N_k, t_k, x_k)_k$ and $(M_k, s_k, y_k)_k$ are
    orthogonal if
    \[
    \lim_{k\to+\infty} \left( \ln \left|\frac{N_k}{M_k}\right|+ N_k^2
    \left|t_k-s_k\right| + N_k\left|x_k-y_k\right|\right) =\infty.
    \]
    Two frames that are not orthogonal are called equivalent.
    \item If $\mathcal{O} =(N_k, t_k, x_k)_k$ is a Euclidean frame and if
    $\phi\in\dot{H}^1(\RRR^4)$, we define the Euclidean profile associated to
    $(\phi, \mathcal{O})$ as the sequence $\widetilde{\phi}_{\mathcal{O}_k}$:
    \[
    \widetilde{\phi}_{\mathcal{O}_k} := \Pi_{t_k, x_k} (T_{N_k}\phi).\]
  \end{enumerate}
\end{definition}

\begin{prop}[Equivalence of frames \cite{ionescu2012global2}]\label{prop:equivalenceFrames}

\noindent (1) If $\mathcal{O}$ and $\mathcal{O}'$ are equivalent Euclidean frames, then there
exists an isometry $T: \dot{H}^1(\RRR^4) \to \dot{H}^1(\RRR^4)$ such that
for any profile $\widetilde{\phi}_{\mathcal{O}'_k}$, up to a subsequence there holds that
\[
\limsup_{k\to\infty} \|\widetilde{T\phi}_{\mathcal{O}_k} - \widetilde{\phi}_{\mathcal{O}'_k}\|_{H^1(\TTT^4)} = 0.
\]

(2) If $\mathcal{O}$ and $\mathcal{O}'$ are orthogonal Euclidean frames and $\widetilde{\phi}_{\mathcal{O}_k}$, $\widetilde{\phi}_{\mathcal{O}'_k}$ are corresponding profiles, then, up to a subsequence:
\begin{align}\label{eq:almostorth1}
\lim_{k\to\infty} \langle \widetilde{\phi}_{\mathcal{O}_k},  \widetilde{\phi}_{\mathcal{O}'_k}\rangle_{H^1\times H^1(\TTT^4)} = 0;\\\label{eq:almostorth2}
\lim_{k\to\infty} \langle \widetilde{|\phi}_{\mathcal{O}_k}|^2,  |\widetilde{\phi}_{\mathcal{O}'_k}|^2\rangle_{L^2\times L^2(\TTT^4)} = 0.
\end{align}
\end{prop}

\begin{lem}[Refined Strichartz inequality]
  Let $f\in H^1(\TTT^4)$ and $I\subset [0,1]$. Then 
  \[
  \|e^{it\Delta} f\|_{Z(I)} \lesssim (\|f\|_{H_x^1(\TTT^4)})^{\frac{5}{6}} \sup_{N\in 2^{\ZZZ}} (N^{-1} \|P_N e^{it\Delta}f\|_{L^\infty_{t,x}(I\times\TTT^4)})^{\frac{1}{6}}.
  \]
\end{lem}

\begin{proof}
By the definition of $Z$-norm,
\[
\|e^{it\Delta} f\|_{Z(I)} = \left( \sum_N N^2 \|P_N e^{it\Delta} f\|^4_{L^4_{t,x}}\right)^{\frac{1}{4}}
= \left\|N^{\frac{1}{2}} \|P_Ne^{it\Delta}f\|_{L^4_{t,x}} \right\|_{l^4_N}.
\]
By H\"{o}lder inequality and Proposition \ref{prop:strichartz}, we have that
\begin{align*}
    &\left\|N^{\frac{1}{2}} \|P_Ne^{it\Delta}f\|_{L^4_{t,x}} \right\|_{l^4_N}\\
    \lesssim & \left\|
    \left(N^{\frac{4}{5}}\|P_Ne^{it\Delta}f\|_{L^{\frac{10}{3}_{t,x}}}\right)^{\frac{5}{6}} \left(N^{-1}
    \|P_N e^{it\Delta} f\|_{L^\infty_{t,x}}\right)^{\frac{1}{6}}\right\|_{l^4_N}\\
    \lesssim & \sup_{N}\left(N^{-1}
    \|P_N e^{it\Delta} f\|_{L^\infty_{t,x}}\right)^{\frac{1}{6}} 
    \left( \sum_N N^{\frac{8}{3}}\|P_N e^{it\Delta} f\|^{\frac{10}{3}}_{L^{\frac{10}{3}}_{t,x}}
    \right)^{\frac{5}{6}}\\
    \lesssim & \sup_{N}\left(N^{-1}
    \|P_N e^{it\Delta} f\|_{L^\infty_{t,x}}\right)^{\frac{1}{6}} \left(\|f\|_{H^1_x(\TTT^4)}\right)^{\frac{5}{6}}.
    \end{align*}
\end{proof}

\begin{remark}
  The refined Strichartz estimate actually give us a fact:
  \[
  \sup_{N\in 2^{\ZZZ}} (N^{-1} \|P_N e^{it\Delta}f\|_{L^\infty_{t,x}})\gtrsim
  \frac{\|e^{it\Delta}f\|^6_{Z(I)}}{\|f\|^5_{H^1(\TTT^4)}},
  \]
  the linear solutions with non-trivial space-time norm must concentrate on at least one
  frequency annulus and around some points in space-time space.
\end{remark}


\begin{prop}[Profile decompositions \cite{ionescu2012global2}]\label{prop:ProfileDecomposition}
  Consider $\{ f_k\}_{k}$ a sequence of functions in $H^1(\TTT^4)$ and $0<A<\infty$ satisfying
  \[
  \limsup_{k\to +\infty} \|f_k\|_{H^1(\TTT^4)}\leq A
  \]
  and a sequence of intervals $I_k=(-T_k, T^k)$ such that $|I_k|\to 0$ as $k\to\infty$.
  Up to passing to a subsequence, assume that $f_k\rightharpoonup g\in H^1(\TTT^4).$
  There exists $J^*\in \{0, 1, ...\}\cup \{\infty\}$, and a sequence of profile $\widetilde{\psi}^\alpha_{k}:= \widetilde{\psi}^\alpha_{\mathcal{O}_k^\alpha}$
  associated to pairwise orthogonal Euclidean frames $\mathcal{O}^\alpha$ and $\psi^\alpha \in {H}^1(\RRR^4)$ such that
  extracting a subsequence, for every $0 \leq J\leq J^*$, we have
  \begin{equation}\label{5.1}
    f_k = g + \sum_{1\leq \alpha \leq J} \widetilde{\psi}^\alpha_{k}
    + R_k^J
  \end{equation}

  where $R^J_k$ is small in the sense that

  \begin{equation}\label{5.2}
    \limsup_{J\to J^*}\limsup_{k\to\infty}\|e^{it\DD}R_k^J\|_{Z(I_k)} = 0.
  \end{equation}

  Besides, we also have the following orthogonality relations:
\begin{equation}\label{5.3}
  \begin{split}
    \|f_k\|_{L^2}^2 = \|g\|_{L^2}^2 + \|R_k^J\|^2_{L^2}+o_k(1).\\
    \|\nabla f_k\|_{L^2}^2=\|\nabla g\|_{L^2}^2 +
    \sum_{\alpha\leq J}\|\nabla_{\RRR^4} \psi^\alpha\|^2_{L^2(\RRR^4)} +\|\nabla R_k^J\|^2_{L^2}
    +o_k(1).\\
    \lim_{J\to J^*}\limsup_{k\to\infty}
    \left| \|f_k\|^4_{L^4} - \|g\|_{L^4}^4 -\sum_{\alpha\leq J}
    \|\widetilde{\psi^\alpha_k}\|_{L^4}^4 \right| = 0.
  \end{split}
\end{equation}
\end{prop}
\begin{remark}
  $g$ and $\widetilde{\psi}^\alpha_{k}$ for all $\alpha$ are
  called profiles. In addition, we call $g$ is Scale-1-profile,
  and $\widetilde{\psi}^\alpha_{k}$ are called Euclidean profiles.
\end{remark}

\begin{remark}[Almost orthogonality of the energy]\label{rmk:EnergyDecoupling}
By (\ref{def:euclideanProfile}), we have that $\|\widetilde{\psi}^\alpha_{k}\|_{L^2(\TTT^4)} \leq \frac{1}{N_k}\|\psi^\alpha\|_{L^2(\RRR^4)}\to 0$ as $k\to \infty$ and $\|\widetilde{\psi}^\alpha_{k}\|^2_{\dot{H}^1(\TTT^4)} = \frac{1}{N}\|\nabla\eta(\frac{\cdot}{N^{\frac{1}{2}}})\psi^\alpha\|^2_{L^2(\RRR^4)} + \|\eta(\frac{\cdot}{N^{\frac{1}{2}}})\psi^\alpha\|^2_{\dot{H}^1(\RRR^4)}$. Then above  and (\ref{5.3}), we know that
\[
\lim_{J\to J^*}\lim_{k\to\infty} \left(  \sum_{1\leq \alpha \leq J}
  E(\widetilde{\psi}^\alpha_k)
  + E(R_k^J) + E(g)- E(f_k)\right) = 0.
\]
\end{remark}

\begin{lem}[Almost orthogonality of nonlinear profiles]\label{prop:almostorth}
Define $U_k^\alpha$, $U_k^\beta$ as the maximal life-span $I_k$  solutions of (\ref{eq:NLS}) with
  initial data $U_k^\alpha(0) = \widetilde{\psi}^\alpha_{\mathcal{O}^\alpha_k}$, $U_k^\beta(0) = \widetilde{\psi}^\beta_{\mathcal{O}^\beta_k}$, where $\mathcal{O}^\alpha$ and $\mathcal{O}^\beta$ are orthogonal.
  And define $G$ to be the solution of the maximal lifespan $I_0$ of (\ref{eq:NLS}) with initial
  data $G(0) =g$.  And $0\in I_k$ and $\lim_{k\to \infty} |I_k| = 0$. 
   Then 
  \begin{equation}
  \lim_{k\to \infty} \sup_{t\in I_k}\, \langle U^\alpha_k(t), U^\beta_k(t) \rangle_{\dot{H}^1\times \dot{H}^1} = 0, \quad \lim_{k\to \infty} \sup_{t\in I_k\cap I_0} \langle U^\alpha_k(t), G(t) \rangle_{\dot{H}^1\times \dot{H}^1} = 0.
  \end{equation}
\end{lem}
\begin{proof}
Set $U^0_k(0) = g$ and $U^0_k = G$ for all $k$, such that $U^0_k$ can be considered as a nonlinear profile with a trivial frame $\mathcal{O} =(1, 0, 0)_k$. 

For any $\theta > 0$, by the decomposition of the nonlinear profiles $U^\alpha$ and $U^\beta$ (Corollary \ref{lem:decomposition1}), there exist $T_{\theta, \alpha}$, $R_{\theta, \alpha}$, $T_{\theta, \beta}$, $R_{\theta, \beta}$ sufficiently large
\begin{align*} 
    U_k^\alpha &= \w_k^{\alpha, \theta,-\infty}+\w_k^{\alpha, \theta,+\infty}
    +\w_k^{\alpha, \theta}+\rho^{\alpha, \theta}_k,\\
    U_k^\beta &= \w_k^{\beta, \theta,-\infty}+\w_k^{\beta, \theta,+\infty}
    +\w_k^{\beta, \theta}+\rho^{\beta, \theta}_k.
\end{align*}
For $U^0_k$,   set $U^0_k := \w_k^{\alpha,\theta,-\infty}+\w_k^{\alpha,\theta,+\infty}
  +\w_k^{\alpha,\theta}+\rho^{\alpha, \theta}_k$ where $\rho_{k}^{0, \theta} =
  \w_k^{0,\theta} = 0$ and $\w^{0,\theta,+\infty}=\w^{0,\theta,-\infty}=\frac{1}{2}G$.
  And by taking $T_{\theta,0}$ large, it is easy to make $\|G\|_{Z'(-T_{\theta,0},T_{\theta,0})}\leq \theta$. So
  $\langle U^\alpha_k(t), G(t) \rangle_{\dot{H}^1\times \dot{H}^1}$  can be considered as a special case of $\langle U^\alpha_k(t), U^\beta_k(t) \rangle_{\dot{H}^1\times \dot{H}^1}$ when $\beta = 0$.

Since $\rho^{\alpha,\theta}_k$, $\rho^{\beta,\theta}_k$ are the small term with the $X^1$-norm  less then $\theta$, for any fixed $t\in I_k$, it will suffice to consider the following three terms:
\begin{enumerate}
\item $\langle\w_k^{\alpha, \theta,\pm\infty} , \w_k^{\beta, \theta,\pm\infty} \rangle_{\dot{H}^1\times \dot{H}^1}$;
\item $\langle\w_k^{\alpha, \theta,\pm\infty} , \w_k^{\beta, \theta} \rangle_{\dot{H}^1\times \dot{H}^1}$;
\item $\langle\w_k^{\alpha, \theta} , \w_k^{\beta, \theta} \rangle_{\dot{H}^1\times \dot{H}^1}$.
\end{enumerate}
\case{(1)}{ $\langle\w_k^{\alpha, \theta,\pm\infty} , \w_k^{\beta, \theta,\pm\infty} \rangle_{\dot{H}^1\times \dot{H}^1}$.}
By the constructions of $\w_k^{\alpha, \theta,\pm\infty}, \w_k^{\beta, \theta,\pm\infty}$ in the proof of Lemma \ref{lem:decomposition1}, we obtain that 
\begin{align}
\w_k^{\alpha, \theta,\pm\infty} := \mathds{1}_{\{
    \pm(t-t_k^{\alpha})\geq T_{\alpha, \theta}(N^{\alpha}_k)^{-2},|t|\leq T_{\alpha, \theta}^{-1}\}}
    \left(\Pi_{t_k^\alpha-t, x_k^\alpha} T_{N^\alpha_k}\phi^{\alpha, \theta,\pm\infty}\right),\\
   \w_k^{\alpha, \theta,\pm\infty} := \mathds{1}_{\{ \pm(t-t_k^{\beta})\geq T_{\beta, \theta}(N^{\beta}_k)^{-2},|t|\leq T_{\beta, \theta}^{-1}\}}
    \left(\Pi_{t_k^\beta-t, x_k^\beta} T_{N^\beta_k}\phi^{\beta, \theta,\pm\infty}\right).
    \end{align}
For any fixed $t\in I_k$, we obtain that \[\langle\w_k^{\alpha, \theta,\pm\infty}(t) , \w_k^{\beta, \theta,\pm\infty}(t) \rangle_{\dot{H}^1\times \dot{H}^1} = \langle \phi^{\alpha, \theta,\pm\infty}_{\mathcal{O}^\alpha_k}, \phi^{\beta, \theta,\pm\infty}_{\mathcal{O}^\beta_k} \rangle_{\dot{H}^1\times \dot{H}^1}.\]
By (\ref{eq:almostorth1}) of Proposition \ref{prop:equivalenceFrames}, we obtain that 
\[\lim_{k\to\infty}\sup_t \langle\w_k^{\alpha, \theta,\pm\infty}(t) , \w_k^{\beta, \theta,\pm\infty}(t) \rangle_{\dot{H}^1\times \dot{H}^1} = 0.\]

\case{(2)}{$\langle\w_k^{\alpha, \theta,\pm\infty} , \w_k^{\beta, \theta} \rangle_{\dot{H}^1\times \dot{H}^1}$.} 
By the constructions of $\w_k^{\alpha, \theta,\pm\infty}, \w_k^{\beta, \theta,\pm\infty}$ in the proof of Lemma \ref{lem:decomposition1}, we obtain that 
\[
\w_k^{\beta, \theta} := \widetilde{u_k}^\beta\cdot \mathds{1}_{S_k^{ \beta, \theta}},
\]
    where $
    S_k^{\beta, \theta} :=\{ (x,t)\in \TTT^4\times (-T_{\beta,\theta}, T_{\beta, \theta})
    : |t-t_k^\beta|< T_{\beta, \theta}(N^\beta_k)^{-2},\ |x-x^\beta_k|\leq R_{\beta, \theta}(N^\beta_k)^{-1}\}$ and $\widetilde{u_k}^\beta$ is defined in (\ref{eq:tildeu}).
Following a similar proof of the \textbf{Case 4} in the proof of (\ref{eq:6.8}) in Lemma \ref{lem:6.2}, we have that $\lim_{k\to \infty}\sup_t\langle\w_k^{\alpha, \theta,\pm\infty} , \w_k^{\beta, \theta} \rangle_{\dot{H}^1\times \dot{H}^1} = 0$.
    
\case{(3)}{$\langle\w_k^{\alpha, \theta} , \w_k^{\beta, \theta} \rangle_{\dot{H}^1\times \dot{H}^1}$.} 
For $\ee > 0$ small.

If $N^\alpha_{k}/N^\beta{k} + N^\beta_{k}/N^\alpha_{k} \leq \ee^{-1000}$ and
$k$ is large enough then $S^{\alpha, \theta}_{k}\cap S^{\beta, \theta}_{k} = \emptyset$. (By the definition of
orthogonality of frames,  $N^\alpha_{k}/N^\beta_{k} + N^\beta_{k}/N^\alpha_{k} \leq \ee^{-1000}$ implies
$(N^\alpha_{k})^2|t^\alpha_{k}-t^\beta_{k}|\to \infty$ or $N^\alpha_{k}|x^\alpha_{k}-x^\beta_{k}|\to\infty$, so
$S^{\alpha, \theta}_{k}\cap S^{\beta, \theta}_{k} = \emptyset$.)
In this case, $\w^{\alpha,\theta}_{k}\, \w^{\beta,\theta}_{k}\ \equiv 0$.

If $N^\alpha_{k}/N^\beta_{k} \geq \ee^{-1000}/2$.
Denote that
\[
\w^{\alpha,\theta}_{k} \w^{\beta,\theta}_{k} = \w^{\alpha,\theta}_{k} \widetilde{ \w}^{\beta,\theta}_{k}:=
\w^{\alpha,\theta}_{k} \cdot (\w^{\beta,\theta}_{k} \mathds{1}_{(t^\alpha_{k}-T_{\alpha,\theta} (N^{\alpha}_{k})^{-2}, t^\alpha_{k}+T_{\alpha,\theta} (N^{\alpha}_{k})^{-2})}(t)).
\]
By $\ee^{10}N_k^{\alpha}>>\ee^{-10}N_k^{\beta}$ and the \textit\textbf{Claim $\dagger$} in the proof of Lemma \ref{lem:7.2}, we obtain that
\begin{align*}
  \langle\w^{\alpha,\theta}_{k}, \w^{\beta,\theta}_{k}\rangle_{\dot{H}^1\times\dot{H}^1}& \leq \langle P_{\leq \ee^{10}N_k^{\alpha}}\w^{\alpha,\theta}_{k},
  \w^{\beta,\theta}_{k}\rangle_{\dot{H}^1\times\dot{H}^1}
  +\langle P_{> \ee^{10}N_k^{\alpha}} \w^{\alpha,\theta}_{k}, P_{>\ee^{-10} N_k^{\beta}}\w^{\beta,\theta}_{k} \rangle_{\dot{H}^1\times\dot{H}^1}\\
  &+\langle P_{> \ee^{10}N_k^{\alpha}} \w^{\alpha,\theta}_{k}, \w^{\beta,\theta}_{k}\rangle_{\dot{H}^1\times\dot{H}^1} \\
  &\lesssim \ee.
\end{align*}

\end{proof}

\section{Proof of the main theorems}
It suffice to prove the solutions remain bounded in $Z$-norm on intervals of
length at most 1. To obtain this, we run the induction on the $E(u) + M(u)$ ($\mu = +1$) and $\|u\|_{L^\infty_t\dot{H}^1}$ ($\mu = -1$).

\begin{definition}
  Define
  \[
  \Lambda (L,\tau) = 
  \begin{cases}\sup_{\text{u is a solution}\atop \text{of } (\ref{eq:NLS})}\{ \|u\|_{Z(I)}: E(u) + M(u) \leq L, \ |I|\leq \tau\}\qquad \text{if } \mu = +1\\
  \sup_{\text{u is a solution}\atop \text{of } (\ref{eq:NLS})}\{ \|u\|_{Z(I)}:  \sup_{t\in I} \|u(t)\|^2_{\dot{H}^1(\TTT^4)}<L,\ |I|\leq \tau\} \qquad \text{if } \mu = -1
  \end{cases}
  \]
where $u$ is any strong solution of (\ref{eq:NLS}) with initial data $u_0$ in interval $I$ of length $|I|\leq \tau$.
\end{definition}

It is easy to see that $\Lambda$ is an increasing function of both $L$ and $\tau$,
and moreover, by the definition
 we have the sublinearity of $\Lambda$ in $\tau$:
$
\Lambda (L, \tau + \sigma) \leq \Lambda (L, \tau) + \Lambda (L, \sigma).
$
Hence we define
\[
\Lambda_0 (L) = \lim_{\tau\to 0} \Lambda (L,\tau),
\]
and for all $\tau$, we have that 
$
\Lambda(L, \tau)<+\infty  \Leftrightarrow \Lambda_0 (L)<+\infty.
$
Finally, we define
\[
E_{max} = \sup\{L: \Lambda_0 (L)<+\infty\}.
\]

\begin{thm}\label{thm:main2}
Consider $E_{max}$ defined above, if $\mu = +1$ (the defocusing case), then $E_{max} = +\infty$; if $\mu = -1$ (the focusing case), then $E_{max}\geq \|W\|^2_{\dot{H}^1(\RRR^4)}$.
\end{thm}

\begin{cor}
Suppose $u$ is a solution of (\ref{eq:NLS}) in some time interval with the initial data $u_0\in H^1(\TTT^4)$.
\begin{enumerate}
\item (the defocusing case) If $\mu = +1$ and $\|u_0\|_{H^1(\TTT^4)}<+\infty$, then $u$ is a global solution.
\item (the focusing case) If $\mu = -1$ and under the assumption that 
\begin{equation*}
\sup_{t}\|u(t)\|_{\dot{H}^1(\TTT^4)}<\|W\|_{\dot{H}^1(\RRR^4)},
\end{equation*}  then $u$ is a global solution.
\end{enumerate}
\end{cor}

\begin{proof}[Proof of Theorem \ref{thm:main2}]
  Suppose for contradiction that $E_{max}<+\infty$ (if $\mu = +1$), or $E_{max} < \|W\|_{\dot{H}^1(\RRR^4)}$ (if $\mu = -1$). By the definition of $E_{max}$, there exists
  a sequence of solutions $u_k$ such that
  \begin{equation}\label{6.2}
  \begin{cases}
    E(u_k)+M(u_k),\ (\text{if } \mu = +1)\\
    \sup_{t\in [-T_k, T^k]}\|u(t)\|_{\dot{H}^1(\TTT^4)},\ (\text{if } \mu = -1)
    \end{cases}\to E_{max},
    \quad \|u_k\|_{Z(-T_k, 0)},\ \|u_k\|_{Z(0,T^k)}\to +\infty.
  \end{equation}
  for some $T_k,\ T^k\to 0$ as $k\to +\infty$. For the simplicity of notations, set 
  \[
	L(\phi) :=   \begin{cases}
  E(\phi)+M(\phi),\ (\text{if } \mu = +1),\\
    \sup_{t\in [-T_k, T^k]}\|u_{\phi}(t)\|^2_{\dot{H}^1(\TTT^4)},\ (\text{if } \mu = -1),
    \end{cases}
  \]
where $u_{\phi}(t)$ is the solution of (\ref{eq:NLS}) with initial data $u_{\phi}(0) =\phi$.
  By the Proposition \ref{prop:ProfileDecomposition}, after extracting a subsequence, (\ref{6.2}) gives a sequence of profiles $\widetilde{\psi}^\alpha_k$, where $\alpha, k = 1, 2,\cdots$, and a decomposition
  \begin{equation*}
    u_k(0) = g + \sum_{1\leq\alpha\leq J}
    \widetilde{\psi}^\alpha_k + R^J_k.
  \end{equation*}
satisfying
  \begin{equation}\label{63}
  \limsup_{J\to\infty}\limsup_{k\to\infty} \|e^{it\DD}R_k^J\|_{Z(I_k)} =0.
  \end{equation}
 And moreover the almost orthogonality in the Proposition \ref{prop:ProfileDecomposition} and the almost orthogonality of nonlinear profiles (Lemma \ref{prop:almostorth}), we obtain that 
 \begin{equation}\label{eq:EnergyDecoupling}
 \begin{split}
 L(\alpha) := \lim_{k\to+\infty} L(\widetilde{\psi}^\alpha_{\mathcal{O}^\alpha_k})
  \in [0,E_{max}],\\
  \lim_{J\to J^*} \left(  \sum_{1\leq \alpha \leq J}
 L(\alpha)
  +\lim_{k\to\infty} L(R_k^J)\right) + L(g)= E_{max},
 \end{split}
 \end{equation}

  \case{1}{$g\neq 0$ and no any Euclidean profiles.}
  There is no any Euclidean profiles, and by Remark \ref{rmk:EnergySimSobolev}, $\|g\|_{H^1(\TTT^4)}\lesssim L(g)\leq E_{max}$. Then, by $I_k\to 0$ as $k\to \infty$, there exist, $\eta >0$,
  s.t. for $k$ large enough
  \[
  \|e^{it\DD} u_k(0)\|_{Z(-T_k, T^k)} \leq \|e^{it\DD} g\|_{Z(-\eta, \eta)}
  + \ee \leq \dd_0
  \]
  where $\dd_0$ is given by the local theory in Proposition \ref{prop:lwp}. In this case.
  we conclude that $\|u_k\|_{Z(-T_k, T^k)}\lesssim 2\dd_0$ which contradicts (\ref{6.2}).

  \case{2}{$g=0$ and only one Euclidean profile $\widetilde{\psi}^1_k$ such that $L(1) = E_{max}$.}
     By Remark \ref{rmk:EnergyDecoupling} and (\ref{eq:EnergyDecoupling}), we obtain that $L(\widetilde{\psi}^1_k) \leq E_{max}$ which implies $\|\psi\|_{\dot{H}^1(\RRR^4)}<\infty$ (if $\mu = +1$) or $\sup_{t}\|u_{\psi}\|_{\dot{H}^1(\RRR^4)}<\|W\|_{\dot{H}^1(\RRR^4)}$ (if $\mu = -1$). Denote $U^1_k$ is the solution of (\ref{eq:NLS}) with $U^1_k(0)= \widetilde{\psi}^1_k$.
   In this case, we use the part (1) of Proposition \ref{prop:4.4} and Remark \ref{rmk:EnergySimSobolev},
  Given some $\epsilon>0$, for $k$ large enough, we have that
  \begin{equation}\label{eq:case2}
  \|U^1_k\|_{X^1(-T_k, T^k)}\leq \|U^1_k\|_{X^1(-\dd, \dd)} \lesssim 1, \qquad \text{and}\qquad \|U^1_k(0) - u_k(0)\|_{H^1(\TTT^4)}\leq \epsilon.
  \end{equation}
  By (\ref{eq:case2}) and Proposition \ref{prop:stability}, we obtain that
  \[
  \|u_k\|_{Z(I_k)} \lesssim \|u_k\|_{X^1(I_k)} \lesssim 1,
  \]
  which contradicts (\ref{6.2}).

  \case{3}{At least two of all profiles are nonzero}
 By (\ref{eq:EnergyDecoupling}), $L(g)<E_{max}$ and $L(\alpha)<E_{max}$ for any $\alpha = 1, 2,\cdots$
  By almost orthogonality and relabeling the profiles, we can assume that for all $\alpha$,
  \[L(\alpha)\leq L(1)< E_{max} -\eta, \ L(g)<E_{max} -\eta, \text{ for some }\eta>0.\]

  Define $U_k^\alpha$ as the maximal life-span solution of (\ref{eq:NLS}) with
  initial data $U_k^\alpha(0) = \widetilde{\psi}^\alpha_k$
  and $G$ to be the maximal life-span solution  of (\ref{eq:NLS}) with initial
  data $G(0) =g$.

  By the definition of $\Lambda$ and the hypothesis $E_{max}<\infty$ (if $\mu = +1$) and $E_{max}< E_{W}$ (if $\mu = -1$), we have
  \[
  \|G\|_{Z(-1,1)} + \lim_{k\to\infty} \|U_k^\alpha\|_{Z(-1,1)} \leq
  2\Lambda(E_{max}-\eta/2, 2) \lesssim 1.
  \]
  By Proposition \ref{lem:cgwp}, it follows that for any $\alpha$ and any
  $k>k_0(\alpha)$ sufficient large,
  \[
  \|G\|_{X^1(-1,1)} +\|U^\alpha_k\|_{X^1(-1, 1)}\lesssim 1.
  \]

  For $J, k \geq 1$, we define
  \[
  U_{prof, k}^J:= G + \sum_{\alpha =1}^J U_k^\alpha = \sum_{\alpha = 0}^J U_k^\alpha.
  \]
  where we set that $U_k^0 := G$.

  \noindent\textit{\textbf{Claim} that there is a constant $Q$ such that
  \begin{equation}\label{6.6}
      \|U^J_{prof, k}\|^2_{X^1(-1,1)} +\sum_{\alpha=0}^J \|U_k^\alpha\|^2_{X^1(-1,1)}
      \leq Q^2,
  \end{equation}
  uniformly on $J$.
  }

  From (\ref{63}) we know that there are only finite many profiles such that
  $L(\alpha)\geq \frac{\dd_0}{2}$. We may assume that for all $\alpha\geq A$,
  $L(\alpha) \leq\dd_0$. Consider $U_k^\alpha$ for $k\geq A$, by small
  data GWP result (Proposition \ref{prop:smallDataGWP}), we have that
  \begin{align*}
    &\|U^J_{prof, k}\|_{X^1(-1,1)} = \|\sum_{0\leq\alpha\leq J}
    U^\alpha_k\|_{X^1(-1,1)} \\
    \leq &\sum_{0\leq \alpha\leq A} \|U^\alpha_k\|_{X^1(-1,1)}
    +\|\sum_{A\leq \alpha\leq J} (U^\alpha_k -e^{it\DD}U^\alpha_k(0))\|_{X^1(-1,1)}
    +\|e^{it\DD} \sum_{A\leq \alpha \leq J} U^\alpha_k(0)\|_{X^1(-1,1)}\\
    \lesssim& (A+1)+\sum_{A\leq\alpha\leq J} \|U^\alpha_k(0)\|^2_{H^1} +
    \|\sum_{A\leq\alpha\leq J} U_k^{\alpha}(0)\|_{H^1}\\
    \lesssim& (A+1) + \sum_{A\leq\alpha\leq J} L(\alpha) + E_{max}^{\frac{1}{2}}\\
    \lesssim& 1.
  \end{align*}
And also similarly, we have that
\begin{align*}
  \sum_{\alpha = 0}^{J} \|U_k^\alpha\|^2_{X^1(-1,1)} &=
  \sum_{\alpha =0}^{A-1} \|U_k^\alpha\|^2_{X^1(-1,1)} +
  \sum_{A\leq\alpha\leq J}\|U_k^\alpha\|^2_{X^1(-1,1)}  \\
  &\lesssim A + \sum_{A\leq\alpha\leq J} L(\alpha)\\
  &\lesssim 1.
\end{align*}

We denote that
\[
U^J_{app, k} = \sum_{0\leq \alpha\leq J} U^\alpha_k + e^{it\DD}R_k^J
\]
is a solution of the approximation equation (\ref{eq:aNLS}) with
the error term:
\begin{align*}
e &= (i\partial_t +\DD) U_{app, k}^J - F(U^J_{app, k})\\
& = \sum_{0\leq \alpha\leq J} F(U_k^\alpha) - F(\sum_{0\leq \alpha\leq J} U^\alpha_k
+e^{it\DD}R_k^J),
\end{align*}
where $F(u) = u|u|^2.$

From (\ref{6.6}) we know
$\|U_{app, k}^J\|_{X^1(-1,1)}\leq Q$

By Lemma \ref{lem:6.2} (proven later), we obtain that
\[
\limsup_{k\to\infty} \|e\|_{N(I_k)} \leq \ee/2, \text{ for } J\geq J_0(\ee).
\]
We use the stability proposition (Proposition \ref{prop:stability}) to conclude that
$u_k$ satisfies
\[
\|u_k\|_{X^1(I_k)} \lesssim \|U_{app, k}^J\|_{X^1(I_k)}\leq
\|U^J_{prof, k}\|_{X^1(-1,1)} \|e^{it\DD}R_k^J\|_{X^1(-1,1)} \lesssim 1.
\]
which contradicts (\ref{6.2}).
\end{proof}

\begin{lem}\label{lem:6.2}
  With the same notation, we obtain that
  \begin{equation}
    \limsup_{J\to\infty}\limsup_{k\to \infty}
    \|\sum_{0\leq \alpha\leq J} F(U_k^\alpha) -
    F(\sum_{0\leq \alpha\leq J} U^\alpha_k
    +e^{it\DD}R_k^J)\|_{N(I_k)} = 0.
  \end{equation}
\end{lem}

\section{Proof of Lemma \ref{lem:6.2}}
Consider 
\begin{align*}
  &\|\sum_{0\leq \alpha\leq J} F(U_k^\alpha) -
F(U_{prof, k}^J
+e^{it\DD}R_k^J)\|_{N(I_k)} \\
\leq&
\|F(U^J_{prof, k}+e^{it\DD}R_k^J))-F(U^J_{prof, k})\|_{N(I_k)}+
\|F(U^J_{prof, k}) -\sum_{0\leq \alpha\leq J} F(U_k^\alpha)\|_{N(I_k)}.
\end{align*}

It will suffice that we can prove
\begin{equation}\label{eq:6.8}
  \limsup_{J\to\infty}\limsup_{k\to \infty}
  \|F(U^J_{prof, k}+e^{it\DD}R_k^J))-F(U^J_{prof, k})\|_{N(I_k)} = 0,
\end{equation}
and
\begin{equation}\label{eq:6.9}
  \limsup_{J\to\infty}\limsup_{k\to \infty}
  \|F(U^J_{prof, k}) -\sum_{0\leq \alpha\leq J} F(U_k^\alpha)\|_{N(I_k)} = 0.
\end{equation}

Before prove (\ref{eq:6.8}) and (\ref{eq:6.9}), we need several lemma.
\begin{lem}[Decomposition of $U_k^\alpha$]\label{lem:decomposition}
  Consider $U^\alpha_k$ is the nonlinear profiles defined above. For any
  $\theta>0$, there exists $T_{\theta,\alpha}^0$ sufficiently large such that
  for all $T_{\theta, \alpha}\geq T_{\theta, \alpha}^0$ there is $R_{\theta, \alpha}$
  sufficiently large such that for all $k$ large enough
  (depending on $R_{\theta,\alpha}$) we can decompose $U^\alpha_k$ as following:
  \[  \mathds{1}_{(-T_{\theta, \alpha}^{-1},T_{\theta, \alpha}^{-1})}(t)
    U^\alpha_k = \w_k^{\alpha,\theta,-\infty}+\w_k^{\alpha,\theta,+\infty}
    +\w_k^{\alpha,\theta}+\rho^{\alpha, \theta}_k,\]
    and $\w_k^{\alpha,\theta,\pm\infty}$, $\w_k^{\alpha,\theta}$,
    and $\rho^{\alpha, \theta}_k$ satisfy the following
    conditions:
  \begin{equation}\label{eq:decomposition}
  \begin{split}
    \|\w_k^{\alpha,\theta,\pm\infty}\|_{Z'(-T_{\theta, \alpha}^{-1},T_{\theta, \alpha}^{-1})}
    +\|\rho^{\alpha, \theta}_k\|_{X^1(-T_{\theta, \alpha}^{-1},T_{\theta, \alpha}^{-1})}\leq \theta,\\
    \|\w_k^{\alpha,\theta,\pm\infty}\|_{X^1(-T_{\theta, \alpha}^{-1},T_{\theta, \alpha}^{-1})}+
    \|\w_k^{\alpha,\theta}\|_{X^1(-T_{\theta, \alpha}^{-1},T_{\theta, \alpha}^{-1})}\lesssim 1,\\
    \w_k^{\alpha, \theta,\pm\infty} = P_{\leq R_{\theta,\alpha}N_{k,\alpha}}
    \w_k^{\alpha, \theta,\pm\infty}\\
    |\nabla_x^m \w_k^{\alpha, \theta}|+(N_{k, \alpha})^{-2}\mathds{1}_{S_k^{\alpha, \theta}}
    |\partial_t \nabla_x^m \w_k^{\alpha, \theta}|\leq R_{\theta, \alpha}
    (N_{k,\alpha})^{|m|+1}\mathds{1}_{S_k^{\alpha, \theta}},\ 0\leq |m| \leq 10,
  \end{split}
  \end{equation}
  where
  \[
  S_k^{\alpha,\theta} :=\{ (x,t)\in \TTT^4\times (-T_{\theta,\alpha}, T_{\theta,\alpha})
  : |t-t_{k,\alpha}|< T_{\theta,\alpha}(N_{k,\alpha})^{-2},\ |x-x_{k,\alpha}|\leq R_{\theta, \alpha}(N_{k,\alpha})^{-1}\}.
  \]
\end{lem}

\begin{proof}
  First, if $\alpha = 0$. Set $G = U^0_k := \w_k^{\alpha,\theta,-\infty}+\w_k^{\alpha,\theta,+\infty}
  +\w_k^{\alpha,\theta}+\rho^{\alpha, \theta}_k$ where $\rho_{k}^{0, \theta} =
  \w_k^{0,\theta} = 0$ and $\w^{0,\theta,+\infty}=\w^{0,\theta,-\infty}=\frac{1}{2}G$.
  And by taking $T_{\theta,0}$ large, it is easy to make $\|G\|_{Z'(-T_{\theta,0},T_{\theta,0})}\leq \theta$.

  For a fixed $\alpha$ which is not $0$, by Proposition \ref{prop:4.4}, there exists
  $T(\phi^\alpha, \frac{\theta}{4})$, such that for all
  $T\geq T(\phi^\alpha, \frac{\theta}{4})$, there exists $R(\phi^\alpha,\frac{\theta}{4},T)$
  such that for all $R\geq R(\phi^\alpha,\frac{\theta}{2},T)$, there holds that
  \begin{equation*}
    \|U^\alpha_k - \widetilde{u_k^\alpha}\|_{X^1(\{
      |t-t_k^\alpha|\leq T(N_{k,\alpha})^{-2}\}\cap\{
      |t|<T^{-1}
    \})}\leq \frac{\theta}{2},
  \end{equation*}
  for $k$ large enough, where
  \begin{equation*}
    \left(\pi_{-x_k^\alpha} \widetilde{u_k^\alpha}\right)(x,t)
    =N_{k,\alpha} \eta(N_{k,\alpha} \Psi^{-1}(x)/R)
    u(N_{k,\alpha} \Psi^{-1}(x), N_{k,\alpha}^2(t-t_k^\alpha)),
  \end{equation*}
  where $u$ is a solution of (\ref{eq:NLS}) with scattering data $\phi^{\pm\infty}$.

  In addition, up to subsequence,
  \begin{equation*}
    \|U_k^\alpha - \Pi_{t_k^\alpha-t, x_k^\alpha} T_{N_{k,\alpha}} \phi^{\pm\infty,\alpha}
    \|_{X^1(\{\pm(t-t_k^\alpha)\geq T(N_k^\alpha)^{-2}\}\cap\{
    |t|\leq T^{-1}
    \})}\leq \frac{\theta}{4},
  \end{equation*}
  for $k$ large enough (depending on $\phi^\alpha$, $\theta$, $T$, and $R$).

  Choose a sufficiently large $T_{\theta,\alpha}> T(\phi^\alpha,\frac{\theta}{4})$
  based on the extinction lemma(Lemma \ref{lem:extinction}), such that
  \begin{equation*}
    \|e^{it\DD} \Pi_{t_k^\alpha, x_k^\alpha} T_{N_k^\alpha} \phi^{\pm\infty, \alpha}
    \|_{Z(T_{\theta,\alpha}(N_k^\alpha)^{-2}, T_{\theta,\alpha}^{-1})}\leq \frac{\theta}{4}
  \end{equation*}
  when k large enough.

  And then we choose $R_{\theta, \alpha} = R(\phi^{\alpha}, \frac{\theta}{2}, T_{\theta, \alpha})$.

  Denote:
  \begin{enumerate}
    \item $\w_k^{\alpha,\theta,\pm\infty} := \mathds{1}_{\{
    \pm(t-t_k^\alpha)\geq T_{\theta,\alpha}(N_k^\alpha)^{-2},|t|\leq T_{\theta,\alpha}^{-1}\}}
    \left(\Pi_{t_k^\alpha-t, x_k^\alpha} T_{N^\alpha_k}\phi^{\alpha,\theta,\pm\infty}
    \right)$,

    {where }
    \[
    \|\phi^{\alpha,\theta,\pm\infty}\|_{\dot{H}^1(\RRR^4)}\lesssim 1,
      \ \phi^{\alpha,\theta,\pm\infty} = P_{\leq R_{\theta,\alpha}} (\phi^{\alpha,\theta,\pm\infty}),
      \]
      which implies   $\w_k^{\alpha, \theta,\pm\infty} = P_{\leq R_{\theta,\alpha}N_{\theta,\alpha}}
        \w_k^{\alpha, \theta,\pm\infty}$.
    \item $\w_k^{\alpha,\theta} := \widetilde{u_k^\alpha}\cdot \mathds{1}_{S_k^{\alpha, \theta}},$
    where $
    S_k^{\alpha,\theta} :=\{ (x,t)\in \TTT^4\times (-T_{\theta,\alpha}, T_{\theta,\alpha})
    : |t-t_k^\alpha|< T_{\theta,\alpha}(N_k^\alpha)^{-2},\ |x-x_k|\leq R_{\theta, \alpha}(N_k^\alpha)^{-1}\}$.

   \noindent By the stability property (Proposition \ref{prop:stability}) and Theorem \ref{thm:4.2}, we can adjust
    $\w_k^{\alpha,\theta}$ and $\w_k^{\alpha,\theta,\pm\infty}$, with an acceptable
    error, to make
    \[|\nabla_x^m \w_k^{\alpha, \theta}|+(N_k^\alpha)^{-2}\mathds{S_k^{\alpha, \theta}}
    |\partial_t \nabla_x^m \w_k^{\alpha, \theta}|\leq R_{\theta, \alpha}
    (N_k^\alpha)^{|m|+1}\mathds{1}_{S_k^{\alpha, \theta}},\ 0\leq |m| \leq 10.\]

    \item $\rho_k^{\alpha}:= \mathds{1}_{(-T_{\theta,\alpha}^{-1},T_{\theta,\alpha}^{-1})}(t)
    U_k^{\alpha} -\w^{\alpha,\theta}_k -\w^{\alpha,\theta,+\infty}-\w^{\alpha,\theta,-\infty}$.
  \end{enumerate}
  By (\ref{4.4.2}) and (\ref{4.4.3}), we obtain that
  \[
  \|\rho_k^{\alpha,\theta}\|_{X^1(\{|t|<T_{\theta,\alpha}^{-1}\})}\leq \frac{\theta}{2}.
  \]
  and then we have
  \begin{align*}
    \|\w_k^{\alpha,\theta,\pm\infty}\|_{Z'(-T_{\theta, \alpha}^{-1},T_{\theta, \alpha}^{-1})}
    +\|\rho^{\alpha, \theta}_k\|_{X^1(-T_{\theta, \alpha}^{-1},T_{\theta, \alpha}^{-1})}\leq \theta,\\
    \|\w_k^{\alpha,\theta,\pm\infty}\|_{X^1(-T_{\theta, \alpha}^{-1},T_{\theta, \alpha}^{-1})}+
    \|\w_k^{\alpha,\theta}\|_{X^1(-T_{\theta, \alpha}^{-1},T_{\theta, \alpha}^{-1})}\lesssim 1.
  \end{align*}
\end{proof}

Denote that $\mathfrak{D}_{p,q}(a,b)$ stands for a $p+q$ - linear
expression with $p$ factors consisting of either $\overline{a}$ or $a$
and $q$ factors consisting of either $\overline{b}$ or $b$.

  \begin{lem}[a high-frequency linear solution does not interact significantly
    with a low-frequency profile]\label{lem:7.1}
    Assume that $B, N\geq 2$, and dyadic numbers, and assume that
  $\w :\TTT^4\times(-1, 1)\to \mathbb{C}$ is a function satisfying
  \[
  |\nabla^j \w|\leq N^{j+1} \mathds{1}_{|x|\leq N^{-1}, |t|\leq N^{-2}},\ j=0, 1.
  \]
  Then we hold that
  \[
  \|\mathfrak{D}_{2,1}(\w, e^{it\DD} P_{>BN} f)\|_{N(-1,1)}\lesssim
  (B^{-1/{200}}+N^{-1/{200}}) \|f\|_{H^1(\TTT^4)}.
  \]
\end{lem}

\begin{proof}
We may assume that $\|f\|_{H^1(\TTT^4)} = 1$ and $f = P_{>BN} f$.
By Proposition \ref{prop:dual}, we obtain that
\begin{align*}
  &\|\mathfrak{D}_{2,1}(\w, e^{it\DD} P_{>BN} f)\|_{N(I)}\\
  \leq & \|\mathfrak{D}_{2,1}(\w, e^{it\DD} P_{>BN} f)\|_{L^1((-1,1), H^1)}\\
  \lesssim& \|\mathfrak{D}_{2,1}(\w, \nabla e^{it\DD}f)\|_{L^1((-1,1), L^2)}
  +\|e^{it\DD}f\|_{L^\infty_tL^2_x} \|\w\|_{L^2_tL^\infty_x}\||\nabla \w|+|\w|\|_{L^2_tL^\infty_x}\\
  \lesssim & \|\mathfrak{D}_{2,1}(\w, \nabla e^{it\DD}f)\|_{L^1((-1,1),L^2)}+ B^{-1}.
\end{align*}
(It's easy to check that $\|\w\|_{L^2_tL_x^\infty} \leq \left( \int_{-N^{-2}}^{N^{-2}}
(N)^2\, dt\right)^{\frac{1}{2}} =$, $\|\nabla\w\|_{L^2_tL_x^\infty}\leq \left(
\int_{-N^{-2}}^{N^{-2}}N^4\,dt\right)^\frac{1}{2} = N$, and $\|P_{>BN} f\|_{L^2}\leq \frac{1}{BN} \|f\|_{H^1}$.)

Now we let $W(x,t):= N^4 \eta_{\RRR^4}(N\Psi^{-1}(x))\eta_{\RRR}(N^2 t)$,
\begin{align*}
  &\|\mathfrak{D}_{2,1} (\w, \nabla e^{it\DD}f)\|^2_{L^1((-1,1),L^2)}\\
  =&\left(\int_{-1}^1 \|\mathfrak{D}_{2,1}(\w, \nabla e^{it\DD}f)\| \, dt\right)^2\\
  \leq& \|\w\|^4_{L^4_tL_x^\infty}\|\frac{1}{N^2}W^{\frac{1}{2}}\nabla e^{it\DD} f\|_{L^2(\TTT^4\times[-1,1])}\\
  \lesssim&N^{-2} \|W^{\frac{1}{2}} \nabla e^{it\DD} f\|^2_{L^2(\TTT^4\times[-1,1])}\\
  \lesssim& \sum_{j=1}^4 \langle e^{it\DD}\partial_j f, We^{it\DD}\partial_j f\rangle_{L^2\times L^2}\, dt\\
  \lesssim& \sum_{j=1}^4
  \int_{-1}^1\langle \partial_j f, [\int_{-1}^1 e^{-it\DD}We^{it\DD}\,dt]\rangle_{L^2\times L^2}.
\end{align*}

It remains to prove that
\[
\|K\|_{L^2(\TTT^4)\to L^2(\TTT^4)} \lesssim N^2(B^{-\frac{1}{100}}+N^{-\frac{1}{100}}),
\]
where $K = P_{>BN} \int_{\RRR} e^{-it\DD}W e^{it\DD} P_{>BN}\,dt$.

We compute the Fourier coefficients of $K$ as follows:
\begin{align*}
  c_{p,q} &= \langle e^{ipx}, Ke^{iqx}\rangle\\
  & =\int_{\TTT^4} \overline{P_{>BN}e^{ipx}} \int_\RRR e^{-it\DD}W e^{it\DD} P_{>BN}\,dt
  dx\\
  &= (1-\eta_{\RRR^4})(p/{BN}) (1-\eta_{\RRR^4})(q/{BN})
  \int_{\TTT^4} \overline{e^{ipx}}\int_\RRR e^{-it\DD}W e^{it\DD} e^{iqx}\,dt
  dx\\
  &=(1-\eta_{\RRR^4})(p/{BN}) (1-\eta_{\RRR^4})(q/{BN})
  \int_{\TTT^4\times[-1,1]} \overline{e^{-it|p|^2 +ipx}}W(t,x) e^{-it|q|^2+iqx}\,dxdt\\
  & = (1-\eta_{\RRR^4})(p/{BN}) (1-\eta_{\RRR^4})(q/{BN}) C_{\mathcal{F}_{x,t}}(p-q, |q|^2-|p|^2).
\end{align*}

Hence, we obtain that
\[
|c_{p,q}|\lesssim N^{-2}\left( 1+ \frac{\left||p|^2-|q|^2\right|}{N^2}\right)^{-10}
\left( 1+\frac{|p-q|}{N} \right)^{-10}\mathds{1}_{\{|p|\geq BN\}}\mathds{1}_{\{|q|\geq BN\}}.
\]

Using Schur's lemma.
\[
\|K\|_{L^2(\TTT^4)\to L^2(\TTT^4)}\lesssim \sup_{p\in\ZZZ^4}
\sum_{q\in\ZZZ^4}|c_{p,q}|+\sup_{q\in\ZZZ^4}\sum_{p\in\ZZZ^4}|c_{p,q}|.
\]

It suffices to prove that
\begin{equation}\label{eq:7.3}
  N^{-4}\sup_{|p|\geq BN} \sum_{v\in\ZZZ^4}
  \left( 1+ \frac{\left||p|^2-|p+v|^2\right|}{N^2}\right)^{-10}
  \left( 1+\frac{|v|}{N} \right)^{-10}\lesssim B^{-\frac{1}{100}} + N^{-\frac{1}{100}}
\end{equation}

Consider (\ref{eq:7.3}) in the following 3 cases.

\case{1}{}
\begin{align*}
  &\sum_{|v|\geq NB^{\frac{1}{100}}}
  \left( 1+ \frac{\left||p|^2-|p+v|^2\right|}{N^2}\right)^{-10}
    \left( 1+\frac{|v|}{N} \right)^{-10}\\
\lesssim& \sum_{|v|\geq NB^{\frac{1}{100}}}\left( 1+\frac{|v|}{N} \right)^{-10}\\
\lesssim& \int_{|v|\geq N,\, v\in\RRR^4} \left( 1+\frac{|v|}{N} \right)^{-10}\,dv\\
\lesssim& \left( 1+\frac{NB^{\frac{1}{100}}}{N} \right)^{-6}\\
\lesssim& B^{-\frac{6}{100}}.
\end{align*}

\case{2}{}
\begin{align*}
&\sum_{|v|\leq NB^{\frac{1}{100}}\atop |v\cdot p|\geq N^2 B^{\frac{1}{10}}}
\left( 1+ \frac{\left||p|^2-|p+v|^2\right|}{N^2}\right)^{-10}
  \left( 1+\frac{|v|}{N} \right)^{-10}\\
  \lesssim &\sum_{|v|\leq NB^{\frac{1}{100}}\atop |v\cdot p|\geq N^2 B^{\frac{1}{10}}}
  \left(
  1 + \frac{2|v\cdot p|}{N^2}
  \right)^{-10}\\
  \lesssim & ( 1+ B^{\frac{1}{10}})^{-6} \lesssim B^{-\frac{6}{10}}.
\end{align*}

\case{3}{}
Denote $\hat{p} = \frac{p}{|p|}$
\begin{align*}
  &N^{-4}\sup_{|p|\geq BN} \sum_{|v|\leq NB^{\frac{1}{100}}\atop |p\cdot v|\leq N^2B^{\frac{1}{10}}}
  \left( 1+ \frac{\left||p|^2-|p+v|^2\right|}{N^2}\right)^{-10}
  \left( 1+\frac{|v|}{N} \right)^{-10}\\
  \leq & N^{-4}\sup_{|p|\geq BN} \sum_{|v|\leq NB^{\frac{1}{100}}\atop |\hat{p}\cdot v|\leq NB^{-\frac{9}{10}}}
  \left( 1+ \frac{\left||p|^2-|p+v|^2\right|}{N^2}\right)^{-10}
  \left( 1+\frac{|v|}{N} \right)^{-10}\\
  \leq & N^{-4}\sup_{|p|\geq BN} \sum_{|v|\leq NB^{\frac{1}{100}}\atop |\hat{p}\cdot v|\leq NB^{-\frac{9}{10}}} 1\\
  \leq & N^{-4}\sup_{|p|\geq BN} \# \{v: |v|\leq NB^{\frac{1}{100}},\ |\hat{p}\cdot v|\leq NB^{-\frac{9}{10}}\}\\
  =& N^{-4} (NB^{\frac{1}{100}})^3NB^{-\frac{9}{10}}\\
  \leq & B^{-\frac{87}{100}}.
\end{align*}
\end{proof}

\begin{lem}\label{lem:7.2}
  Assume that $\mathfrak{O}_\alpha = (N_{k,\alpha},t_{k,\alpha}, x_{k,\alpha})_k
  \in \mathcal{F}_e, \ \alpha\in\{1, 2\}$, are two orthogonal frames,
  $I\subseteq (-1, 1)$ is a fixed open interval, $0\in I$, and $T_1$, $T_2$,
  $R\in [1,\infty)$ are fixed numbers, $R\geq T_1 + T_2$. For $k$ large enough,
for $\alpha \in\{1, 2\}$
  \[
  |\nabla_x^m \w_k^{\alpha, \theta}|+(N_{k, \alpha})^{-2}\mathds{1}_{S_k^{\alpha, \theta}}
  |\partial_t \nabla_x^m \w_k^{\alpha, \theta}|\leq R_{\theta, \alpha}
  (N_k^\alpha)^{|m|+1}\mathds{1}_{S_k^{\alpha, \theta}},\ 0\leq |m| \leq 10,
  \]
  where
  \[
  S_k^{\alpha,\theta} :=\{ (x,t)\in \TTT^4\times I
  : |t-t_{k,\alpha}|< T_{\alpha}(N_{k,\alpha})^{-2},\ |x-x_{k,\alpha}|\leq R(N_{k,\alpha})^{-1}\}.
  \]
  and assume that $(\w_{k,1}, w_{k,2}, f_k)_k$ are 3 sequences of functions with
  properties $\|f_k\|_{X^1(I)}\leq 1$ for all $k$ large enough,
  then
  \[
  \limsup_{k\to\infty} \|\w_{k, 1}\, \w_{k, 2}\, f_k\|_{N(I)} = 0\]
\end{lem}

\begin{proof}
For $\ee > 0$ small.

If $N_{k,1}/N_{k,2} + N_{k,2}/N_{k,1} \leq \ee^{-1000}$ and
$k$ is large enough then $S_{k,1}\cap S_{k,2} = \emptyset$. (By the definition of
orthogonality of frames, $N_{k,1}/N_{k,2} + N_{k,2}/N_{k,1} \leq \ee^{-1000}$ implies
$N_{k,1}^2|t_{k,1}-t_{k,2}|\to \infty$ or $N_{k,1}|x_{k,1}-x_{k,2}|\to\infty$, so
$S_{k,1}\cap S_{k,2} = \emptyset$.)
In this case, $\w_{k, 1}\, \w_{k, 2}\, f_k\equiv 0$.

If $N_{k,1}/N_{k,2} \geq \ee^{-1000}/2$.
Denote that
\[
\w_{k,1}\w_{k,2} = \w_{k,1}\widetilde{\w}_{k,2}:=
\w_{k,1}\cdot (w_{k,2}\mathds{1}_{(t_{k,1}-T_1N_{k,1}^{-2}, t_{k,1}+T_1N_{k,1}^{-2})}(t)).
\]

\noindent\textit{\textbf{Claim $\dagger$} For $k$ large enough,
\begin{enumerate}
  \item $\|\widetilde{\w}_{k,2}\|_{X^1(I)}\lesssim_R 1$;
  \item $\|P_{>\ee^{-10}N_{k,2}} \widetilde{\w}_{k,2}\|_{X^1(I)}\lesssim_R \ee$;
  \item $\|\widetilde{\w}_{k,2}\|_{Z(I)}\lesssim_R \ee$;
  \item $\|\w_{k,1}\|_{X^1(I)}\lesssim_R 1$;
  \item $\|P_{\leq \ee^{10}} \w_{k,1}\|_{X^1(I)}\lesssim_R \ee$.
\end{enumerate}
}

By this \textit{Claim $\dagger$}, Proposition \ref{prop:nonlinear}, and $\ee^{10}N^1>>\ee^{-10}N_2$ we obtain that
\begin{align*}
  \|\w_{k,1}\w_{k,2}f_k\|_{N(I)} \leq& \|(P_{\leq \ee^{10}N_{k,1}} \w_{k,1})
  (\widetilde{\w}_{k,2}) f_k\|_{N(I)}
  +\|(P_{> \ee^{10}N_{k,1}} \w_{k,1})(P_{>\ee^{-10} N_{k,2}}\widetilde{\w}_{k,2})
  f_k\|_{N(I)}\\
  &+\|(P_{> \ee^{10}N_{k,1}} \w_{k,1})(P_{\leq\ee^{-10} N_{k,2}}\widetilde{\w}_{k,2})
  f_k\|_{N(I)}\\
  \lesssim_R& \ee.
\end{align*}

More detail about the \textit{Claim $\dagger$}:

\noindent (1) Consider $\widetilde{\w}_{k,2} w_{k,2}\mathds{1}_{(t_{k,1}-T_1N_{k,1}^{-2}, t_{k,1}+T_1N_{k,1}^{-2})}(t)$.
\begin{align*}
  \|\widetilde{\w}_{k,2}\|_{X^1(I)}&\lesssim \|\widetilde{\w}_{k,2}(0)\|_{H^1}
  +\left(\sum_{N} \|P_N(i\partial_t +\DD) \widetilde{\w}_{k,2}\|^2_{L^1_t([0,1],H^1)}
  \right)^{\frac{1}{2}}\\
  &\lesssim\left( \int_{|x-x_{k,2}|\leq RN_{k,2}^{-1}} |\langle \nabla \rangle \widetilde{\w}_{k,2}(0)
  |^2\,dx \right)^{\frac{1}{2}}\\
  &+ \left(\sum_{N} \left( \int_I dt
  \|P_N (\partial_t \widetilde{\w}_{k,2})\|_{H^1} + \|P_N \DD \widetilde{\w}_{k,2}\|_{H^1}
  \right)^2 \right)^{\frac{1}{2}}\\
  &\lesssim (R^2 N_{k,2}^4 R^4 N_{k,2}^{-4})^{\frac{1}{2}} + \int_I (\|\partial_t
  \widetilde{\w}_{\alpha, k}\|_{H^1}+\|\DD \widetilde{\w}_{k,2}\|_{H^1})\, dt\\
  &\lesssim 1.
\end{align*}

\noindent (2) Consider the high frequency part of $\widetilde{\w}_{k,2}$.
\begin{align*}
  &\|P_{>\ee^{-10}N_{k,2}} \widetilde{\w}_{k,2}\|_{X^1(I)}\\
  \lesssim&
  \|P_{>\ee^{-10}}\widetilde{\w}_{k,2}(0)\|_{H^1}
  +\left(\sum_{N>\ee^{-10}N_{k,2}} \|P_N(i\partial_t +\DD) \widetilde{\w}_{k,2}\|^2_{L^1_t([0,1],H^1)}
  \right)^{\frac{1}{2}}\\
  \leq&\left( \int_{|x-x_{k,2}|\leq RN_{k,2}^{-1}} |P_{>\ee^{-10} N_{k,2}}\langle \nabla \rangle \widetilde{\w}_{k,2}(0)
  |^2\,dx \right)^{\frac{1}{2}}+\int \|P_{>\ee^{-10}N_{k,2}} (i\partial_t+\DD) \widetilde{\w}_{k,2}\|_{H^1}\, dt\\
  \leq & \left( \int_{|x-x_{k,2}|\leq RN_{k,2}^{-1}}\left(\frac{\ee^{10}}{N_{k,2}}\right)^2
  |P_{>\ee^{-10} N_{k,2}}\langle \nabla \rangle^2 \widetilde{\w}_{k,2}(0)
  |^2\,dx \right)^{\frac{1}{2}}+\int_{|t-t_{k,2}|<N^{-2}_{k,2} R} \frac{\ee^{10}}{N_{k,2}}
  \| (i\partial_t+\DD) \widetilde{\w}_{k,2}\|_{H^2}\, dt\\
  \leq& \ee^{10} R^3 + N_{k,2}^{-2}R\frac{\ee^{10}}{N_{k,2}}(R^4 N_{k,2}^{-2} R^2 N_{k,2}^10)^{\frac{1}{2}}\\
  \lesssim& \ee^{10} R^4.
\end{align*}

\noindent (3): Consider the $Z$-norm of  $\widetilde{\w}_{k,2}$.
\begin{align*}
  \|\widetilde{\w}_{k,2}\|_{Z(I)}&\leq
  \left(\sum_N N^2 \|P_N \widetilde{\w}_{k,2}\|^4_{L^4(\TTT^4\times
   (t_{k,1}-RN_{k,1}^{-2}, t_{k,1}+RN_{k,1}^{-2}))}\right)^{1/4}\\
  &\leq \left(
  \sum_N \|\nabla^{\frac{1}{2}} P_N  \widetilde{\w}_{k,2}\|_{L^4_{t,x}}^4
 \right)^{1/4}\\
 &\lesssim \|\left(\sum_{N} |\nabla^{1/2} P_N \widetilde{\w}_{k,2}|^4\right)^{1/4}\|_{L^4}\\
 &\lesssim \|\left(\sum_{N} |\nabla^{1/2} P_N \widetilde{\w}_{k,2}|^2\right)^{1/2}\|_{L^4}\\
 &\lesssim \|\nabla^{1/2} \widetilde{\w}_{k,2}\|_{L^4(\TTT^4\times
  (t_{k,1}-RN_{k,1}^{-2}, t_{k,1}+RN_{k,1}^{-2}))}\\
 &\lesssim R^{\frac{9}{4}} \left(\frac{N_{k,2}}{N_{k,1}} \right)^{\frac{1}{2}}\\
 &\leq R^{\frac{9}{4}} \ee^{500}.
\end{align*}

\noindent (4): Similar with (1).

\noindent (5):
\begin{equation}\label{eq:**}\begin{split}
  &\|P_{\leq \ee^{10}N_{k,1}} \w_{k,1}\|_{X^1(I)}\\
  \leq&\|P_{\leq \ee^{10} N_{k,1}} \w_{k,1}(0)\|_{H^1}+\int \|P_{\leq \ee^{10}N_{k,1}}
  (i\partial_t +\DD) \w_{k,1}\|_{H^1}\,dt\\
  \lesssim& \ee^{10} N_{k,1}\left( \|P_{\leq \ee^{10} N_{k,1}} \w_{k,1}(0)\|_{L^2}
  +\int \|P_{\leq \ee^{10}N_{k,1}}
  (i\partial_t +\DD) \w_{k,1}\|_{L^2}\,dt
\right)\\
\lesssim& \ee^{10} R^4.
\end{split}
\end{equation}
\end{proof}

\begin{proof}[Proof of (\ref{eq:6.8})]
\begin{align*}
  &F(U_{prof,k}^J + e^{it\DD}R_k^J) - F(U_{prof,k}^J)\\
  =&\mathfrak{D}_{2,1}(U^J_{prof, k}, e^{it\DD}R_k^J) +
  \mathfrak{D}_{1,2}(U^J_{prof, k}, e^{it\DD}R_k^J)+ |e^{it\DD}R_k^J|^2 e^{it\DD} R_k^J
\end{align*}

First, by the nonlinear estimate (Proposition \ref{prop:nonlinear}), we have
\begin{align*}
  &\||e^{it\DD}R_k^J|^2 e^{it\DD} R_k^J\|_{N(I_k)}\\
  \lesssim & \|e^{it\DD}R_k^J\|^2_{Z'(I_k)}\|e^{it\DD}R_k^J\|_{X^1(I_k)}
\end{align*}

Since $\|e^{it\DD}R_k^J\|_{Z'(I_k)}\to 0$ as $J, k\to \infty$, and $\|e^{it\DD}
R_k^J\|_{X^1(I_k)}\lesssim 1$,
\[
\limsup_{J\to\infty}\limsup_{k\to\infty} \||e^{it\DD}R_k^J|^2 e^{it\DD} R_k^J\|_{N(I_k)}
= 0.
\]

Second, also by the nonlinear estimate Proposition \ref{prop:nonlinear} and
Proposition \ref{prop:ZinX},
\begin{align*}
  &\|\mathfrak{D}_{1,2}(U^J_{prof, k}, e^{it\DD}R_k^J)\|_{N(I_k)}\\
  \lesssim&
  \|U_{prof,k}^J\|_{X^1(I_k)}\|e^{it\DD}R_k^J\|_{X^1(I_k)}\|e^{it\DD}R_k^J\|_{Z'(I_k)}
  \to 0,
\end{align*}
as $k, J \to \infty.$

Third, consider
\[\|
\mathfrak{D}_{2,1}(U^J_{prof, k}, e^{it\DD}R_k^J)
\|_{N(I_k)},\]
assume $\ee>0$ is fixed, there exists $A = A(\ee)$ sufficiently large, such that
for all $J\geq A$ and $k\geq k_0(J)$
\[
\|U_{prof,k}^J - U_{prof,k}^A\|_{X^1(-1,1)}\leq \ee.
\]
Then
\begin{align*}
  &\|\mathfrak{D}_{2,1}(U^J_{prof, k}, e^{it\DD}R_k^J)\|_{N(I_k)}\\
  \leq& \|\mathfrak{D}_{2,1}(U^A_{prof, k}, e^{it\DD}R_k^J)\|_{N(I_k)}+
  \|\mathfrak{D}_{1,1,1}(U^A_{prof, k}, U^J_{prof, k}-U^A_{prof, k}, e^{itDD}R_k^J)\|_{N(I_k)}\\
  &+ \|\mathfrak{D}_{2,1}(U^J_{prof, k}-U^A_{prof, k}, e^{it\DD}R_k^J)\|_{N(I_k)}
  \to \|\mathfrak{D}_{2,1}(U^A_{prof, k}, e^{it\DD}R_k^J)\|_{N(I_k)} + \ee,
\end{align*}
as $k, J \to \infty.$

It remains to prove that
\[
\limsup_{J\to\infty}\limsup_{k\to\infty}
\|\mathfrak{D}_{2,1}(U^A_{prof, k}, e^{it\DD}R_k^J)\|_{N(I_k)}\lesssim \ee.
\]

By the definition of $U^A_{prof,k}$, it suffices to prove that for any $\alpha_1,
\alpha_2\in \{0, 1, \cdots, A\}$,

Fix $\theta = \ee A^{-2}/ 10$, apply the decomposition in Lemma \ref{lem:decomposition}
to all nonlinear profiles $U^\alpha_k$, $\alpha = 1, 2, \cdots, A$. We assume that
\[
T_{\theta,\alpha} = T_{\theta},\ \text{ and } R_{\theta,\alpha} = R_{\theta},
\]
for any $\alpha = 1, 2, \cdots, A$.

\begin{equation}
  \limsup_{J\to\infty}\limsup_{k\to\infty}
  \|\mathfrak{D}_{1,1,1}(U^{\alpha_1}_k, U^{\alpha_2}_k, e^{it\DD}R^J_k)\|_{N(I_k)}
  \lesssim \ee A^{-2}.
\end{equation}

\case{1}{$\alpha_1 = 0$ or $\alpha_2 = 0$.}
Without loss of generality, suppose $\alpha_2 = 0$.

Since $\|U^0_{k}\|_{X^1(-1,1)} = \|G\|_{X^1(-1,1)} \lesssim 1$, for any $k$ large
enough such that $\|G\|_{Z'(I_k)} \lesssim \ee A^{-2}$, and $\|G\|_{X^1(I_k)} \lesssim 1$.

By the nonlinear estimate Proposition \ref{prop:nonlinear} and
Proposition \ref{prop:ZinX},

\begin{align*}
  &\|\mathfrak{D}_{1,1,1}(G, U^{\alpha_2}_k, e^{it\DD}R^J_k)\|_{N(I_k)}\\
  \lesssim& \|G\|_{Z'(I_k)}\|U^{\alpha_2}_k\|_{Z'(I_k)}\|e^{it\DD}R_k^J\|_{X^1(I_k)}
  + \|G\|_{Z'(I_k)}\|U^{\alpha_2}_k\|_{X^1(I_k)}\|e^{it\DD}R_k^J\|_{Z'(I_k)}\\
  &+ \|G\|_{X^1(I_k)}\|U^{\alpha_2}_k\|_{Z'(I_k)}\|e^{it\DD}R_k^J\|_{Z'(I_k)}\\
  \lesssim& \ee A^{-2},
\end{align*}
when taking $k$, $J$ large enough.

\case{2}{$\alpha_1 \neq 0$, $\alpha_2\neq 0$ and $\alpha_1 = \alpha_2$.}

Taking $k$ large enough, we have $I_k \subset (-T_{\theta}^{-1},T_{\theta}^{-1})$
\[
\mathds{1}_{I_k}(t)
  U^\alpha_k = \w_k^{\alpha,\theta,-\infty}+\w_k^{\alpha,\theta,+\infty}
  +\w_k^{\alpha,\theta}+\rho^{\alpha, \theta}_k.\]

By  the nonlinear estimate Proposition \ref{prop:nonlinear}, (\ref{eq:decomposition})
and Lemma \ref{lem:7.2} (since $\|e^{it\DD}R_k^J\|_{X^1(I_k)}\lesssim 1$ uniformly
for both $k$ and $J$),
we obtain that
\begin{align*}
\|\mathfrak{D}_{1,1,1}(U_k^{\alpha_1}, U_k^{\alpha_2}, e^{it\DD}R_k^J)\|_{N(I_k)}
&\lesssim \frac{1}{2} A^{-2}\ee +
\|\mathfrak{D}_{1,1,1}(\w_k^{\alpha_1, \theta,+\infty}, \w_k^{\alpha_1, \theta,-\infty}, e^{it\DD}R_k^J)\|_{N(I_k)}\\
&\lesssim A^{-2}\ee,
\end{align*}
when $k$ large enough.

\case{3}{$\alpha_1 \neq 0$, $\alpha_2\neq 0$ and $\alpha_1 \neq \alpha_2$.}
Using Lemma \ref{lem:7.1}, and set $B$ sufficiently large and $k$ sufficiently
large, we obtain that,
\begin{equation}\label{eq:7.14}
  \begin{split}
  \|\mathfrak{D}_{2,1}(\w_k^{\alpha,\theta}, P_{>BN_{k,\alpha}} e^{it\DD}R_k^J)\|_{N(I_k)}
  &\lesssim (\frac{1}{B^{1/{200}}}+\frac{1}{N_{k,\alpha}^{1/{200}}})
  \|R_k^J\|_{H^1}\\
  &\lesssim \frac{\ee}{4} A^{-2}.
\end{split}
\end{equation}

We may also assume that $B$ is sufficiently large such that, for $k$ large enough,
by a similar estimate as (\ref{eq:**}), we obtain that
\begin{equation}\label{eq:7.15}
  \|P_{\leq B^{-1}N_{k,\alpha}} \w_k^{\alpha,\theta}\|_{X^1(I_k)}\leq \frac{\ee}{4}A^{-2}.
\end{equation}

Using the modified nonlinear estimate
(\ref{eq:nonlinearP}) of Lemma \ref{prop:nonlinear} and bounds (\ref{eq:7.14})
(\ref{eq:7.15}), it remains to prove that
\[
\limsup_{J\to\infty}\limsup_{k\to\infty} \|\mathfrak{D}_{2,1}
(P_{>B^{-1}N_{k,\alpha}}\w^{\alpha,\theta}_k, P_{\leq BN_{k,\alpha}}e^{it\DD}R_k^J)
\|_{N(I_k)} = 0.
\]
\end{proof}

\begin{proof}[Proof of (\ref{eq:6.9})]
\begin{equation*}
  F(U^J_{prof, k}) -\sum_{0\leq \alpha\leq J} F(U_k^\alpha)
  = \sum_{0\leq \alpha_1, \alpha_2, \alpha_3 \leq J\atop
  \alpha_1\neq\alpha_2\text{ or }\alpha_1\neq\alpha_3\text{ or }
  \alpha_2\neq\alpha_3
  }\mathfrak{D}_{1,1,1}(U_k^{\alpha_1}, U_k^{\alpha_2}, U_k^{\alpha_3})
\end{equation*}
By (\ref{6.6}), we choose $A(\theta)$ large enough, such that $\sum_{A\leq \alpha \leq J}
\|U_\alpha\|^2_{X^1(-1, 1)} \leq \theta$.

So we have
\begin{align*}
&\|\sum_{0\leq \alpha_1, \alpha_2, \alpha_3 \leq J\atop
\alpha_1\neq\alpha_2\text{ or }\alpha_1\neq\alpha_3\text{ or }
\alpha_2\neq\alpha_3
}\mathfrak{D}_{1,1,1}(U_k^{\alpha_1}, U_k^{\alpha_2}, U_k^{\alpha_3})\|_{N(I_k)}\\
\leq &\|\sum_{0\leq \alpha_1, \alpha_2, \alpha_3 \leq A\atop
\alpha_1\neq\alpha_2\text{ or }\alpha_1\neq\alpha_3\text{ or }
\alpha_2\neq\alpha_3
}\mathfrak{D}_{1,1,1}(U_k^{\alpha_1}, U_k^{\alpha_2}, U_k^{\alpha_3})\|_{N(I_k)}
+ \theta.
\end{align*}

Using Lemma \ref{lem:decomposition},
\begin{align*}
  &\|\sum_{0\leq \alpha_1, \alpha_2, \alpha_3 \leq A\atop
  \alpha_1\neq\alpha_2\text{ or }\alpha_1\neq\alpha_3\text{ or }
  \alpha_2\neq\alpha_3
  }\mathfrak{D}_{1,1,1}(U_k^{\alpha_1}, U_k^{\alpha_2}, U_k^{\alpha_3})\|_{N(I_k)}\\
  \leq& \|\sum_{F}
  \mathfrak{D}_{1,1,1}(W_k^{1}, W_k^{2}, W_k^{3})\|_{N(I_k)},
\end{align*}
where 
\begin{align*}
F:=& 
\{ (W_k^{1}, W_k^{2}, W_k^{3}) : W^i_k\in\{\w_k^{\alpha,\theta, +\infty},
 \w_k^{\alpha,\theta, -\infty}, \w_k^{\alpha,\theta}
,\rho_k^{\alpha, \theta}\}, \\ 
&\text{ for } 0\leq\alpha\leq A,\text{ and each }i,
\text{at least two different }\alpha \}
\end{align*}
and $\#F< A^3$

\noindent Consider the following several cases:

\case{1}{the terms containing one error component $\rho^{\alpha,\theta}_k$.}
By the nonlinear estimate (Proposition \ref{prop:nonlinear}),
\[
\|\mathfrak{D}_{1,1,1}(W_k^1, W_k^2, \rho_k^{\alpha, \theta})\|_{N(I_k)}
\leq \|\rho_k^{\alpha, \theta}\|_{X^1(I_k)} \|W_k^1\|_{X^1(I_k)}\|W_k^2\|_{X^1(I_k)}
\lesssim \theta,\]
for $k$ large enough.

\case{2}{the terms containing two scattering components $\w_k^{\alpha,\theta,\pm\infty}$
and $\w_k^{\beta,\theta,\pm\infty}$(maybe $\alpha =\beta$ or not).}
\begin{align*}
&\|\mathfrak{D}_{1,1,1}(\w_k^{\alpha,\theta,\pm\infty}, \w_k^{\beta,\theta,\pm\infty}, W_k^3)\|_{N(I_k)}\\
\leq&\|W^3_k\|_{X^1(I_k)}(\|\w_k^{\alpha,\theta,\pm\infty}\|_{X^1(I_k)}+\|\w_k^{\beta,\theta,\pm\infty}\|_{X^1(I_k)})
(\|\w_k^{\alpha,\theta,\pm\infty}\|_{Z'(I_k)}+\|\w_k^{\beta,\theta,\pm\infty}\|_{Z'(I_k)})\\
\lesssim&  \theta
,\end{align*}
for $k$ large enough.

\case{3}{the terms containing two different cores $\w^{\alpha,\theta}_k$ and
$\w^{\beta,\theta}_k$ with $\alpha\neq\beta$.}
By Lemma \ref{lem:7.2}, for $k$ large enough, we obtain that
\[
\|\mathfrak{D}_{1,1,1}(\w^{\alpha,\theta}_k, \w^{\beta,\theta}_k, W_k^3)\|_{N(I_k)}
\lesssim \theta.
\]
\case{4}{the others: $\mathfrak{D}_{2,1}(\w_k^{\alpha,\theta},
\w_k^{\beta, \theta,\pm\infty})$ with $\alpha\neq \beta$.}

\case{4.1}{$\limsup_{k\to \infty}\frac{N_{k,\beta}}{N_{k,\alpha}} = +\infty$.}
By Lemma \ref{lem:7.1}, and choosing $B$ and $k$ large enough,
\begin{equation}\label{eq:***}
  \|\mathfrak{D}_{2,1}(\w^{\alpha, \theta},
  P_{>BN_{k,\alpha}}\w^{\beta,\theta,\pm\infty})\|_{N(I_k)}\lesssim
  (B^{-1/{200}}+ N_{k, \alpha}^{-1/{200}})\lesssim \theta.
\end{equation}
And for the other part,
\begin{align*}
\|P_{\leq BN_{k,\alpha}} \w^{\beta, \theta, \pm\infty}_k\|_{X^1(I_k)}
&= \|P_{\leq BN_{k,\beta}\frac{N_{k,\alpha}}{N_{k,\beta}}}
\w^{\beta, \theta, \pm\infty}_k\|_{X^1(I_k)}\\
&= \|P_{\leq BN_{k,\beta}\frac{N_{k,\alpha}}{N_{k,\beta}}}
\pi_{x^\beta_k}T_{N_{k,\beta}}(\phi^{\beta,\theta,\pm\infty})\|_{H^1(\TTT^4)}\\
& = \|P_{\leq B\frac{N_{k,\alpha}}{N_{k,\beta}}}
\phi^{\beta,\theta,\pm\infty}\|_{\dot{H}^1(\RRR^4)}\to 0, \text{ as } k \to\infty.
\end{align*}

So for $k$ large enough, we obtain that
\begin{equation*}
\|\mathfrak{D}_{2,1}(\w_k^{\alpha, \theta},
P_{\leq BN_{k,\alpha}}\w_k^{\beta,\theta,\pm\infty})\|_{N(I_k)}
\lesssim \|P_{\leq BN_{k,\alpha}} \w_k^{\beta, \theta, \pm\infty}\|_{X^1(I_k)}
\|\w_k^{\alpha, \theta}\|^2_{X^1(I_k)} \lesssim \theta.
\end{equation*}
\case{4.2}{$\limsup_{k\to \infty}\frac{N_{k,\alpha}}{N_{k,\beta}} = +\infty$.}

We assume that $B$ is sufficiently large such that for $k$ large, by a similar
estimate as (\ref{eq:***}), we obtain that
\begin{equation*}
  \|\mathfrak{D}_{2,1}(\w_k^{\alpha, \theta},
  P_{>BN_{k,\beta}}\w_k^{\beta,\theta,\pm\infty})\|_{N(I_k)}\lesssim
  (B^{-1/{200}}+ N_{k, \beta}^{-1/{200}})\lesssim \theta.
\end{equation*}

And by the similar estimate as (\ref{eq:**}), for $k$ large enough,
 we obtain that
\begin{equation*}
  \|P_{\leq N_{k,\beta}} \w_k^{\alpha,\theta}\|_{X^1(I_k)}=
  \|P_{\leq N_{k,\alpha}\frac{N_{k,\beta}}{N_{k,\alpha}}} \w_k^{\alpha,\theta}\|_{X^1(I_k)}
  \lesssim \theta.
\end{equation*}
and $\|P_{> N_{k,\beta}} \w_k^{\alpha,\theta}\|_{X^1(I_k)} \lesssim 1.$

Consider the remaining part, by the nonlinear estimate (\ref{eq:nonlinear}) and
(\ref{eq:nonlinearP}),
\begin{align*}
  &\|\mathfrak{D}_{2,1}(\w_k^{\alpha, \theta},
  P_{\leq BN_{k,\beta}}\w_k^{\beta,\theta,\pm\infty})\|_{N(I_k)}\\
  \lesssim& \|\mathfrak{D}_{2,1}(P_{\leq N_{k,\beta}}\w_k^{\alpha, \theta},
  P_{\leq BN_{k,\beta}}\w_k^{\beta,\theta,\pm\infty})\|_{N(I_k)}
  +
  \|\mathfrak{D}_{2,1}(P_{> N_{k,\beta}}\w_k^{\alpha, \theta},
  P_{\leq BN_{k,\beta}}\w_k^{\beta,\theta,\pm\infty})\|_{N(I_k)} \\
  &+
  \|\mathfrak{D}_{1,1,1}(P_{> N_{k,\beta}}\w_k^{\alpha, \theta},
  P_{\leq N_{k,\beta}}\w_k^{\alpha, \theta},
  P_{\leq BN_{k,\beta}}\w_k^{\beta,\theta,\pm\infty})\|_{N(I_k)}\\
  \lesssim&\|P_{\leq N_{k,\beta}}\w_k^{\alpha, \theta}\|_{X^1(I_k)}+
  \|\mathfrak{D}_{2,1}(P_{> N_{k,\beta}}\w_k^{\alpha, \theta},
  P_{\leq BN_{k,\beta}}\w_k^{\beta,\theta,\pm\infty})\|_{N(I_k)}\\
  \lesssim &\theta + \|P_{\leq BN_{k,\beta}}\w_k^{\beta,\theta,\pm\infty}\|_{Z'(I_k)}\\
  \lesssim& \theta,
\end{align*}
where $\limsup_{J\to\infty}\limsup_{k\to\infty}
\|P_{\leq BN_{k,\beta}}\w_k^{\beta,\theta,\pm\infty}\|_{Z'(I_k)} = 0$(by a similar
estimate with (\ref{eq:decomposition}) from extinction lemma.)

\case{4.3}{$N_{k,\alpha} \approx N_{k,\beta}${ and }
$N_{k,\alpha}|x_k^\alpha - x_k^\beta|\to \infty$ as $k\to \infty$.}

From Proposition \ref{prop:equivalenceFrames}, we can use an equivalent frame of $\mathcal{O}^\alpha$ to adjust $N_{k,\alpha}$ and $t_k^\alpha$ such that
$N_{k,\alpha} = N_{k,\beta}$ and $t_k^\alpha = t_k^\beta$.

By the definition of $\w_k^{\alpha,\theta}$ and $\w_k^{\beta,\theta,\pm\infty}$,
for $k$ large enough, we obtain that
$\w_k^{\beta,\theta} \w_k^{\alpha,\theta,\pm\infty} \equiv 0.$

\case{4.4}{$N_{k,\alpha} \approx N_{k,\beta}${ and }
$N^{2}_{k,\alpha}|t_k^\alpha - t_k^\beta|\to \infty$ as $k\to \infty$.}

By Proposition \ref{prop:equivalenceFrames}, we can adjust $N_{k,\alpha}$ such that $N_{k,\alpha} = N_{k,\beta} := N_k$.

By the definition of $\w_k^{\alpha,\theta}$ and $\w_k^{\beta,\theta,\pm\infty}$,
taking $k$ large enough and $N_k^2|t_k^\alpha - t_k^\beta| > T_\theta$, we obtain that
\[
\w_k^{\alpha,\theta}\w_k^{\beta,\theta,\pm\infty} =
\mathds{1}_{[t_\alpha - \frac{T_\theta}{N_k^2},t_\alpha + \frac{T_\theta}{N_k^2}]}
\w_k^{\alpha,\theta}\w_k^{\beta,\theta,\pm\infty},
\]
and also $\w_k^{\alpha, \theta,\pm\infty} = P_{\leq R_{\theta}N_{k}}
\w_k^{\alpha, \theta,\pm\infty}$.

By (\ref{4.12}) and (\ref{4.13}), for any $T\leq N_k$, we obtain that
\begin{equation}\label{eq:L2}
  \begin{split}
  \|\w_k^{\beta,\theta,\pm\infty}\|_{L^2(\TTT^4)}
  &= \|P_{\leq R_{\theta}N_k}\w_k^{\beta,\theta,\pm\infty}\|_{L^2(\TTT^4)}\\
  &\lesssim (1 + {R_\theta})^{-10}\frac{1}{N_{k}},
\end{split}
\end{equation}
and
\begin{equation}\label{eq:Linfty}
  \sup_{|t-t_{k}^\beta|\in [TN_k^{-2}, T^{-1}]}
  \|w_k^{\beta,\theta,\pm\infty}\|_{L^\infty{(\TTT^4)}} \lesssim T^{-2}
  R_\theta^4 N_k.
\end{equation}

Interpolate (\ref{eq:L2}) and (\ref{eq:Linfty}), we can obtain that
\begin{equation}\label{eq:Lp}
  \sup_{|t-t_{k}^\beta|\in [TN_k^{-2}, T^{-1}]}
  \|w_k^{\beta,\theta,\pm\infty}\|_{L^p(\TTT^4)}
  \lesssim_{R_\theta} T^{\frac{4}{p}-2} N_k^{1-\frac{4}{p}}.
\end{equation}

By choosing $T_k = N_k |t_k^\alpha - t_k^\beta|^{\frac{1}{2}}\to \infty$ as $k\to \infty$
 and using (\ref{eq:Lp}),
we obtain that
\begin{equation}\label{eq:Linfty2}
  \sup_{t\in [t_k^\alpha - \frac{T_\theta}{N_k^2},t_k^\alpha + \frac{T_\theta}{N_k^2}]}
  \|\w_k^{\beta, \theta,\pm\infty}\|_{L^\infty(\TTT^4)}
  \lesssim_{R_\theta} T_k^{-2} N_k,
\end{equation}
and
\begin{equation}\label{eq:L4}
  \sup_{t\in [t_k^\alpha - \frac{T_\theta}{N_k^2},t_k^\alpha + \frac{T_\theta}{N_k^2}]}
  \|\langle \nabla \rangle \w_k^{\beta, \theta,\pm\infty}\|_{L^4(\TTT^4)}
  \lesssim_{R_\theta}  T_k^{-1}N_k.
\end{equation}

So by using of Leibniz rule, (\ref{eq:decomposition}) (\ref{eq:L4}) and (\ref{eq:Linfty2}), we obtain that
\begin{align*}
&\|\mathfrak{D}_{2,1}(\w^{\alpha,\theta}_k, \w^{\beta,\theta,\pm\infty}_k)
\|_{N( [t^\alpha_k - \frac{T_\theta}{N_k^2},t^\alpha_k + \frac{T_\theta}{N_k^2}])}\\
\lesssim& \|\mathfrak{D}_{2,1}(\w^{\alpha,\theta}_k, \w^{\beta,\theta,\pm\infty}_k)
\|_{L^1([t_k^\alpha - \frac{T_\theta}{N_k^2},t_k^\alpha + \frac{T_\theta}{N_k^2}], H^1(\TTT^4))}\\
\lesssim&\int_{t_k^\alpha - \frac{T_\theta}{N_k^2}}^{t_k^\alpha + \frac{T_\theta}{N_k^2}}
\left(\|\mathfrak{D}_2(\langle \nabla \rangle\w_k^{\alpha,\theta})\|_{L^2(\TTT^4)}
\|\w_k^{\beta,\theta,\pm\infty}\|_{L^\infty(\TTT^4)}
+ \|\mathfrak{D}_2(\w_k^{\alpha,\theta})\|_{L^4(\TTT^4)}
\|\langle \nabla \rangle\w_k^{\beta,\theta,\pm\infty}\|_{L^4(\TTT^4)}
\right)\, dt\\
\lesssim& \int_{t_k^\alpha - \frac{T_\theta}{N_k^2}}^{t_k^\alpha + \frac{T_\theta}{N_k^2}}
\left(
N_k^2 T_k^{-2} R_\theta^8+
N_k^2 T_k^{-1} R_\theta^3
\right)\, dt
\\
\lesssim&
T_k^{-1} T_\theta R_\theta^8 \to 0 \text{ as } k\to\infty.
\end{align*}

\end{proof}

\bibliographystyle{amsplain}
\bibliography{Yue_ref.bbl}{}

\end{document}